\documentclass[10pt]{amsart}

\usepackage[latin1]{inputenc}
\usepackage[T1]{fontenc}

\usepackage{amsmath}
\usepackage{hyperref} 
\usepackage{amsfonts}
\usepackage{amssymb}
\usepackage{amsthm}
\usepackage[arrow, matrix, curve]{xy} 
\usepackage{hyperref}
\usepackage{comment}
\usepackage{amssymb,amscd,url,tikz,stmaryrd,mathtools,fixltx2e}
\usepackage{amssymb,amscd,url}
\usepackage{tikz, pgfplots}
\usepackage{tikz-cd} 
\usepackage{color}

\newcommand{\Z}{{\mathbb Z}}
\renewcommand{\S}{{\mathbb S}}
\newcommand{\R}{{\mathbb R}}
\newcommand{\C}{{\mathbb C}}
\newcommand{\D}{{\mathbb D}}
\newcommand{\N}{{\mathbb N}}
\renewcommand{\H}{{\mathbb H}}

\newcommand{\uc}{\overline{u}}
\newcommand{\vc}{\overline{v}}

\def\Res{{\,\rm Res}}
\def\Re{{\rm Re}}
\def\Im{{\rm Im}}
\def\Id{{\rm Id}}
\def\ii{{\rm i}}

\def\tr{{\rm trace}}

\def\SL{{\rm SL}}
\def\SU{{\rm SU}}

\def\cvx{\cv{x}}
\renewcommand{\matrix}[1]{\left(\begin{array}{cc} #1\end{array}\right)}

\newcommand{\wt}[1]{\widetilde{#1}}
\newcommand{\wh}[1]{\widehat{#1}}
\newcommand{\cal}[1]{{\mathcal #1}}
\newcommand{\low}[1]{{#1}_{\mbox{\scriptsize lower}}}
\newcommand{\lowind}[2]{{#1}_{\mbox{\scriptsize lower},#2}}

\theoremstyle{plain}
\newtheorem{theorem}{Theorem}
\newtheorem*{theorem*}{Theorem}
\newtheorem{lemma}{Lemma}
\newtheorem{proposition}[lemma]{Proposition}
\newtheorem{remark}[lemma]{Remark}
\newtheorem{corollary}[lemma]{Corollary}
\newtheorem{claim}[lemma]{Claim}

\newtheorem{ansatz}[lemma]{Ansatz}
\newtheorem{definition}[lemma]{Definition}

\newtheorem{example}[lemma]{Example}

\newtheorem{convention}[lemma]{Convention}
\DeclareFontFamily{U}{mathx}{\hyphenchar\font45}
\DeclareFontShape{U}{mathx}{m}{n}{
      <5> <6> <7> <8> <9> <10>
      <10.95> <12> <14.4> <17.28> <20.74> <24.88>
      mathx10
      }{}
\DeclareSymbolFont{mathx}{U}{mathx}{m}{n}
\DeclareFontSubstitution{U}{mathx}{m}{n}
\DeclareMathAccent{\widecheck}{0}{mathx}{"71}
\DeclareMathAccent{\widetilde}{0}{mathx}{"72}
\DeclareMathAccent{\widebar}{0}{mathx}{"73}
\DeclareMathAccent{\widevec}{0}{mathx}{"74}
\DeclareMathAccent{\widehat}{0}{mathx}{"70}
\DeclareMathAccent{\widefrown}{0}{mathx}{"75}
\DeclareMathAccent{\chinesehat}{0}{mathx}{"69}
\newenvironment{smatrix}{\bigl(\begin{smallmatrix}}{\end{smallmatrix}\bigr)}

\def\sign{\varepsilon}

\def\Res{{\,\rm Res}}
\def\Re{{\rm Re}}
\def\Im{{\rm Im}}
\def\Id{{\rm Id}}
\def\ii{{\rm i}}

\def\tr{{\rm tr}}

\def\End{{\rm End}}
\def\SL{{\rm SL}}

\def\SU{{\rm SU}}

\newcommand{\Li}{\operatorname{Li}}
\renewcommand{\and}{\quad\text{and}\quad}
\newcommand{\with}{\quad\text{with}\quad}
\newcommand{\dbar}{{\bar\partial}}
\newcommand{\del}{{\partial}}
\renewcommand{\matrix}[1]{\left(\begin{array}{cc} #1\end{array}\right)}

\newcommand{\cv}[1]{\underline{#1}}
\newcommand{\wtcv}[1]{\widetilde{\underline{#1}}}

\title[Loop group methods for the non-abelian Hodge correspondence]{Loop group methods for the non-abelian Hodge correspondence on a 4-punctured sphere}

 \author{Lynn Heller}
 \author{Sebastian Heller}
 \author{Martin Traizet}

\address{Beijing Institute of Mathematical Sciences and Applications\\
Beijing, China} 
\email{lynn@bimsa.cn}

\address{Beijing Institute of Mathematical Sciences and Applications\\
Beijing, China} 
\email{sheller@bimsa.cn}

\address{Institut Denis Poisson, CNRS UMR 7350 \\
Facult\'e des Sciences et Techniques \\
Universit\'e de Tours\\France }
\email{martin.traizet@univ-tours.fr }

\begin{document}

\begin{abstract}
The non-abelian Hodge correspondence   
is a real analytic map between the moduli space of stable Higgs bundles and the deRham moduli space of irreducible flat connections mediated by solutions to the self-duality equations. In this paper we construct self-duality solutions 
for strongly parabolic $\mathfrak{sl}(2,\C)$ Higgs fields on a $4$-punctured sphere with parabolic weights
 $t \sim 0$ using complex analytic methods. 
We identify  the rescaled limit hyper-K\"ahler
moduli space $\mathcal M_t$ at $t=0$ to be the completion of the nilpotent orbit in $\mathfrak{sl}(2, \C)$ modulo a $\Z_2\times\Z_2$ action, equipped with the Eguchi-Hanson metric. 
Our methods and computations are based on the twistor approach to the self-duality equations
using Deligne and Simpson's  $\lambda$-connections interpretation.
By construction we can compute the Taylor expansions of the holomorphic symplectic form $\varpi_t$ on $\mathcal M_t$ at $t=0$ which turn out to have closed form expressions in terms of multiple polylogarithms (MPLs). The geometric properties of $\mathcal M_t$ lead to some identities of certain MPLs which we believe deserve further investigations.
\end{abstract}
\thanks{\\We thank Philip Boalch, Andy Neitzke and Hartmut Wei{\ss} for helpful comments. Moreover, we would like to thank the anonymous referee for a very thorough report which greatly helped to improve the presentation.\\
LH and SH have been supported by the  {\em Deutsche Forschungsgemeinschaft} within the priority program {\em Geometry at Infinity} and by the Beijing Natural Science Foundation IS23002 (LH) and IS23003 (SH).\\MT is supported by the French ANR project Min-Max (ANR-19-CE40-0014).\\
On behalf of all authors, the corresponding author states that there is no conflict of interest and a supplementary Mathematica file can be found under the web address \href{https://www.idpoisson.fr/traizet/}{https://www.idpoisson.fr/traizet/}}
\setcounter{tocdepth}{1}
\maketitle
\tableofcontents

\section{Introduction}

Hitchin's  self-duality equations \cite{Hi1} on a degree zero and rank two hermitian vector bundle $V \rightarrow \Sigma$ over a compact Riemann surface $\Sigma$ are equations on a  pair $(\nabla, \Psi)$ consisting of
a special unitary connection $\nabla$ and an (trace free) endomorphism valued $(1,0)$-form $\Psi$ satisfying
\begin{equation}\label{SD}\dbar^\nabla \Psi = 0 \quad \text{ and }\quad F^\nabla + [\Psi, \Psi^*] = 0.\end{equation}
These equations are equivalent to the flatness of the whole associated family of connections

\begin{equation}\label{nabla-lambda}\nabla^\lambda  = \nabla + \lambda^{-1} \Psi + \lambda \Psi^*\end{equation}
parametrized by $\lambda \in \C^*$ -- the spectral parameter. Solutions to \eqref{SD} give rise to equivariant harmonic maps from the Riemann surface into the hyperbolic 3-space  SL$(2, \C)/ \SU(2)$ by considering the Higgs field $\Psi$ as the $(1,0)$-part of the differential of the harmonic map \cite{Do}. An application of a  classical result by Eells-Sampson \cite{ES} then shows the existence of a unique solution in each homotopy class of equivariant maps, i.e., when prescribing the (totally reducible) monodromy of the harmonic map \cite{Do}. 

As first recognized by Hitchin \cite{Hi1}, the moduli space of solutions to the self-duality equations (modulo gauge transformations) $\mathcal M_{SD}$ has two very different incarnations as complex analytic spaces. Firstly as the moduli space of Higgs bundles $\mathcal M_{Higgs},$ i.e., as the moduli space of polystable pairs $(\dbar_V, \Psi)$
satisfying $\dbar_V\Psi=0$, and secondly as the moduli space of totally reducible flat connections $\mathcal M_{dR}$. 
The non-abelian Hodge correspondence is the map between $\mathcal M_{Higgs}$ and 
$\mathcal M_{dR}$ obtained through solutions to Hitchin's equations. Both complex analytic spaces $\mathcal M_{Higgs}$ and $\mathcal M_{dR}$ induce anti-commuting complex structures on $\mathcal M_{SD}$ turning it into a hyper-K\"ahler manifold
when equipped with its natural $L^2$-metric. Since the  mapping is mediated by the harmonic map this correspondence is not explicit and it is not possible to see every facet of the geometry within one framework only. 

Every hyper-K\"ahler manifold can be described using complex analytic data (subject to additional reality conditions) via twistor theory \cite{HKLR}. For $\mathcal M_{SD}$, the twistor space has been identified with the so-called Deligne-Hitchin moduli space $\mathcal M_{DH}$
 introduced by Deligne and Simpson \cite{Si}. The space  $\mathcal M_{DH}$ is obtained by gluing the moduli space of $\lambda$-connections on the Riemann surface with the moduli space of $\lambda$-connections on the complex conjugate Riemann surface, and the
associated family of flat connections \eqref{nabla-lambda} naturally extends to a special real holomorphic section -- a twistor line. The main subject of the paper is to construct  twistor lines and the hyper-K\"ahler structure of $\mathcal M_{SD}$
{\em `entirely complex analytically, so we bypass the nonlinear elliptic theory necessary to define the harmonic metrics' \footnote{C. Simpson in \cite[Section 4]{Si}}.}

The self-duality equations \eqref{SD} generalize to punctured Riemann surfaces by imposing
 first order poles of the Higgs fields and a growth condition of the harmonic metric  determined by their parabolic weights, called tameness, see Simpson \cite{Si1}. 
 In the following, we 
 restrict to {\em strongly} parabolic case, where the Higgs fields have nilpotent residues adapted to the parabolic filtration.
Then the associated family of flat connections has (up to conjugation)
the same local monodromy at each puncture 
 \cite[table on page 720]{Si1} for all spectral parameters $\lambda\in\C^*$.
The moduli space of solutions has again a hyper-K\"ahler structure which was first studied by Konno \cite{Konno}. The complex structure $I$ denotes hereby the one of the moduli space of (polystable) strongly parabolic Higgs fields, and the complex structure $J$ denotes that of the moduli space of logarithmic connections with prescribed conjugacy class of their local monodromy.

In this paper, we study the simplest non-trivial case, where the underlying Riemann surface is a 4-punctured sphere and restrict to Fuchsian systems, i.e., logarithmic connections  on the trivial holomorphic (rank 2) bundle. On the Higgs side,
 we assume that the parabolic Higgs bundles are strongly parabolic
 with the same parabolic weights $\pm t$ with $t\in(0,\tfrac{1}{4})$ at each of the four singular points. The moduli space of self-duality equations is then denoted by $\mathcal M_t.$
For $t\rightarrow0$ the space of  polystable parabolic Higgs pairs bounded by $tC$ (for some fixed $C>0$) degenerates to the point given by the trivial parabolic Higgs pair.  When rescaling by $t$ its blow-up limit $\tfrac{1}{t}\mathcal M_t$ at $t=0$ is given by  the moduli space of parabolic Higgs fields $\Phi$
on the trivial holomorphic bundle. Using an implicit theorem argument we construct for $t \sim 0$ twistor lines 
in the (parabolic) Deligne-Hitchin moduli space near this singular limit using complex analytic methods, i.e., we construct solutions for which the Higgs field $\Psi \sim t\Phi$ is small enough. The length of $\Phi$ is measured using its $L^2$-norm, 
or equivalently by the
 energy $\mathcal E(\nabla, \Phi)$ of the equivariant harmonic map corresponding to the Higgs pair.

\begin{theorem}\label{construction}
For $t\sim 0$ fixed let $C>0$ and consider Higgs fields with $\mathcal E(d, \Phi) <C.$ Then, there exists $\varepsilon >0,$ depending only on the constant $C$, such that 
twistor lines $s_t$ can be constructed with
$s_t(\lambda = 0) = t \Phi$
using complex analytic methods only. 
In particular, the method allows to explicitly compute the Taylor expansions of the twistor lines in $t$.
\end{theorem}
In fact, we find two different families of real holomorphic sections $s^\pm_t$ of the Deligne-Hitchin moduli space for a given initial Higgs field $\Phi.$ Only one of them $s^+_t$ turn out to be twistor lines, while the other family $s^-_t$ corresponds to higher solutions as introduced in \cite{HH}. By construction Theorem \ref{construction} also allows us to study the Taylor expansions of the geometric structures on $\mathcal M_t$, such as its hyper-K\"ahler structure, explicitly.  
In particular,  the (rescaled) limit metric for $t\rightarrow 0$ can be identified to be the Eguchi-Hanson metric $g_{EH}$ modulo a $\mathbb Z_2\times\mathbb Z_2$-action. For rational $t = \tfrac{l}{2k}$, we can take a suitable $k$-fold cover $\Sigma_k$ of the four punctured sphere (depending on the value of $t$), such that the (pull-back of) logarithmic connections on the sphere are gauge equivalent to symmetric (smooth) connections on the compact Riemann surface $\Sigma_k$ of genus $k-1 = O(\tfrac{1}{t})$. The moduli space of equivariant solutions to the self-duality equations $\mathcal M^l_{SD}$ on $\Sigma_k$ then converge
(on every compact subset) to a corresponding compact subset of the Eguchi-Hanson space modulo the $\mathbb Z_2\times\mathbb Z_2$-action as $k\rightarrow \infty$. More precisely, 

\begin{theorem}\label{thm:compactcon}
Let $C>0$ and $l\in\N^{>0}$ be fixed.  Consider the compact subspaces 
\[\mathcal C^l_k:=\{[\nabla,\Phi]\in \mathcal M^l_{SD}(\Sigma_k)\mid \mathcal E(\nabla,\Phi)\leq C\}\]
 of $\mathcal M^l_{SD}$ with the induced hyper-K\"ahler metric $g_k$.
Then for $k \rightarrow \infty$, we obtain smooth convergence
\[(\mathcal C^l_k,g_k)\longrightarrow (\{v\in T^*\C P^1\mid \mathcal E_{EH}(v)\leq C\}, 32\pi\,l\, g_{EH})/(\Z_2\times\Z_2).\]
\end{theorem}
The main advantage of the approach is that the higher order terms of the metric expansion are also accessible. We have
 \begin{theorem}\label{mainT}
The hyper-K\"ahler metric of the moduli space $\mathcal M_t$ of strongly parabolic Higgs bundles on the 4-punctured sphere is real analytic in its weight $t$ and there exist an explicit algorithm to computing its Taylor expansion at the trivial parabolic Higgs pair in $t=0$. More precisely, when computing the Taylor expansion of the twistor lines in $t$ at the trivial connection, the $n$-th order coefficients are polynomials in $\lambda$ of degree $n+1$ and can be expressed explicitly in terms of multiple polylogarithms of depth and weight at most $n+1$.
\end{theorem}
Also the second family of solutions yields a hyper-K\"ahler metric, the $n$-th order coefficients of the two metric expansions differ only by $(-1)^n$. 

It would be interesting to compare our results to those of \cite{FMSW},
where the moduli space is studied for fixed weights and large Higgs fields in the regular locus, as well as in the context of current developments on hyper-K\"ahler 4-manifolds (see e.g. \cite{CC1,CC2,CC3}), and in view of the different ways of obtaining hyper-K\"ahler structures on moduli spaces (see e.g. \cite{Boa} and the references therein).\\

The paper is organized as follows. In Section \ref{pre} we introduce Higgs bundles, the associated moduli spaces as well as its hyper-K\"ahler structure for compact Riemann surfaces. Their complex geometric properties are encoded in the Deligne-Hitchin moduli space, see \cite{Si}, and we discuss twistor lines as special real holomorphic sections induced by self-duality solutions. Furthermore, we identify the twisted holomorphic symplectic structure on the  Deligne-Hitchin moduli space with (a version of) Goldman's symplectic structure.
We extend this discussion to the case of punctured surfaces and strongly parabolic Higgs pairs in Section \ref{sec:paradh}, and
explain the relationship between the compact and punctured case for rational weights.
Thereafter, we give the ansatz in terms of Fuchsian systems for constructing real holomorphic sections through an implicit function theorem argument and loop group methods in Section \ref{Ansatz} and prove Theorem \ref{construction} in Section \ref{IFT}.  To find appropriate coordinates, we first consider the 4-fold covering of the Higgs bundles moduli space $\mathcal M_{Higgs}$ and we rescale it by  by the factor $\frac{1}{t}$. The limit space at $t=0$ is then given
by the blow-up of $\C^2/\Z_2$ (with coordinates $(u,v)$) at the origin. To perform the implicit theorem argument, we thus first consider the regular case in Theorem \ref{thm:IFT}  and then when $(u, v) \rightarrow (0,0)$ in Theorem \ref{blowuplimit}.
 It turns out that in the appropriate $(u, v) \rightarrow (0,0)$ limit the Higgs field vanishes, and the corresponding sections $s^+_0$
are flat unitary connections, which are twistor lines. Due to the fact that twistor lines form a connected component of the space of real holomorphic sections, all constructed sections must be twistor lines as well. Interestingly, we find a Lax pair type equation describing the deformation given by the implicit function theorem.  In  Section \ref{NAHCt=0}, we  compute the limit non-abelian Hodge correspondence at $t=0$, and identify the
  rescaled limit hyper-K\"ahler metric to be Eguchi-Hanson.  One advantage of our construction is that we can compute Taylor expansion of all geometric quantities. As an example, we  compute first order derivatives of the parameters in Section \ref{firstderivative} and put them together to obtain the first order expansion of the non-abelian Hodge correspondence, the twisted holomorphic symplectic structure and the hyper-K\"ahler metric.  In the last Section \ref{hod} we analyse the structure of the higher order derivatives of the twistor
sections with respect to $t$. As in \cite{HHT3} for minimal surfaces in the 3-sphere, the $n$-th order derivatives of the parameters are polynomial in $\lambda$ of order $(n+1)$. Since the twisted holomorphic symplectic form is a Laurent polynomial of degree one in $\lambda$, evaluating it on
meromorphic 1-forms which represent tangent vectors to $\mathcal M(t)$ yields infinitely many cancelations for higher order terms in $\lambda$ leading to identities for some iterated integrals, called {\em $\Omega$-values},  which can be expressed
in
terms of multiple polylogarithms.  In Section \ref{nOid} we give some identities of depth three $\Omega$-values obtained from this idea which are non-trivial in the sense that they cannot be derived from the shuffle and stuffle relations alone.

\section{Preliminaries}\label{pre}

In this preliminary section, we recall basic results about the moduli spaces of solutions to Hitchin's self-duality equations with a focus on the twistorial
description of its hyper-K\"ahler structure. Unless stated otherwise, the underlying Riemann surface $\Sigma$ is assumed  to be compact throughout the section.

\subsection{The hyper-K\"ahler structures and their twistor spaces}
The moduli space of irreducible solutions of Hitchin's self-duality equations $\mathcal M_{SD}$ has through the non-abelian Hodge correspondence three complex structures $I,J,K= I J$ which are K\"ahler with respect to the same Riemannian metric $g$. In fact there exist a whole $\C P^1$ worth of complex structures defined by  

$$I_\lambda = \frac{1-|\lambda|^2}{1+|\lambda|^2}I+\frac{\lambda + \bar \lambda}{1+|\lambda|^2}J + \frac{i(\lambda - \bar \lambda)}{1+|\lambda|^2} K$$

for $\lambda \in \C \subset \C P^1.$
Within this family we find the complex structure $I$ at $\lambda = 0$, and the complex structure $J$ at $\lambda = 1$.
Let
$$\omega_I = - g(., I .) \quad \omega_J = - g(., J.) \and \omega_K = - g(., K .)$$
be the associated K\"ahler forms so that $h= g+ i\omega_L$ is the corresponding hermitian metric for $L\in\{I,J,K\}$. The twistor space of a hyper-K\"ahler manifold $\mathcal M$, introduced in \cite{HKLR}, is the smooth manifold$$\mathcal P := \mathcal M \times \C P^1$$
equipped with the (integrable) complex structure
 $$\mathbb I = \left (I_\lambda,  \ii \right),$$
 at the point $x \in \mathcal M $ and $\lambda \in \C \subset \C P^1$,
where $\ii$ is the standard complex structure of $\C P^1.$ 
Furthermore, the twistor space  has a natural anti-holomorphic involution $\mathcal T$ given by $(x,\lambda)\mapsto(x,-\bar\lambda^{-1})$.
By construction the twistor space has a holomorphic projection $\pi$ to $\C P^1$ and a twisted relative holomorphic symplectic form given by

\[\wh\varpi\in H^0(P,\Lambda^2V^*\otimes\mathcal O(2))\]
where
$V=\ker d\pi$ (is the complex tangent bundle to the fibers)
 and $\mathcal O(2)$ denotes the dual of the pull-back of the canonical bundle over $\C P^1.$
In terms of the K\"ahler forms the twisted relative holomorphic symplectic form has the following explicit expression
\begin{equation}\label{deftwistedOmega}\wh\varpi=\varpi\otimes\lambda\frac{\partial}{\partial\lambda},\end{equation}
where 
\begin{equation}\label{deftwistedOmega2}\varpi=\lambda^{-1}(\omega_J+i\omega_K)-2\omega_I-\lambda(\omega_J-i\omega_K),\end{equation}
see \cite[Equation (3.87)]{HKLR} or \eqref{eq:omlam} below. Sections of $\mathcal P$ are holomorphic maps 
$s \colon \C P^1 \rightarrow \mathcal P$ with $\pi \circ s = \Id.$ 

When restricting to $\mathcal M = \mathcal M_{SD},$ the simplest sections are {\em constant} sections
\begin{equation}\label{twistorPline}\lambda \mapsto ([\nabla, \Psi], \lambda)\end{equation}
for a solution of the self-duality equations $(\nabla, \Psi)$. By dimension count, these twistor lines give rise to an open subspace of the space of real holomorphic sections \cite{HKLR}.  The following well known theorem then follows  from the completeness of the moduli space of self-duality solutions \cite{Hi1}. 
\begin{theorem} \label{connectedcomponent}
Twistor lines form  an open and closed subset of the space of real holomorphic sections.
\end{theorem}

A convenient   set-up for studying associated families of flat connections obtained from solutions to self-duality equations is to consider them as real sections of the Deligne-Hitchin moduli space $\mathcal M_{DH}$. 
 
\subsection{The Deligne-Hitchin moduli space $\mathcal M_{DH}$}\label{DH}
The Deligne-Hitchin moduli space was first introduced by
Deligne (see \cite{Si, Si2}) as a complex analytic way of viewing the associated twistor space  of the moduli space of solutions to the self-duality equations. As such it interpolates between the moduli space of Higgs bundles $\mathcal M_{Higgs}$ and the moduli space of flat connections $\mathcal M_{dR}$.

\begin{definition}
Let $\Sigma$ be a compact Riemann surface and $\lambda\in\C$ fixed.
A (integrable) $\lambda$-connection on a $\mathcal C^\infty$-complex vector bundle $V\to \Sigma$ is a pair $(\dbar, D)$ consisting of a holomorphic structure $\dbar$ on $V$ and a linear first order differential operator
\[D\colon\Gamma(\Sigma,V)\to\Omega^{(1,0)}(\Sigma,V )\]
satisfying the $\lambda$-Leibniz rule
\[D(fs)=\lambda\partial f\otimes s+f Ds,\]
where $\partial$ is the trivial $\partial$-operator for functions $f$ and $s$ is a section of $V$, and the integrability condition
\begin{equation}\label{intcond}
D\dbar+\dbar D=0.\end{equation}
\end{definition}

\begin{remark}
The operators  $D$ and $\dbar$ also act on $(0,1)$-forms and $(1,0)$-forms respectively. 
For $\lambda = 0$ the integrability condition \eqref{intcond} is equivalent to
 \[D=\Psi\in H^0(M, K_\Sigma \End(V))\]
being a holomorphic endomorphism-valued 1-form, and for $\lambda \neq 0$ we have that 
\[\nabla=\tfrac{1}{\lambda}D+\dbar\]
is a flat connection. \end{remark}

\begin{example}
Consider on the hermitian  bundle $V \to\Sigma $ a solution $(\nabla=\del^\nabla+\dbar^\nabla, \Psi)$
of the self-duality equations.
 Then, the pair
\[(\dbar^\nabla + \lambda \Psi^*, \lambda \del^\nabla + \Psi) \]
defines a $\lambda$-connection on $V$ for every $\lambda \in \C$ which coincides with the Higgs pair  $(\dbar^\nabla, \Psi)$ at $\lambda = 0$ and with the flat connection $ \nabla^1 = \nabla 
+ \Psi + \Psi^*$ at $\lambda = 1.$
\end{example}

\begin{definition}\label{def:sllambda}
A $\SL (2, \C)$ $\lambda$-connection  is a $\lambda$-connection on a rank 2 vector bundle $V\rightarrow \Sigma$, such that
the induced $\lambda$-connection on the determinant bundle $\Lambda^2 V$ is trivial.\end{definition}

\begin{definition}
A $\SL(2, \C)$ $\lambda$-connection $(\dbar, D)$ is called stable, if
every $\dbar$-holomorphic subbundle $L\subset V$  with
\[D(\Gamma(\Sigma,L))\subset\Omega^{(1,0)}(\Sigma,L)\]
is of negative degree and semi-stable if its degree is non-positive.
All other $\lambda$-connections are called unstable. A $\SL(2,\C)$ $\lambda$-connection is called polystable if it is either stable or the direct sum
of dual $\lambda$-connections on degree zero line bundles.
\end{definition}

For $\lambda\neq0$,  every $\dbar$-holomorphic and $D$-invariant line subbundle $L\subset V$ must be parallel with respect to the flat connection $\nabla=\tfrac{1}{\lambda}D+\dbar$. Therefore, the degree of $L$ is $0$ and the $\lambda$-connection  $(\dbar, D)$  is semi-stable. 
Moreover, $(\dbar, D)$ is stable if and only if the flat connection $\nabla=\tfrac{1}{\lambda}D+\dbar$ is irreducible. The situation is different at $\lambda = 0$ and we need to  restrict to polystable
$\lambda$-connections to obtain a well-behaved moduli space.

\begin{definition}
The Hodge moduli space $\mathcal M_{Hod}=\mathcal M_{Hod}(\Sigma)$ is the space of all polystable $\SL(2,\C)$ $\lambda$-connections on $V=\Sigma\times\C^2\to \Sigma$ modulo gauge transformations, i.e., 
 $\mathcal M_{Hod}$ consists of gauge classes of  triples $(\lambda,\dbar,D)$ for $\lambda \in \C$ and $(\dbar,D)$ a polystable $\lambda$-connection. The gauge-equivalence class of $(\lambda,\dbar,D)$ is denoted by $$[\lambda,\dbar,D]\in\mathcal M_{Hod}$$ or by
\[[\lambda,\dbar,D]_\Sigma \in\mathcal M_{Hod}( \Sigma)\]
to emphasis its dependence on the Riemann surface.
\end{definition}

The Hodge moduli space admits a natural holomorphic map 
\[\pi_\Sigma\colon\mathcal M_{Hod}\longrightarrow \C; \quad [\lambda,\dbar,D] \longmapsto\lambda\] whose fiber at
$\lambda=0$ is the (polystable) Higgs moduli space $\mathcal M_{Higgs}$, and at $\lambda=1$ it is the deRham moduli space of flat (and totally reducible) $\SL(2,\C)$-connections $\mathcal M_{dR}$. We consider both spaces as complex analytic spaces endowed with their natural complex structures $I$ and $J,$ respectively. 

The next step is then to compactify the $\lambda$-plane $\C$ to $\C P^1.$
For a given Riemann surface $\Sigma$ let $\overline \Sigma$ be the Riemann surface with conjugate complex structure. As differentiable manifolds we have $\Sigma\cong\overline \Sigma$ and thus their deRham moduli spaces of flat $\SL(2,\C)$-connections
are naturally isomorphic. Then the two Hodge moduli spaces $\mathcal M_{Hod}(\Sigma)$ and $\mathcal M_{Hod}(\overline \Sigma)$ can be the glued together via Deligne gluing \cite{Si2} 

\[\mathcal G \colon\mathcal M_{Hod}(\Sigma)\setminus \pi_\Sigma^{-1}(0)\longrightarrow \mathcal M_{Hod}(\overline \Sigma)\setminus \pi_{\overline \Sigma}^{-1}(0);\quad [\lambda,\dbar,D]_\Sigma\longmapsto[\tfrac{1}{\lambda},\tfrac{1}{\lambda}D,\tfrac{1}{\lambda}\dbar]_{\overline \Sigma}\]
along $\lambda \in \C^*$ to give Deligne-Hitchin moduli space 
\[\mathcal M_{DH}=\mathcal M_{Hod}(\Sigma)\cup_{\mathcal G}\mathcal M_{Hod}(\overline \Sigma).\]

The natural fibration $\pi_\Sigma$ on $\mathcal M_{Hod} $ extends holomorphically to the whole Deligne-Hitchin moduli space to give  $\pi\colon\mathcal M_{DH}\to\C P^1$ whose
restriction to $\mathcal M_{Hod}(\overline \Sigma)$ is $1/\pi_{\overline \Sigma}.$

\begin{remark}\label{smoothMdh}
Note that the Deligne gluing  map $\mathcal G$ maps stable $\lambda$-connections over $\Sigma$ to stable $\tfrac{1}{\lambda}$-connections on $\overline \Sigma.$ Hence, it maps the smooth locus of $\mathcal M_{Hod}(M)$ (consisting of stable $\lambda$-connections) to the smooth locus of $\mathcal M_{Hod}(\overline M)$, and thus $\mathcal M_{DH}$ is equipped
with a complex manifold structure at all stable points. Moreover,  since $\lambda$-connections for $\lambda \in \C^*$ can be reinterpreted as flat connections, we can identify
$$\C^*\times \mathcal M_{dR}(\Sigma) = \mathcal M_{Hod}(\Sigma)\setminus \pi_\Sigma^{-1}(0).$$

\end{remark}

\begin{definition}
A section of $\mathcal M_{DH}$ is a holomorphic map 
$$s: \C P^1 \longrightarrow \mathcal M_{DH}$$
such that $\pi \circ s = $Id. 
\end{definition}
\begin{example}
The associated family of flat connections $\nabla^\lambda$ \eqref{nabla-lambda} to a solution of the self-duality equations  \eqref{SD} gives rise to a section of $\mathcal M_{DH} \rightarrow \C P^1$ via 
\begin{equation}\label{twistorsection}
s(\lambda) =  [\lambda,\dbar^\nabla+\lambda\Psi^*,\lambda \partial^\nabla+\Psi]_\Sigma\in\mathcal M_{Hod}(\Sigma)\subset\mathcal M_{DH}.\end{equation}
 
When identifying the Deligne-Hitchin moduli space with
 the twistor space $\mathcal P\to\C P^1$ of the hyper-K\"ahler space $\mathcal M_{SD}$ (at the smooth points),  the section given by \eqref{twistorsection} is identified with the `constant' twistor line \eqref{twistorPline}, see \cite{Si}.
\end{example}

\begin{definition}
A section $s$ of $\mathcal M_{DH}$ is called stable, if the $\lambda$-connection
$s(\lambda)$ is stable for all $\lambda\in \C^*$  and if the Higgs pairs $s(0)$   on $\Sigma$ and $s(\infty)$ on $\overline{\Sigma}$ are stable.  
\end{definition}
It follows from Hitchin \cite{Hi1} and Donaldson \cite{Do} that every stable point
in $\mathcal M_{DH}$ uniquely determines a twistor line. Therefore, a twistor line $s$ is already stable if $s(\lambda_0)$ is stable for some $\lambda_0\in\C$. Moreover,  twistor lines are in one-to-one correspondence with self-duality solutions \eqref{SD}. The following characterization of twistor lines as particular {\em negative real holomorphic  sections} of $\mathcal M_{DH}$ is useful to decide when certain real sections give rise to global solutions to the self-duality equations. 

\subsection{Automorphisms of the Deligne-Hitchin moduli space}
To define a real structure  on $\mathcal M_{DH}$ we need  to look at some natural automorphisms of  Deligne-Hitchin moduli space first. For every $\mu\in\mathbb C^*$ the (multiplicative) action of $\mu$ on $\C P^1$ has a natural lift
to $\mathcal M_{DH}$ by
\[\mu([\lambda,\dbar,D])=[\mu\lambda,\dbar,\mu D].\]

\begin{definition}
We denote by $N : \mathcal M_{DH} \rightarrow \mathcal M_{DH}$ 
the map given by multiplication with $\mu=-1$, namely
\[[\lambda,\dbar,D] \longmapsto [-\lambda,\dbar, -D].\]
\end{definition}

Furthermore, we have a natural anti-holomorphic automorphism denoted by $C.$ 
\begin{definition}\label{defC2}
Let $C\colon \mathcal M_{DH}\longrightarrow \mathcal M_{DH}$ be (the continuation of) the map
\[\widetilde C: \mathcal M_{Hod}(\Sigma) \longrightarrow \mathcal M_{Hod}(\overline \Sigma)\] given by 
 \begin{equation}\label{defcomcon}[\lambda,\dbar,D] \longmapsto [\bar\lambda,\bar\dbar,\bar D]_{\overline \Sigma}.\end{equation}

To be more concrete, for  
\[\dbar = \dbar_0+\eta \quad \text{and} \quad D = \lambda(\partial_0)+\omega\]
where $d=\partial_0+\dbar_0$ is the trivial connection, $\eta\in\Omega^{0,1}(\Sigma,\mathfrak{sl}(2,\C)),$ and $\omega\in\Omega^{1,0}(\Sigma,\mathfrak{sl}(2,\C)),$ we define
the complex conjugate on the trivial $\C^2$-bundle over $\overline \Sigma$ to be
 \[\bar \dbar = \partial_0+\bar\eta \quad \text{and} \quad \bar D = \bar\lambda(\dbar_0)+\bar\omega.\]
\end{definition}

The map $C$ covers the map 
\[\lambda\in\C P^1\longmapsto \bar\lambda^{-1}\in\C P^1.\]
Since $C$ and $N$ commute and both maps are involutive,  their composition 
\[\mathcal T=CN\] is an involution as well,
which covers the fixed-point free involution
$\lambda\mapsto - \bar\lambda^{-1}$ on $\C P^1.$

\subsection{Real sections}\label{subSrealS}
Consider
the anti-holomorphic involution of the associated Deligne-Hitchin moduli space 
\[\mathcal T=CN\colon\mathcal M_{DH}\longrightarrow\mathcal M_{DH} \]
covering the antipodal involution \[\lambda\longmapsto-\bar\lambda^{-1}\] of $\mathbb CP^1.$ We call a holomorphic section  $s$ of $\mathcal M_{DH}$ real (with respect to $\mathcal T$) if
\begin{equation}\label{taurealdefin}\mathcal T(s(\lambda))=s(-\bar\lambda^{-1})\end{equation}
holds for all $\lambda\in\mathbb CP^1$.
\begin{example}
Twistor lines \eqref{twistorsection} are real holomorphic sections with respect to $\mathcal T$.  
 Let $(\nabla,\Psi)$ be a solution of the self-duality equations on $\Sigma$ with respect to the standard hermitian metric on $\underline \C^2\to \Sigma.$ Because we are dealing with $\mathfrak{sl}(2,\C)$-matrices, the unitary connection $\nabla$ satisfies
  \[\nabla=\nabla^*=\bar\nabla.\begin{pmatrix}0&1\\-1&0\end{pmatrix}\]
  which is equivalent to
  \[\bar\partial^\nabla=\overline{\partial^\nabla}.\begin{pmatrix}0&1\\-1&0\end{pmatrix}\quad\text{and}\quad \partial^\nabla =\overline{\bar\partial^\nabla}.\begin{pmatrix}0&1\\-1&0\end{pmatrix},\]
 and  analogously we have
 \[\Psi=-\begin{pmatrix}0&-1\\1&0\end{pmatrix}\overline{\Psi^*} \begin{pmatrix}0&1\\-1&0\end{pmatrix}\quad\text{and}\quad \Psi^*=-\begin{pmatrix}0&-1\\1&0\end{pmatrix}\bar\Psi \begin{pmatrix}0&1\\-1&0\end{pmatrix}.\]
Therefore, the twistor line \eqref{twistorsection} satisfies
\begin{equation}
\begin{split}
\mathcal T(s(\lambda))&=\mathcal T([\lambda, \bar\partial^\nabla+\lambda\Psi^*,\Psi+\lambda\partial^\nabla]_\Sigma)\\
&=[-\bar\lambda^{-1}, \bar\lambda^{-1}(\overline{\Psi +\lambda\partial^\nabla}),-\bar\lambda^{-1}(\overline{\bar\partial^\nabla+\lambda\Psi^*})]_\Sigma\\
&=[-\bar\lambda^{-1}, \overline{\partial^\nabla}+\bar\lambda^{-1}\overline{\Psi},-\overline{\Psi^*}
-\bar\lambda^{-1}\overline{\bar\partial^\nabla}]_\Sigma\\
&=[(-\bar\lambda^{-1}, \overline{\partial^\nabla}+\bar\lambda^{-1}\overline{\Psi},-\overline{\Psi^*}
-\bar\lambda^{-1}\overline{\bar\partial^\nabla}   ).\begin{pmatrix}0&1\\-1&0\end{pmatrix} ]_\Sigma\\
&=[(-\bar\lambda^{-1}, \bar\partial^\nabla-\bar\lambda^{-1}\Psi^*,\Psi
-\bar\lambda^{-1}\partial^\nabla   )]_\Sigma= s(-\bar\lambda^{-1}).
\end{split}
\end{equation}
\end{example}

Let $s(\lambda)$ be a real holomorphic section with $\nabla^\lambda\sim \lambda^{-1}\Psi+\nabla+\dots$ being a lift to the space of flat connections such that 
$(\bar\partial^\nabla,\Psi)$ is stable. 
For the existence of such lifts see \cite[Lemma 2.2 ] {BHR}.
Then the reality condition \eqref{taurealdefin} gives rise to a holomorphic family of gauge transformations $g(\lambda)$ satisfying 
$$\overline{\nabla^{-\bar\lambda^{-1}}}=  \nabla^\lambda .g(\lambda).$$

Applying this equation twice we obtain 

$$\nabla^\lambda.g(\lambda) \overline{g(-\bar\lambda^{-1})} = \nabla^\lambda.$$

Because the section $s$ is stable, the connections $\nabla^\lambda$ are irreducible for all $\lambda \in \mathbb C^*$. Therefore $g(\lambda) \overline{g(-\bar\lambda^{-1})}$ is a constant multiple of the identity for every $\lambda \in \C^*$.
By \cite[Lemma 1.18]{HH}  we can choose $g$ to be SL$(2, \C)$-valued and thus
\begin{equation}\label{realeqsecsign}
g(\lambda)\overline{g(-\bar\lambda^{-1})}=\pm\text{Id}.\end{equation}
The sign on the right hand side is independent of the lift $\nabla^\lambda$ of $s$ and is preserved in a connected component of real sections motivating the following definition.
\begin{definition}\cite[Definition 2.16]{BHR}
A stable real holomorphic section $s$ of $\mathcal M_{DH}$ is called positive or negative depending on the sign of \eqref{realeqsecsign}.
\end{definition}

\begin{example}\label{exa:negative}
Twistor lines are always negative sections.
In fact, the associated family of flat connections is a canonical lift of the twistor line to the space of flat connections, and with respect to the standard hermitian structure on $\mathbb C^2$ the gauge 
\begin{equation}\label{g=delta}
g(\lambda)=\begin{pmatrix}0&1\\-1&0\end{pmatrix},
\end{equation}
as in \eqref{realeqsecsign}, is constant in $\lambda$ and squares 
 to $-\text{Id}.$
\end{example}

The space of real sections of $\mathcal M_{DH}$ has multiple connected components. A negative real holomorphic section lying in the connected component of the twistor lines must be a twistor line itself by Theorem \ref{connectedcomponent}.  It corresponds therefore to a global solution of the self-duality equation.

\subsection{Loop groups}\label{sec:loops}

We briefly introduce basic notions and statements  about loop groups which are relevant with regard to the paper. For details we refer to \cite{PS} or \cite{DPW}. Define 
\begin{itemize}
\item $\Lambda \mathrm{SL}(2,\C):= \{$ real analytic maps (loops) $\Phi\colon \S^1\longrightarrow \mathrm{SL}(2,\C),\quad \lambda \longmapsto \Phi^\lambda\};$

\item $\Lambda \mathfrak{sl}(2,\C):=\{$ real analytic maps (loops) $\eta\colon \S^1\longrightarrow\mathfrak{sl}(2,\C),\quad  \lambda \longmapsto \eta^\lambda\}.$ 
\end{itemize}
Then, $\Lambda \mathrm{SL}(2,\C)$ is an infinite dimensional Frechet Lie group via pointwise multiplication with $\Lambda \mathfrak{sl}(2,\C)$ as its
Lie algebra. We consider $\S^1=\{\lambda\in\C P^1\mid \lambda\bar\lambda=1\},$ and denote 
\begin{equation}
\begin{split}
\Lambda_+\mathrm{SL}(2,\C)&=\{\Phi\in\Lambda \mathrm{SL}(2,\C)\mid \Phi \text{ extends holomorphically to } \lambda=0\}\\
\Lambda_-\mathrm{SL}(2,\C)&=\{\Phi\in\Lambda \mathrm{SL}(2,\C)\mid \Phi \text{ extends holomorphically to } \lambda=\infty\}
\end{split}\end{equation}
and
\[\Lambda_+\mathfrak{sl}(2,\C)=\{\eta\in\Lambda \mathfrak{sl}(2,\C)\mid \eta \text{ extends holomorphically to } \lambda=0\}.\]
We also denote
\[\Lambda_-^*\mathrm{SL}(2,\C)=\{B\in\Lambda_-\mathrm{SL}(2,\C)\mid B(\infty)=\Id\}.\]

The following classical theorem,  due to Birkhoff, Grothendieck and others, is essential for our method, see \cite{PS} for details and \cite{DPW} for the loop group method for harmonic maps from simply connected 2-dimensional domains into
positively curved symmetric spaces.

\begin{theorem}\label{thm:birk}
There is an open and dense subset $\mathcal U\subset\Lambda \mathrm{SL}(2,\C)$ called the {\em big cell} such that every
$g\in\mathcal U$
admits a  factorization
 $$g=g^+ g^-$$
 with $g^+\in\Lambda_+\mathrm{SL}(2,\C)$ and $g^- \in \Lambda_-^*\mathrm{SL}(2,\C).$
This splitting is unique and depends holomorphically on $g\in\mathcal U.$ The pair $(g^+,g^-)$ is called the Birkhoff factorization 
of $g$.\\
\end{theorem}

\subsection{Reconstruction of self-duality solutions from admissible negative real sections}\label{sec:reconstruction}

A section $s$ of the  Deligne-Hitchin moduli space has lifts to families of flat connections in both Hodge moduli spaces. Let 
$$\wt\nabla^\lambda = \lambda^{-1} \Psi_1 + \nabla + \text{higher order terms in } \lambda$$
 be a lift around $\lambda = 0.$ Due to $s$ being real, there exists a family of $\mathrm{SL}(2,\C)$ gauges $g(\lambda)$ satisfying 
$$\overline{\wt\nabla^{-\bar \lambda^{-1}}} = \wt \nabla^\lambda . g(\lambda)\,.$$

 Assume that  the Birkhoff factorization 
$$g(\lambda) = g^+(\lambda) g^-(\lambda)$$
 exist for every $z \in \Sigma$.

Assume that the section $s$ is negative, i.e, 
\begin{equation}\label{eq:pmgpm1}\overline{g(-\bar\lambda^{-1})}^{-1} =   \overline{g^- (-\bar \lambda^{-1})}^{-1}  \overline{g^+ (-\bar \lambda^{-1})}^{-1} = -g(\lambda).\end{equation}
Since $\lambda \rightarrow -\bar \lambda^{-1}$ interchanges the $()^+$ and $()^-$ parts of the Birkhoff factorization, the uniqueness assertion gives
\begin{equation}\label{eq:pmgpm2}g^+(\lambda) = -\overline{g^- (-\bar \lambda^{-1})}^{-1} B^{-1} \quad \text{ and } \quad  g^-(\lambda) = B \overline{g^+ (-\bar \lambda^{-1})}^{-1}\end{equation}
for  $B=(g^+(\lambda=0))^{-1}\colon\Sigma\rightarrow \SL(2, \C)$.
Combining \eqref{eq:pmgpm1} and \eqref{eq:pmgpm2} we get
 $\bar B B = -\Id$. This implies that $B$ lies in the same conjugacy class as $\delta = \begin{pmatrix}0 & 1\\-1 & 0 \end{pmatrix},$ i.e., there exist $G\colon\Sigma\to\mathrm{SL}(2,\C)$ satisfying
$\delta=G^{-1} B \overline{G}.$ Then 
$$\nabla^\lambda =\wt \nabla^\lambda . (g^+(\lambda) G) $$
 satisfies $ \overline{\nabla^{-\bar \lambda^{-1}}}= \nabla^\lambda .\delta$ and is therefore the associated family of a self-duality solution with respect to the standard hermitian metric.
 
\begin{remark}
In this paper we construct real holomorphic sections via loop group methods, i.e, we write down a particular lift of a real section $s$ of the Deligne-Hitchin moduli space around $\lambda =0$ in terms of a Fuchsian potential $\eta$ (i.e., a $\lambda$-dependent  Fuchsian connection 1-form). To obtain the actual harmonic map into the hyperbolic 3-space $\H^3$ (or the associated self-duality solution), we need thus to first perform a global Birkhoff factorization and then show that the obtained section is negative. This needs further conditions, as examples for which the harmonic map into $\H^3$ becomes singular and intersects the boundary at infinity of $\H^3$ exists
 \cite{HH} as well as positive real sections with global Birkhoff factorization that give rise to harmonic maps into the de-Sitter 3-space \cite[Theorem 3.4]{BHR}. 
  \end{remark}

\subsection{Goldman's symplectic form on the moduli space of $\lambda$-connections}\label{sec:goldman}

Fix $\lambda\in\C$
and consider the space of $\lambda$-connections modulo gauge transformations.
Let $(\bar\partial,D)$ be a $\lambda$-connection. A tangent vector to the (infinite dimensional) space of $\lambda$-connection is given by 
\[(A,B)\in\Omega^{(0,1)}(\mathfrak{sl}(2,\C))\oplus\Omega^{(1,0)}(\mathfrak{sl}(2,\C))\]
satisfying the linearized compatibility (flatness) condition
\begin{equation}
\begin{split}
0&=\bar\partial B+ DA.
\end{split}
\end{equation}
Tangent vectors at $(\dbar, D)$ which are generated by the infinitesimal gauge transformation $\xi\in\Gamma(\Sigma,\mathfrak{sl}(2,\C))$
 are given by
\[(A,B)=(\bar\partial\xi,D\xi).\]

The Goldman symplectic structure $\Omega^\lambda$ on the moduli space of $\lambda$-connections \cite[Section 1]{Gold} is defined to be
\begin{equation}\label{GoldmanO}
\Omega^\lambda( (A_1,B_1),(A_2,B_2))=4\int_\Sigma\tr(A_1\wedge B_2-A_2\wedge B_1)\end{equation}
for $(A_1,B_1), (A_2,B_2)\in\Omega^{(0,1)}(\mathfrak{sl}(2,\C))\oplus\Omega^{(1,0)}(\mathfrak{sl}(2,\C))$
representing tangent vectors.
On a compact Riemann surface, the symplectic structure $\Omega^\lambda$ is gauge invariant, and thus can be computed using arbitrary representatives of the tangent vectors, which makes it well-defined on the moduli space of $\lambda$-connections. Since we are not aware of any explicit reference (except for $\lambda=0$ and $\lambda=1$), we give the simple proof
(which is of course an  instance of an infinite-dimensional symplectic reduction for the gauge group action with the curvature as moment map, see \cite{AB}).
\begin{lemma}
Let $\lambda\in\C$ be fixed.
The holomorphic symplectic form $\Omega^\lambda$ is well-defined on the moduli space of $\lambda$-connections.
\end{lemma}
\begin{proof}
By construction, $\Omega^\lambda$ is complex bilinear. Proposition \ref{proId} below identifies $\Omega^\lambda$ with the
twisted holomorphic symplectic form \eqref{deftwistedOmega2} restricted to the fiber over $\lambda,$  up to a constant factor.  Thus closeness and non-degeneracy follows from \cite[ Section 3(F)] {HKLR}. 
It remains to show that $\Omega^\lambda$ descents to the moduli space. A gauge transformation acts on tangent vectors  to the space of connections by conjugation, and since the trace is conjugation invariant we obtain
\[\Omega^\lambda((A_1,B_1).g,(A_2,B_2).g)=\Omega^\lambda((A_1,B_1),(A_2,B_2)).\]
Let $(A_1,B_1)$ satisfy $0=\bar\partial B+ DA$ and let $(A_2,B_2)=(\bar\partial\xi,D\xi)$ be a tangent vector that is also tangent to the gauge orbit.
Then,
\begin{equation}\label{def:goldi}
\begin{split}
\Omega^\lambda((A_1,B_1),(A_2,B_2))&=4 \int_\Sigma\tr(A_1\wedge D\xi-\bar\partial\xi\wedge B_1)\\
&=-4\int_\Sigma d\,\tr(A_1\xi+\xi B_1)-4\int_\Sigma \tr(DA_1\xi+\xi \bar\partial B_1) =0.
\end{split}
\end{equation}
Therefore,  $\Omega^\lambda$ is well-defined on the quotient space (the moduli space of $\lambda$-connections). 
\end{proof}

The Goldman symplectic form in fact coincides (up to scaling) with the twisted holomorphic symplectic form \eqref{deftwistedOmega2}
on the twistor space $\mathcal P = \mathcal M_{SD} \times \C P^1$ of the moduli space of self-duality solutions, also compare with \cite[Equation (4.10)]{GMN}.

\begin{proposition}\label{proId}
Let $\lambda\in\C\subset \C P^1$ be fixed. Then, the fiber $\mathcal P_\lambda=\pi^{-1}(\lambda)$ is the moduli space of $\lambda$-connections, and
\[\lambda\varpi_{\mid  \mathcal P_\lambda} = \Omega^\lambda\]
is the Goldman symplectic form  on the moduli space of $\lambda$-connections.
\end{proposition}

\begin{proof}
For $\lambda \in \C$ fixed, it is shown in \cite[Theorem 4.2]{Si2} that $\mathcal P_\lambda=\pi^{-1}(\{\lambda\})$ is the moduli space of $\lambda$-connections.
In order to evaluate the corresponding symplectic forms, we need to first find appropriate representatives for the tangent vectors.
To do so, let $h$ be the standard hermitian metric and let $(\nabla,\Psi)$ be a solution of the self-duality equations.  The tangent space of $\mathcal M_{SD}$  at $[\nabla,\Psi]$ is given by \[(\xi,\phi) \in \Omega^{(0,1)}(\mathfrak{sl}(2,\C))\oplus\Omega^{(1,0)}(\mathfrak{sl}(2,\C))\]
satisfying 
\begin{equation}\label{goeq1}
\begin{split}
0&=d^\nabla (\xi-\xi^*)+\tfrac{1}{2}[\phi\wedge\Psi^*]+\tfrac{1}{2}[\Psi\wedge \phi^*]\\
0&=\bar\partial^\nabla \phi+[\xi,\Psi]\\
\end{split}
\end{equation}
modulo infinitesimal gauge deformations (see \cite[(6.1)]{Hi1}).  
These equations mean that we have an infinitesimal deformation
\[t\mapsto (\nabla+t (\xi- \xi^*),\Psi+t\phi)\]
 of a solution to the self-duality equation (for the fixed hermitian metric $h$), and we can choose harmonic representatives which are
 orthogonal to the (unitary) gauge orbit, which is equivalent to
 \begin{equation}\label{goeq2}
\begin{split}\mathcal D^*(\xi,\phi)=0
\end{split}
\end{equation}
 where
 \[\mathcal D \colon\Gamma(\Sigma,\mathfrak{su}(2))\to\Omega^{(0,1)}(\Sigma,\mathfrak{sl}(2))\oplus\Omega^{(1,0)}(\Sigma,\mathfrak{sl}(2)); \quad \psi\mapsto \left((d^\nabla\psi)'',[\Psi,\psi]\right)\]
 see \cite{Hi1, Fred}. Provided \eqref{goeq1}, \eqref{goeq2} hold,
the complex structures $I,J,K=IJ$ are given by
\begin{equation}
\begin{split}
I(\xi,\phi)=(i\xi,i\phi)\quad 
J(\xi,\phi)=(i\phi^*,-i\xi^*) \and
K(\xi,\phi)=(-\phi^*,\xi^*),
\end{split}
\end{equation}
see   \cite[page 109]{Hi1}. Provided \eqref{goeq1}, \eqref{goeq2} hold for  $(\xi_k, \phi_k)$, $k=1,2$,
the Hitchin metric on $\mathcal M_{SD}$ (see \cite[(6.2)]{Hi1}) is defined to be
\[g\left((\xi_1,\phi_1),(\xi_2,\phi_2)\right):=2i\int_\Sigma\tr\left(\xi_1^*\wedge \xi_2+\xi_2^*\wedge \xi_1+\phi_1\wedge\phi_2^*+\phi_2\wedge\phi_1^*\right).\]
Thus 
\begin{equation*}
\begin{split}
\omega_I((\xi_1,\phi_1),(\xi_2,\phi_2))&=-g((\xi_1,\phi_1),(i\xi_2,i\phi_2))\\
&=2\int_\Sigma\tr(\xi_1^*\wedge \xi_2-\xi_2^*\wedge \xi_1-\phi_1\wedge\phi_2^*+ \phi_2\wedge\phi_1^*)\\
\omega_J((\xi_1,\phi_1),(\xi_2,\phi_2))&=2\int_\Sigma\tr(\xi_1^*\wedge \phi_2^*-\phi_2\wedge \xi_1+\phi_1\wedge\xi_2-\xi_2^*\wedge\phi_1^*)\\
\omega_K((\xi_1,\phi_1),(\xi_2,\phi_2))&=-2i\int_\Sigma\tr(-\xi_1^*\wedge \phi_2^*-\phi_2\wedge \xi_1+\phi_1\wedge\xi_2+\xi_2^*\wedge\phi_1^*)
\end{split}
\end{equation*}
which gives
\begin{equation*}
\begin{split}
(\omega_J+i\omega_K)(\xi_1,\phi_1),(\xi_2,\phi_2)&=-4\int_\Sigma\tr(\phi_2\wedge\xi_1-\phi_1\wedge\xi_2).
\end{split}
\end{equation*}

Recall that twistor lines corresponding to the self duality solution $(\nabla,\Psi)$ are given by
\[\lambda\longmapsto[\lambda,\bar\partial^\nabla+\lambda\Psi^*,\Psi+\lambda\partial^\nabla]_\Sigma \in \mathcal M_{Hod} \subset \mathcal M_{DH}\]
or equivalently, by the associated family of flat connections
 \begin{equation}\label{associated_family}\lambda\in\C^*\longmapsto \nabla+\lambda^{-1}\Psi+\lambda\Psi^*.\end{equation}
Thus, the tangent vectors $X=(\xi_1,\phi_1)$ and $Y=(\xi_2,\phi_2)$ 
correspond to sections of the normal bundle to the twistor line in $\mathcal M_{DH}$ given by
\[X:  \lambda \longmapsto  [0,\xi_1+\lambda \phi_1^*,\phi_1-\lambda\xi_1^*]\]
and
\[Y: \lambda \longmapsto [0,\xi_2+\lambda \phi_2^*,\phi_2-\lambda\xi_2^*].\]
 Therefore, when fixing $\lambda \in \C$, we obtain from \eqref{GoldmanO}
\begin{equation}\label{eq:omlam}
\begin{split}
\Omega^\lambda(X,Y)=&\Omega^\lambda((\xi_1+\lambda \phi_1^*,\phi_1-\lambda\xi_1^*),(\xi_2+\lambda \phi_2^*,\phi_2-\lambda\xi_2^*))\\
=&(\omega_J+i\omega_K)(X,Y)-2\lambda\omega_I(X,Y)-\lambda^2(\omega_J-i\omega_K)(X,Y).
\end{split}
\end{equation}
\end{proof}

\section{Strongly parabolic Higgs bundles}\label{sec:paradh}
In this section we generalise the setup in section \ref{pre} to punctured Riemann surfaces.  We introduce strongly parabolic Higgs fields, and explain their relationship to Higgs bundles on compact surfaces when their parabolic weights are rational. We will restrict to the $\mathfrak{sl}(2,\C)$-case, which we refer to as {\em trace-free}. In particular,  the holomorphic bundles are of rank 2, have trivial determinant bundle, and the corresponding Higgs fields are trace-free.

\subsection{Parabolic structures and logarithmic connections} Let $\Sigma$ be a compact Riemann surface, $p_1,\dots,p_n\in \Sigma$ be pairwise distinct and
$\Sigma^0=\Sigma\setminus\{p_1,\dots,p_n\}.$ Let $V=\C^2\times\Sigma\to\Sigma$ be the trivial $ C^\infty$-complex vector bundle
over $\Sigma$ of rank two and $d=\partial^0+\bar\partial^0$ be the decomposition of the deRham differential $d$ on $\Sigma$ into its $(1,0)$ and $(0,1)$ parts.
Furthermore, let $\bar\partial=\bar\partial^0+\gamma$, $\gamma\in\Gamma(\Sigma,\bar K_\Sigma\mathfrak{sl}(2,\C))$ be a holomorphic structure on $V$, and consider the corresponding holomorphic rank 2 vector bundle $\mathcal V:=(V,\bar\partial).$ 
Since $\gamma$ is trace free, the determinant bundle $\Lambda^2V$ of $\mathcal V$ is holomorphically trivial. 
 A holomorphic $\mathrm{SL}(2,\C)$ frame $F=(s,t)$ on an open set $U\subset \Sigma$ is given by
two holomorphic sections $s,t\in H^0(U,\mathcal V)$ such that $s\wedge t=1.$

\begin{definition}
A (trace-free) parabolic structure $\mathcal P$ on the holomorphic bundle  $\mathcal V$ with parabolic divisor $D=p_1+\dots+p_n$
is given by a collection of quasiparabolic lines, i.e., 1-dimensional vector subspaces $\ell_j\subset V_{p_j} $,
together with parabolic weights $\alpha_j\in[0,\tfrac{1}{2}[$ 
for $j=1,\dots,n$.
\end{definition}
For a general definition of parabolic structures see for example  \cite{MS,Si1}, and for the trace-free convention see \cite{Pir}.

The {\em parabolic degree} of a holomorphic subbundle $L$ of $\mathcal V$ on $\Sigma$ is defined by
\[\text{pdeg}(L)=\deg(L)+\sum_j (-1)^{\sigma_j} \alpha_j,\quad\text{with}\quad \sigma_j=\begin{cases} 0 &\quad\text{if}\; L_{p_j}=\ell_j\\ 1 &\quad\text{if}\; L_{p_j}\neq\ell_j \end{cases}.\]
Two different parabolic structures $\mathcal P$ and $\wt{\mathcal P}$ on two holomorphic bundles $\mathcal V$ respectively $\wt{\mathcal V}$ over $\Sigma$ are called {\em isomorphic} if the singular divisors $D=\wt D$ coincide, the parabolic weights are the same $\alpha_j=\wt \alpha_j$ for all $j, $  and if there is a holomorphic bundle isomorphism $\Phi\colon \mathcal V\to\wt{\mathcal V}$ mapping the parabolic lines onto each other,  i.e., $\Phi(\ell_k)=\wt \ell_j$ for all $j$.

\begin{definition}
A parabolic structure $\mathcal P$ on the holomorphic bundle  $\mathcal V$ is called (semi-)stable if and only if
all holomorphic subbundles have negative (non-positive) parabolic degree. A parabolic structure is called {\em polystable}
if it is either stable or the underlying holomorphic bundle $\mathcal V$ is the direct sum of two holomorphic line subbundles
of parabolic degree $0.$
\end{definition}
 The notion of a parabolic structure was introduced to generalize the Narasimhan-Seshadri theorem \cite{NS} to the case of punctured Riemann surfaces \cite{MS}. 

\begin{definition}
A (trace-free) logarithmic connection on the holomorphic bundle $\mathcal V \rightarrow \Sigma$ with singular divisor 
$D=p_1+\dots+p_n$ is  a connection $\nabla$ on $\mathcal V{\mid_ {\Sigma^0}}$ with $\bar\partial^\nabla=\bar\partial_{V}{\mid_ {\Sigma^0}},$ such that with respect to any holomorphic $\mathrm{SL}(2,\C)$ frame $F$ of $\mathcal V$ on an open subset $U\subset\Sigma$,
the connection 1-form $\alpha=\alpha_F$ of $\nabla$ with respect to $F$ is a meromorphic $\mathfrak{sl}(2,\C)$-valued 1-form with at most first order poles  at $U\cap\{p_1,\dots,p_n\}$ and holomorphic elsewhere.
\end{definition}
Note that a logarithmic connection on a Riemann surface is automatically flat, as its connection 1-form is meromorphic.
A holomorphic frame $F$ of $\mathcal V$ restricted to $U\subset\Sigma$ can be regarded as a holomorphic isomorphism from the trivial holomorphic rank 2 bundle to $\mathcal V$ restricted to $U.$
Thus, the {\em residue} 
\[\text{res}_{p_j}\nabla:=F \circ \text{res}_{p_j}\alpha_F \circ F^{-1}\in \mathfrak{sl}(V_p)\]
of a logarithmic connection $\nabla$ on a holomorphic bundle $\mathcal V$ at a singular point $p_j$ is well-defined, i.e., independent of the choice of the holomorphic frame.

Assume that the two eigenvalues $\pm\alpha_j$ of $\text{res}_{p_j}\nabla$ do not differ by an integer. Then, the (local) monodromy of $\nabla$ along a simple closed curve surrounding $p_j$ lies in the conjugacy class
of 
\begin{equation}\label{eq:locmon}\exp(2\pi i\, \text{res}_{p_j}\nabla)\cong\begin{pmatrix}e^{2\pi i\alpha_j}&0\\0&e^{-2\pi i\alpha_j}\end{pmatrix}.\end{equation}
In particular, if the eigenvalues $\pm \alpha_j$ are real with $\alpha_j\in (0,\tfrac{1}{2})$, then the local monodromy lies in the conjugacy class of unitary matrices. In the following, we always assume that this is the case, i.e., 
we impose the condition \[\det(\text{res}_{p_j}\nabla)\in(-\tfrac{1}{4},0)\]
on the trace-free residues. Then a logarithmic connection determines a parabolic structure with underlying holomorphic bundle  $\mathcal V$ and singular divisor $D=p_1+\dots p_n$ such that the parabolic weight $0<\alpha_j< \tfrac{1}{2}$ at $p_j$ is
the positive eigenvalue of $\text{res}_{p_j}\nabla$ and the quasiparabolic line $\ell_j$ is the corresponding eigenline. 
We will refer to the positive eigenvalues of the residues $\text{res}_{p_j}\nabla$ as the parabolic weights of $\nabla.$
If the monodromy representation is irreducible and unitary (up to overall conjugation), then the parabolic structure is stable. The converse statement also holds by the  Mehta-Seshadri theorem \cite{MS}: for every stable parabolic structure $\mathcal P$ there exists a (unique) logarithmic connection $\nabla$ with unitary monodromy (up to conjugation) inducing the parabolic structure $\mathcal P.$

This correspondence extends to reducible unitary logarithmic connections with strictly polystable parabolic structures: the induced parabolic structure of a unitary connection with abelian monodromy is the direct sum of two holomorphic line subbundles of parabolic degree 0. And by abelian Hodge theory 
and the residue theorem, the converse is true as well.

\subsubsection{Parabolic structures on the 4-punctured sphere}
In this paper, we mainly study the case of rank $2$ trace-free parabolic structures over $\Sigma=\C P^1$ with 4 singular points $p_1,\dots, p_4\in\mathbb CP^1$ such that the parabolic weights $\alpha_j$ satisfy
\begin{equation}\label{para-weights}\alpha_1=\alpha_2=\alpha_3=\alpha_4=t\in (0,\tfrac{1}{2}).\end{equation}
By the Birkhoff-Grothendieck theorem, every holomorphic rank 2 vector bundle over $\C P^1$ with trivial determinant 
is of the form $\mathcal V=\mathcal O(k)\oplus\mathcal O(-k)$ for some unique $k\in \N^{\geq0}.$
\begin{lemma}\label{lemmaOk}
For $k\geq 2$, the holomorphic bundle $\mathcal V= \mathcal O(k)\oplus\mathcal O(-k) \rightarrow \C P^1$ does not admit a logarithmic connection on the 4-punctured sphere
with parabolic weights \eqref{para-weights}. In particular, there is no stable or semi-stable parabolic structure with
these weights.
\end{lemma}
\begin{proof}
Assume there is such a connection $\nabla$ on  $\mathcal V= \mathcal O(k)\oplus\mathcal O(-k).$
Writing $\nabla$ with respect to this splitting gives
\[\nabla=\begin{pmatrix} D &\beta_+\\ \beta_-&D^*\end{pmatrix}.\]
Then,   the upper-left entry $D$ is a logarithmic connection on $\mathcal O(k)$ (the definition is analogous to the $\mathfrak{sl}(2,\C)$ case) with dual
connection $D^*,$ and $\beta_+\in H^0(\C P^1, K\mathcal O(2k+4))$ and  $\beta_-\in H^0(\C P^1, K\mathcal O(-2k+4))$, where the $\mathcal O(4)$ factor comes from the 4 singularities.
Since $K=K_{\C P^1}\cong\mathcal O(-2)$ we have by degree count that
 $\beta_-\equiv0$. This is a contradiction,
since in this case $D$ would be a meromorphic connection on $L= O(k)$ with residues $\pm t$ at the 4 singular points, thus the residue formula gives $0=\text{pdeg}(L)\geq k-4 t>0.$ 

The same computation gives that the holomorphic bundle $\mathcal V$ does not admit a stable or semi-stable parabolic structure, since 
the parabolic degree of $\mathcal O(k)$ is always positive.
\end{proof}

The two remaining cases $k=1$ and $k=0$ behave quite differently.
\begin{lemma}\label{lemOK1} Let $t \in (0, \tfrac{1}{4}) \cup (\tfrac{1}{4}, \tfrac{1}{2}).$
Then there is, up to isomorphism, exactly one parabolic structure $\mathcal P$ on $\mathcal V=\mathcal O(1)\oplus\mathcal O(-1)$ over $\C P ^1$ with 4 singular points and parabolic weights \eqref{para-weights} which admits a logarithmic connection. 
The parabolic structure $\mathcal P$ is stable if and only if $t>\tfrac{1}{4}.$ Furthermore, up to isomorphism, the space
of logarithmic connections on $\mathcal O(1)\oplus\mathcal O(-1)$ is a complex affine line.
\end{lemma}
\begin{proof}We refer to \cite[Proposition 18]{HH3}  for the proof.
\end{proof}

If the bundle $\mathcal V\cong\mathcal O\oplus\mathcal O$ is trivial, we can compare the different quasiparabolic lines over the different singular points with respect to a fixed background $\mathbb C^2$. Since we will be focusing on $t\sim 0$ later, we only discuss stability for $t<\tfrac{1}{4}$ here.
 \begin{lemma}\label{lem:stableparalemma}
Let $t\in (0,\tfrac{1}{4}).$
A parabolic structure on $\mathcal O\oplus\mathcal O$ over $\C P^1$ with 4 singular points and parabolic weights \eqref{para-weights}
admits a logarithmic connection if and only if every parabolic line $\ell_j,$ $j = 1, .., 4,$ coincides with at most one other line $\ell_k$ for $k= 1, ..., 4$ and $k\neq j$. The parabolic structure $\mathcal P$ is stable 
if and only if the parabolic lines $\ell_j,$ $j= 1, ...,4,$ are pairwise distinct, and $\mathcal P$ is the direct sum of two holomorphic line subbundles
of parabolic degree 0 if and only if we have two pairs of (distinct) parabolic lines.
\end{lemma}
\begin{proof}
Let $\nabla$ be a logarithmic connection on  $\mathcal V=\mathcal O\oplus\mathcal O$. Let $L$ be some holomorphic line subbundle of degree 0, i.e., it is isomorphic to $\mathcal O$. Take a complementary subbundle $L^*$, which is isomorphic to $\mathcal O$ as well, and decompose 
\[\nabla=\begin{pmatrix} D&\beta_+\\\beta_-&D^*\end{pmatrix}\] with respect to $\mathcal V=L\oplus L^*,$ where $D$ is a logarithmic connection on $L$, and \[\beta_-\in H^0(\C P^1,K\mathcal O(4-n)),\]  $n\in\{0,\dots,4\}$ being the number of points $p_j$ for which $L_{p_j} = \ell_j$. For $n>2$ this implies $\beta_-=0$ which then
gives a contradiction to the residue theorem for the line bundle connection $D$. Hence, we obtain $n<3.$
The converse direction can be shown easily by giving explicit formulas using Fuchsian systems, or it follows directly by the Mehta-Seshadri theorem once we know that they are stable.

To show stability, note first that there is no holomorphic sub line bundle of $\mathcal O \oplus \mathcal O$ with positive degree. Moreover, the parabolic degree of a line subbundle $\mathcal O(k)\subset \mathcal O \oplus \mathcal O$ with $k<0$ is always negative since $t<\tfrac{1}{4}$. Thus it remains to consider line subbundles $L$ that are isomorphic to $\mathcal O,$ which means $L$ is constant. If all parabolic lines $\ell_j$ are different, then $n\leq 1$ and we have
pdeg$(L) \leq -2t <0$, from which stability follows.

The case $n=2$ follows from similar arguments together with the residue theorem.
\end{proof}

\begin{proposition}\label{pro:paramodul}
The moduli space $\mathcal M_{Bun}$ of polystable parabolic structures $\mathcal P$ over $\C P^1$ with 4 singular points $p_1,\dots,p_4$ and 
 with parabolic weights given by $t\in(0,\tfrac{1}{4})$ can be identified with $\C P^1$ with three $\Z_2$-orbifold points $0,1,\infty$.
\end{proposition}
\begin{proof}
Since every polystable parabolic bundle admits a (unitary) logarithmic connection by the Mehta-Seshadri theorem
the underlying holomorphic bundle is $\mathcal O\oplus\mathcal O$ by Lemma \ref{lemmaOk} and Lemma \ref{lemOK1}.
Hence, holomorphic gauge transformations of the underlying holomorphic bundle $\mathcal V$ are given by conjugation with constant $\mathrm{SL}(2,\C)$ matrices.
If $\mathcal P$ is stable, all parabolic lines are pairwise distinct.  Fixing the gauge freedom is then equivalent to
normalizing the parabolic lines, i.e., \[\ell_1=0,\quad \ell_2=1,\quad\text{and}\quad \ell_3=\infty.\] Then, 
\[\ell_4=w\in \C P^1\setminus\{0,1,\infty\}\]
uniquely determines a stable parabolic structure (and for each $w\in \C P^1\setminus\{0,1,\infty\}$ there is 
a stable parabolic structure by Lemma \ref{lem:stableparalemma}). We call $w$ the parabolic modulus of the parabolic structure $\mathcal P$. By
identifying
\begin{equation}\label{eq:semisstableparaline}
\begin{split}
 w=0\quad &\Longleftrightarrow  \quad\ell_1=\ell_4 \text{ and } \ell_2=\ell_3\\
 w=1\quad &\Longleftrightarrow  \quad \ell_2=\ell_4 \text{ and } \ell_1=\ell_3\\
 w=\infty\quad &\Longleftrightarrow  \quad\ell_3=\ell_4 \text{ and } \ell_1=\ell_2\\
\end{split}
\end{equation}
we can fill in the missing three points of $\C P^1$.
We first claim  that it is not possible to find a holomorphic family $ w\mapsto \mathcal P(w)$ of polystable  parabolic structures  which represents the modulus $ w$ in any open neighbourhood of
 $ w=0,1,$ or $ w=\infty.$ Consider the case $w=0$ and parametrize 
(in  terms of a centered local holomorphic coordinate $\zeta$ )
\begin{equation}
\begin{split}
\ell_1(\zeta)&=\C(f_1(\zeta)\zeta,1)^T\quad \quad \ell_2(\zeta)=\C(1+f_2(\zeta)\zeta,1)^T\\
 \ell_3(\zeta)&=\C(1+f_3(\zeta)\zeta,1)^T\quad \quad \ell_4(\zeta)=\C(f_4(\zeta)\zeta,1)^T.
 \end{split}
 \end{equation}
 Assume that the lines only meet in first order in $\zeta$, which translates to the condition
  $f_1(0)\neq f_3(0)$ and $f_2(0)\neq f_4(0)$. 
Consider (for small $\zeta\neq0$) the Moebius transformations
 \[x\in\C P^1\longmapsto \frac{\zeta(x-\zeta f_1(\zeta))(f_2(\zeta)-f_3(\zeta)}{(1-x-\zeta f_3(\zeta))(-1+\zeta(f_1(\zeta)-f_2(\zeta)))}\in\C P^1\]
 which satisfy
 \begin{equation}
\begin{split}
  &\ell_1(\zeta)\longmapsto 0, \quad \ell_2(\zeta)\longmapsto 1,\quad  \ell_3(\zeta)\longmapsto \infty\\
  &w(\zeta)=\ell_4(\zeta)\mapsto -\zeta^2\frac{(f_2(\zeta)-f_3(\zeta))(f_1(\zeta)-f_4(\zeta))}{
  (-1-\zeta f_1(\zeta)-\zeta f_2(\zeta))(1+\zeta f_3(\zeta)-\zeta f_4(\zeta))}.
  \end{split}
 \end{equation}
 Hence, $w(\zeta)$ has at $\zeta=0$ a simple branch point as claimed. The other two cases $w=1$ and $w=\infty$ follow analogously.

It remains to exhibit the $\Z_2$-orbifold structure on the moduli space of polystable parabolic structures with parabolic weights $t$ at all 4 singular points.
As we have seen above, the moduli space $\mathcal M_{Bun}$ is given by the set 
of quadruples of  parabolic lines $\ell_j\subset\C$, $j=1,\dots,4$, subject to the condition that the 4 lines are pairwise distinct or  that they are pairwise the same
as in \eqref{eq:semisstableparaline},
up to the action of $\mathrm{SL}(2,\C)$.
Consider the map
\begin{equation}\label{eq:p1Mpbun}\C P^1\to \mathcal M_{Bun}\cong \C P^1;\; [u,v]\mapsto (\ell_1=[u,v],\,\ell_2=[v,-u],\,\ell_3=[-u,v],\,\ell_4=[v,u])/\sim\end{equation}
where $\sim$ denotes the equivalence relation induced by the $\mathrm{SL}(2,\C)$-action. This map is well-defined, as the 4 lines are either pairwise distinct ( $[u,v]\notin\{[\pm 1,1],[\pm i,1],[0,1],[1,0]\}$), 
or  $\ell_1=\ell_2$ and  $\ell_3=\ell_4$ (if $[u,v]=[\pm i,1]$), or $\ell_1=\ell_3$ and  $\ell_2=\ell_4$ (if $[u,v]\in\{[0,1],[1,0]\}$), or $\ell_1=\ell_4$ and  $\ell_2=\ell_3$ (if $[u,v]=[\pm 1,1]$). Consider the $\Z_2\times \Z_2$ action on $\C P^1$ generated by $[u,v]\mapsto [-u,v]$ and
$[u,v]\mapsto [v,u].$ Acting with $\begin{smatrix} i&0\\0&-i\end{smatrix}$ respectively with $\begin{smatrix} 0&i\\ i&0\end{smatrix}$, we see that the map
\eqref{eq:p1Mpbun} is invariant under the $\Z_2\times \Z_2$ action. Furthermore, this map is rational of degree 4, and the only fix points (of the 3 non-trivial elements of order 2) are 
$[\pm 1,1],[\pm i,1],$ or $[0,1]$ and $[1,0],$ respectively.
This proves the proposition.
\end{proof}
As we have seen in Proposition \ref{pro:paramodul} the  $\C P^1$ with  three orbifold singularities $0,1,\infty$ is given by the quotient of $\C P^1$ by the $\Z_2\times\Z_2$ action generated by $\zeta\mapsto -\zeta$ and $\zeta\mapsto \tfrac{1}{\zeta}.$
We will encounter this $\Z_2\times\Z_2$ action again in our symmetric setup, see
Proposition \ref{pro:classparahiggs} below.

\subsubsection{Strongly parabolic Higgs fields}
Let $\nabla$ and $\widetilde\nabla$ be two logarithmic connections with the same induced parabolic structure $\mathcal P$
on the holomorphic bundle $\mathcal V$. The difference $\widetilde\nabla-\nabla$ is a $\mathfrak{sl}(\mathcal V)$-valued meromorphic $1$-form. 
Since the eigenvalues of the residues of $\widetilde\nabla$ and $\nabla$ are the same, and the eigenlines $\ell_j$ for the
positive eigenvalues coincide, the residues 
\[R_j:=\text{res}_{p_j}(\widetilde\nabla-\nabla)=\text{res}_{p_j}(\widetilde\nabla)-\text{res}_{p_j}(\nabla)\] must be nilpotent with $\ell_j\subset \ker R_j.$ This motivates the following definition:
\begin{definition}
Let $\mathcal P$ be a parabolic structure on the holomorphic bundle $\mathcal V$ with singular divisor $D=p_1+\dots+ p_n.$ A strongly parabolic Higgs field is a meromorphic 1-form 
\[\Psi\in H^0(\Sigma, K_\Sigma\mathfrak{sl}(\mathcal V,\C)\otimes \mathcal O_\Sigma(D))\]
with $(\text{res}_{p_j}\Psi)\ell_j=0$ for all $j.$
A strongly parabolic Higgs pair $(\mathcal P,\Psi)$ is called stable if every $\Psi$-invariant holomorphic line subbundle of $\mathcal V$ has negative
degree, and strictly polystable if $\mathcal V$ is the direct sum of two $\Psi$-invariant line bundles of parabolic degree 0.
\end{definition}

We restrict ourselves in the following to the case of the $4$-punctured sphere with parabolic weight $t \in (0, \tfrac{1}{4})$ at each puncture. 
Note that, similarly to the proof of Lemma \ref{lemmaOk}, it can be shown that there is no stable parabolic Higgs pair on $\mathcal O(k)\oplus\mathcal O(-k)$ with $k\geq2$ over $\C P^1$ and only four singular points (together with our restrictions for the range of the weights). As before we treat the cases of $\mathcal V= \mathcal O(-1)\oplus\mathcal O(1)$ and $\mathcal V= \mathcal O\oplus\mathcal O$ separately.

\begin{lemma}\label{lem:unstablepara}
Let $(\mathcal P,\Psi)$ be a stable strongly parabolic Higgs pair with underlying holomorphic bundle $\mathcal V=\mathcal O(-1)\oplus\mathcal O(1).$
Then, up to isomorphisms of $\mathcal P$, all four parabolic lines are contained in $\mathcal O(-1),$ and the Higgs field $\Psi$ takes the form
\[\Psi=\begin{pmatrix}0&\alpha\\ \gamma &0\end{pmatrix}\]
where $\gamma\in H^0(\C P^1,\text{Hom}(\mathcal O(-1),K_{\C P^1}\mathcal O(1)))\cong \C$ is a constant and
$\alpha$ 
is a non-zero meromorphic section of
$\text{Hom}(\mathcal O(1),K_{\C P^1}\mathcal O(-1))$ 
with 4 simple poles at $p_1,\dots,p_4$ which is unique up to a scale that can be fixed by an isomorphism of $\mathcal P$.
\end{lemma}
\begin{proof}
We can assume without loss of generality that $\ell_1,\ell_2$ and $\ell_3$ are contained in $\mathcal O(-1)\subset\mathcal V,$ i.e., $\mathcal O(-1)|_{p_j} = \ell_j,$ for $j = 1, 2, 3.$
Write
\[\Psi=\begin{pmatrix}\beta&\alpha\\\gamma&-\beta\end{pmatrix}\]
with respect to the decomposition $\mathcal V=\mathcal O(-1)\oplus\mathcal O(1).$ Then, $\beta$ is a meromorphic 1-form on $\C P^1$ with at most a simple  pole at $p_4,$ thus $\beta$ must vanish. If $\alpha=0$ then $\mathcal O(1)$ would be an invariant subbundle of positive parabolic degree. Hence,
$\alpha\neq0$ by stability. Since $\text{Hom}(\mathcal O(1),K_{\C P^1}\mathcal O(-1)) \cong \mathcal O(-4)$ is of degree $-4$ and $\alpha\neq0$ can have at most simple poles, $\alpha$ must have a simple pole at each $p_j.$ 
Moreover, this implies $\alpha$ is unique up to scale, which can be fixed by a constant diagonal gauge. Furthermore, we obtain that
$\gamma$ cannot have a pole at $p_1,\dots,p_4$, as the residues of $\Psi$ must be nilpotent, thus $\gamma$ must be a constant which also implies that $\mathcal O(-1)|_{p_4} = \ell_4.$ 
\end{proof}
\begin{lemma}\label{lem:stabhiggspar}
Let $(\mathcal P,\Psi)$ be a stable Higgs pair with strongly parabolic Higgs field $\Psi$ and underlying holomorphic structure $\mathcal O\oplus\mathcal O.$ Then,  $\mathcal P$ is polystable. Conversely, for every polystable $\mathcal P$ on $\mathcal O\oplus\mathcal O$ there is a 
1-dimensional vector space (unique up to parabolic isomorphisms) of strongly parabolic Higgs fields $\Psi$  such that the Higgs pair $(\mathcal P, \Psi)$ is stable if $\Psi \neq 0$.  
\end{lemma}
\begin{proof}
We first show that the underlying parabolic structure $\mathcal P$ for a stable strongly parabolic Higgs pair on $\mathcal O\oplus\mathcal O$, is polystable. Assume $\mathcal P$ is not polystable, then there is a unique degree zero line bundle $L$ with non-negative parabolic degree. Thus, $L$ must contain at least two parabolic lines, i.e., $L|_{p_j} = \ell_j$ for two different $ j \in \{1, ..., 4\}$. 
A strongly parabolic Higgs field $\Psi\in H^0(K_{\C P^1} \mathfrak{sl}(2, \C))$ is given by $$\Psi = \sum_{j=1}^4R_j \tfrac{dz}{z-p_j},$$ with $R_j \in \mathfrak{sl}(2, \C)$ and $\sum_{j=1}^4R_j=0.$ Thus if $L$ contains three or four parabolic lines, in which case $L$ has positive parabolic degree, it must be invariant under $\Psi$ in contradiction to the Higgs pair being stable. If $L$ contains exactly two parabolic lines, then it is of parabolic degree zero, and we can find a complementary constant line bundle $\widetilde L$, which contains either one or two of the remaining parabolic lines. Since the quasiparabolic lines lie in the kernel of $R_j, $ the matrices $R_j$ with respect to the decomposition $\mathcal O\oplus\mathcal O=L\oplus \widetilde L$ are either upper or lower triangular with vanishing diagonals whenever $L$ or $\wt L$ contains the quasiparabolic line $\ell_j$. Thus by residue theorem, as the sum of all residues $R_j$ must be zero, we obtain that $\widetilde L$ must contain two parabolic lines and thus has parabolic degree 0. But then the parabolic bundle is the direct sum of two line bundles of parabolic degree zero and hence $\mathcal P$ must be polystable, which gives a contradiction.

Next, we show that dimension of the space of (stable) strongly parabolic Higgs fields for a given stable parabolic structure is at most 1-dimensional (up to isomorphism). If $\mathcal P$ is stable, i.e., the four quasiparabolic eigenlines are distinct, a short computation shows that there exists no non-zero strongly parabolic Higgs field of $\mathcal P$ with vanishing residue at one of the singular points by the residue theorem. This shows that the space of strongly parabolic Higgs fields is at most 1-dimensional, as the linear combination of two (non-zero) strongly parabolic Higgs fields such that the residue $R_j$ at one singular point vanishes must be identically zero.

The space of stable strongly parabolic Higgs fields
is at least 1-dimensional,
up to conjugation, since every stable parabolic structure $\mathcal P$ with  underlying holomorphic bundle
$\mathcal O\oplus\mathcal O$ and parabolic weights $t\in  (0,\tfrac{1}{4})$ is determined by the quasiparabolic lines
\begin{equation}\label{eq:quasbef}\ell_1=(u,v)^T\C,\quad \ell_2=(v,-u)^T\C,\quad \ell_3=(u,-v)^T\C,\quad \ell_4=(v,u)^T\C
\end{equation}
for suitable $u,v\in\C^2\setminus\{0\}$. 
Then, it can be easily verified that
\begin{equation}\label{eq:psier}
\Psi=
\begin{pmatrix} uv& -u^2\\v^2&-uv\end{pmatrix}\frac{dz}{z-p_1}
+\begin{pmatrix} -uv& -v^2\\u^2&uv\end{pmatrix}\frac{dz}{z-p_2}
+\begin{pmatrix} uv& u^2\\-v^2&-uv\end{pmatrix}\frac{dz}{z-p_3}
+\begin{pmatrix} -uv& v^2\\-u^2&uv\end{pmatrix}\frac{dz}{z-p_4}
\end{equation}
is a non-zero strongly parabolic Higgs field for $\mathcal P$. Since  $\mathcal P$ is already stable, the Higgs pair $(\mathcal P,\Psi)$ is stable as well.

If  $\mathcal P$ is the direct sum of two line bundles of parabolic degree 0, then a stable strongly parabolic Higgs field must be off-diagonal (with respect to the decomposition of the rank two bundle into the two line bundles), with the off-diagonal entries being  two non-zero meromorphic 1-forms with simple poles at the two of the four singular points, where the quasiparabolic lines coincide with the respective holomorphic line bundle. This gives a two-dimensional space (with two complex lines removed)
of possible strongly parabolic Higgs fields. Because for non-vanishing off-diagonals the determinant of the Higgs field is non-vanishing with simple poles, these Higgs fields are stable.
By a diagonal gauge, we can fix the ratio of  the meromorphic 1-forms, and we obtain a complex line (without the origin, where the Higgs field vanishes) of stable parabolic Higgs fields. 

\end{proof}

\begin{lemma}\label{lem:next}
Let $(u,v)\in\C^2\setminus\{0\}$  and consider the strongly parabolic Higgs field $\Psi$ as in \eqref{eq:psier}. Then the underlying
parabolic structures are always polystable and moreover, they are stable except for 
$$uv=0, \quad u^2=v^2 \quad \text{ or } \quad u^2=-v^2.$$
\end{lemma}
\begin{proof} 
Since $4t<1$ every holomorphic subbundle of the holomorphically trivial rank 2 bundle with negative degree also has negative parabolic degree.
Let $L$ be a holomorphic line subbundle of degree $0$, i.e., a constant line in $\C^2.$ Thus the parabolic degree of $L$ can be maximized by choosing $L = \ell_j$ to be one of the quasiparabolic lines. 
These lines $\ell_1, ..., \ell_4$ are given by \eqref{eq:quasbef}. If all four lines are distinct, then any such $L$ has negative parabolic degree $-2t$ and the parabolic structure is stable. From the explicit formulas in \eqref{eq:quasbef} it follows that 
at most two of the lines $\ell_1, ..., \ell_4$ can coincide. And two lines coincide if and only if  $uv=0$, $u^2=v^2$ or $u^2=-v^2.$ In this case we have that there are two complementary lines $L$ and $\wt L$ with parabolic degree zero and $\mathcal P$ is polystable but not stable.
\end{proof}

The parabolic Higgs pair, consisting of the parabolic structure $\mathcal P$ with trivial underlying holomorphic structure and the strongly parabolic Higgs field $\Psi$ as in \eqref{eq:psier}, is uniquely determined by the nilpotent $\mathfrak{sl}(2,\C)$-matrix  \begin{equation}\label{eq:res1higgs}R_1=\text{res}_{p_1}\Psi=\begin{pmatrix} uv& -u^2\\v^2&-uv\end{pmatrix}.\end{equation}
In fact, we have
\begin{equation}\label{eq:psiAs}\Psi=R_1\frac{dz}{z-p_1}+R_2\frac{dz}{z-p_2}+R_3\frac{dz}{z-p_3}+R_4\frac{dz}{z-p_4}\end{equation}
with
\begin{equation}\label{eq:psiAs2}
R_2=(CD)^{-1}R_1(CD),\quad R_3=D^{-1}R_1 D,\quad R_4=C^{-1}R_1 C\end{equation}
for anti-commuting
\begin{equation}\label{eq:CD}
 C=\matrix{0&i\\i&0}\quad \text{and}\quad D=\matrix{i&0\\0&-i}\,.
\end{equation}

Note that the conjugation of $R_1$ by $ C$, $ D$ or $ CD$ leads to a conjugation of the corresponding Higgs field $\Psi$,
which is not true in general for arbitrary $g\in\mathrm{SL}(2,\C)$, as $g$ does neither commute nor anti-commute with $C$ and $D$. In fact, a direct computation gives:

\begin{lemma}\label{41cov}
Let $0\neq R_1\in\mathfrak{sl}(2,\C)$ be nilpotent and $g\in\mathrm{SL}(2,\C).$ Then the parabolic Higgs fields $\Psi$ and $\widetilde \Psi$ corresponding to
$R_1$ and $\widetilde R_1=g^{-1}R_1g$  via \eqref{eq:psiAs} and \eqref{eq:psiAs2} are gauge equivalent if and only if 
\[g\in\,<C,D>,\]
where $<C,D>$ denote the gauge group generated by conjugation with $C$ and $D.$
Therefore, the group of (non-trivial) transformations preserving the gauge orbit is isomorphic to $\Z_2\times\Z_2.$
\end{lemma}
\begin{remark}
Note that a non-zero stable strongly parabolic Higgs field on $\mathcal O\oplus\mathcal O$ uniquely determines the quasiparabolic lines. Hence if the parabolic weights are given, the Higgs field induces a parabolic structure. We call a parabolic Higgs field $\Psi$ of the form \eqref{eq:psier}  {\em symmetric}. 
\end{remark}
As a corollary of the above observations, we have the following classification.
\begin{proposition}\label{pro:classparahiggs}
For $t \in (0, \tfrac{1}{4})$ fixed, the moduli space of polystable strongly  parabolic Higgs fields over $\C P^1$ with 4 singular points 
contains an open dense subset $\mathcal U$ which is isomorphic to the holomorphic cotangent bundle 
$T^*\C P^1$ modulo the $\Z_2\times\Z_2$ action induced by conjugation by $C$ and $D.$
\end{proposition}
\begin{proof}
We define $\mathcal U$ to be the (open) subset of the moduli space of polystable strongly parabolic Higgs pairs over $\C P^1$ with 4 singular points on $\mathcal V= \mathcal O \oplus \mathcal O.$

Whenever the Higgs field is non-zero, the parabolic structure is determined by the symmetric Higgs field $\Psi.$ This Higgs field in turn is determined by the non-zero nilpotent residue
$R_1\in\mathfrak{sl}(2,\C).$ The matrix $R_1$ is determined  unique up to sign by 
$$(u,v)\in\C^2\setminus\{0\}\mapsto\begin{pmatrix} uv &-u^2\\v^2 &-uv\end{pmatrix}.$$ Recall that the kernel of $R_1$ determines the induced quasiparabolic lines via \eqref{eq:psiAs2} or  \eqref{eq:quasbef}. 
Moreover, by Lemma \ref{41cov},  the gauge class of a strongly parabolic Higgs field (or the induced parabolic structure) determines $R_1$ uniquely up to conjugation by the gauge group $\Z_2\times\Z_2$ generated by $C$ and $D.$

When $(u, v) \rightarrow 0$ the induced parabolic structure is determined by the ratio $[u:v]$ in the limit. Thus the moduli space of strongly parabolic structures $\mathcal U$ on $\mathcal O \oplus \mathcal O$ is given by the blow up of $\C^2/\Z_2$ at $(u,v) = 0$ up to taking the quotient by the $\Z_2\times\Z_2$ action. The blow-up of $\C^2/\Z_2$ can then be naturally identified with the cotangent bundle of $\C P^1.$
\end{proof}

The complement of $\mathcal U$ is the complex line of parabolic Higgs fields with underlying holomorphic bundle $\mathcal V=\mathcal O(1)\oplus\mathcal O(-1).$ It is possible to explicitly glue this line to $\mathcal U$ and turn
the moduli space into a complex orbifold, see
for example \cite{Men}. We do not give more details of this, since our loop group methods only work on $\mathcal U$ as of now.

\subsection{Non-abelian Hodge correspondence in the strongly parabolic case}
In this section we recall Simpson's non-abelian Hodge correspondence \cite{Si1} in the case of rank two strongly parabolic Higgs bundles and non-trivial weights,
see also \cite{FMSW} or \cite{KiWi}.

The space of hermitian metrics with determinant one is naturally diffeomorphic to the hyperbolic 3-space.
Thus we can measure the distance between two metrics as the distance of the two corresponding points in the hyperbolic space using the metric $d$.  
Let $\mathcal P$ be a parabolic structure on a holomorphic rank two bundle $\mathcal V$ with trivial determinant bundle, $z$ be a local holomorphic coordinate centered at a singular 
point $p_j\in\Sigma$, and let $\ell_j\subset\mathcal V_{p_j}$ be the parabolic line and $\alpha_j \in (0, \tfrac{1}{2})$ be the corresponding parabolic weight. Choose   a local holomorphic frame $(s_1,s_2)$ of $\mathcal V$ on 
$U_j\ni p_j$ such that
\[s_1\wedge s_2=1\in H^0(U_j,\Lambda^2\mathcal V)=H^0(U_j,\mathcal O)\]
and
\[s_1(p_j)\in\ell_j.\]
Consider the hermitian metric $h_j$ given by
\[h_j=\begin{pmatrix} (z\bar z)^{\alpha_j}&0\\0&(z\bar z)^{-\alpha_j}\end{pmatrix}\]
with respect to the holomorphic frame $(s_1,s_2)$.
We call $h_j$ a model metric with respect to the parabolic structure at $p_j$. A hermitian metric $h$ on $\mathcal V$ over ${\Sigma\setminus\{p_1,\dots,p_n\}}$ is called {\em tame} at $p_j$ with respect to
the parabolic structure if and only if
\[d(h,h_j) \quad \text{is bounded on } U_j\setminus\{p_j\}.\]
Since any two model metrics have finite distance to each other the notion of tameness is independent of the choice of holomorphic frame with the above properties. 
\begin{definition}
Let $\mathcal P$ be a parabolic structure on a holomorphic vector bundle $\mathcal V$ over $\Sigma$
with singular divisor $p_1+\dots+p_n.$
A hermitian metric on a holomorphic bundle $\mathcal V$ over ${\Sigma\setminus\{p_1,\dots,p_n\}}$ is called {\em tame} with respect to $\mathcal P$
if it is tame at all of its singular points $p_j.$
\end{definition}
Given a tame hermitian metric $h$ and
a holomorphic structure $\bar\partial$ we denote its Chern connection on $\Sigma\setminus\{p_1,\dots,p_n\}$ by $D^h=D^{\bar\partial,h}.$
Similarly, for a strongly parabolic Higgs field $\Psi$ we denote its hermitian conjugate by $\Psi^*=\Psi^{*,h}\in\Gamma(\Sigma\setminus\{p_1,\dots,p_n\},\overline{K}\mathfrak{sl}(\mathcal V)).$
\begin{definition}
Let $(\mathcal P,\Psi)$ be a polystable strongly parabolic Higgs pair. A metric $h$ 
is called tame harmonic with respect to the Higgs pair if the metric is tame with respect to $\mathcal P$ and if the
connection
\[\nabla:=D^h+\Psi+\Psi^{*,h}\]
is flat.
\end{definition}
A tame harmonic metric gives a solution of Hitchin's self-duality equations over $\Sigma\setminus\{p_1,\dots,p_n\}$ with appropriate growth conditions at the punctures.
The notation of a (tame) harmonic metric is motivated by the fact that there is a direct link to harmonic maps into hyperbolic 3-space: let $U\subset
\Sigma\setminus\{p_1,\dots,p_n\}$  be open and
$(s_1,s_2)\colon U\to\mathrm{SL}(2,\C)$ be a parallel frame with respect to the flat connection $\nabla.$ Then, the
map \[\begin{pmatrix}h(s_i,s_j)\end{pmatrix}_{i,j}\colon U\to\{A\in\mathrm{SL}(2,\C)\mid \bar A^T=A \text{ and } A>0\}\cong \mathbb H^3\]
is a harmonic map to hyperbolic 3-space $\mathbb H^3.$ Globally, this gives an equivariant  harmonic map (with respect to the monodromy representation of the flat connection $\nabla$) from the universal covering
of $\Sigma\setminus\{p_1,\dots,p_n\}$ into the hyperbolic 3-space. We can now state the non-abelian Hodge correspondence for the parabolic case (pNAH).
\begin{theorem}\cite{Si2}\label{NAHSi}
For every polystable strongly parabolic Higgs pairs with singular divisor $D= p_1+\dots+p_n$ and parabolic weights
$\alpha=(\alpha_1,\dots,\alpha_n), \alpha_i\in(0,\tfrac{1}{2})$ there is a  tame harmonic metric $h$. The tame harmonic metric is unique if
the Higgs pair is stable.
This induces a bijection between the space of polystable strongly parabolic Higgs pairs with singular divisor $D$ and parabolic weights
$\alpha$ modulo isomorphisms and the space of totally reducible
flat logarithmic connections with singularities at $p_j$ of conjugacy class \eqref{eq:locmon} modulo gauge transformations 
by associating $$(\mathcal V,\mathcal P,\Psi)\mapsto \nabla:=D^h+\Psi+\Psi^{*,h}.$$
\end{theorem}
The parabolic structure induced by the logarithmic connection $\nabla=D^h+\Psi+\Psi^{*,h}$ is not 
isomorphic to the parabolic structure underlying the stable strongly parabolic Higgs field in general, but the parabolic weights are the same. Just as in the compact case, we obtain from a tame harmonic metric $h$
for a strongly parabolic Higgs pair $(\mathcal V,\mathcal P,\Psi)$ an associated $\C^*$ family of flat  connections
\begin{equation}\label{eq:asslambdafami}
\lambda\in\C^*\mapsto \nabla^\lambda:=D^h+\lambda^{-1}\Psi+\lambda\Psi^{*,h}.\end{equation}
It follows from Simpson's construction \cite{Si2} (see also \cite{Si21}) that for all $\lambda\in\C^*$, the connections $\nabla^\lambda$ extends naturally to logarithmic connections on $\Sigma$ with the same 
 parabolic weights as the initial parabolic structure $\mathcal P.$ In particular, the parabolic weights of $\nabla^\lambda$ are independent of $\lambda\in\C^*$. 
\subsubsection{Rational weights}\label{ssec:rat}
In the case of rational parabolic weights \[\alpha_1=\tfrac{l_1}{k_1},\dots,\alpha_n=\tfrac{l_n}{k_n}\in\mathbb Q\cap(0,\tfrac{1}{2}),\]
 Theorem \ref{NAHSi} is directly linked to
the non-abelian Hodge correspondence on compact Riemann surfaces through coverings of $\Sigma$, see \cite[Section 3 and Section 5]{NaSt} for details.
The underlying geometric idea is that the equivariant harmonic map to hyperbolic 3-space obtained from the tame 
harmonic metric associated to a polystable strongly parabolic Higgs pair is actually the equivariant harmonic map
associated to a polystable Higgs pair on some compact surface, see Figure \ref{fig:1}.
We will explain the details here only in the case of $\Sigma$ being the $4$-punctured sphere and parabolic weights 
$t=\alpha_j=\tfrac{l}{k}$, $l,k\in\N,$ $2l<k$, but these constructions also work in the general case of rational weights, see \cite{NaSt}.

\begin{figure}
\centering
\includegraphics[width=1\textwidth]{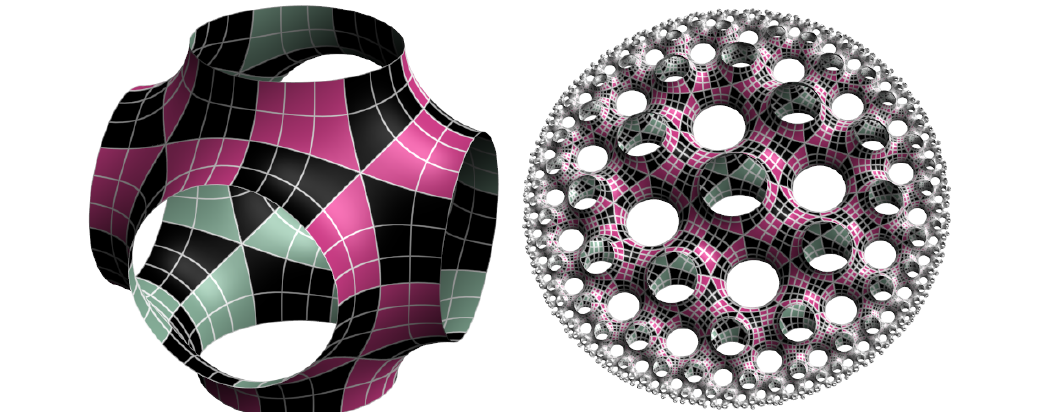}
\caption{
\footnotesize{Two views on an equivariant minimal surface in $\mathbb H^3$ obtained from a strongly parabolic nilpotent Higgs field
on $\C P^1$ with four singular points and parabolic weight $t=\alpha_j=\tfrac{1}{6}$. Image by Nick Schmitt.
}}
\label{fig:1}
\end{figure}

Consider the $k$-fold covering of $\C P^1$ defined by the equation
\begin{equation}\label{eq:sigma_k}\Sigma_k \colon y^k = \frac{(z-p_1)(z-p_3)}{(z- p_2)(z-p_4)},\end{equation}
which is totally branched over $p_1,\dots,p_4$,  and let $\pi=z\colon\Sigma_k\to\C P^1$ denote the covering map. Note that $\Sigma_k$ is a compact Riemann surface of genus $k-1$ and there is an action of
the finite abelian group $\Z_k$ on $\Sigma_k$ such that
 $\Sigma_k/\Z_k=\C P^1.$ 
 Let $\mathcal V\to\C P^1$ be the underlying holomorphic bundle of the parabolic structure $\mathcal P.$ By abuse of notation we also denote  its pull-back bundle by $\mathcal V=\pi^*\mathcal V\to\Sigma_k$. We construct a
 new holomorphic bundle $\wh{\mathcal V}\to\Sigma_k$ as follows. Let $x\colon U_j\subset\Sigma_k\to\C$ be a local holomorphic coordinate
 centered at $\pi^{-1}(p_j)$ satisfying $\sigma^*x=e^{\tfrac{2\pi i}{k}}x$ for a generator $\sigma$ of the $\Z_k$-action.
 Let $g_j\colon\mathcal V_{\mid U_j}\to \C^2$ be a holomorphic $\mathrm{SL}(2,\C)$ trivialisation of $\mathcal V$ over $U$
 such that $g_j(\ell_j)=\C(1,0)$, where $\ell_j$ is the parabolic line  at $\pi^{-1}(p_j)$. Then, $\wh{\mathcal V}$
 is the holomorphic bundle which is $\mathcal V$ over $\Sigma_k\setminus\pi^{-1}\{p_1,\dots,p_4\}$ and over $U_j$
the $\mathrm{SL}(2,\C)$ trivialisation $\wh{g}_j\colon \wh {\mathcal V} |_{U_j} \rightarrow \C^2$ is defined by
 \begin{equation}\label{def:twistbundle}(g_j\circ \wh{g}_j^{-1})|_{U_j\setminus \pi^{-1}(p_j)}=\begin{pmatrix} x^{-l}&0\\0&x^{l}\end{pmatrix}.\end{equation}
 Note that $\wh{\mathcal V}$ has trivial determinant bundle.
Moreover,  the pull-back  of the Higgs field $\pi^*\Psi$
is given by  \[g_j\circ\pi^*\Psi\circ g_j^{-1}=\begin{pmatrix} 0& a\\0&0\end{pmatrix} \frac{dx}{x}+A_1x^{k-1}dx+\text{ higher order terms in $x$}\]
 for some $a\in\C$ and $A_1,\dots\in\mathfrak{sl}(2,\C).$
Thus by conjugating with $g_j\circ \wh{g}_j^{-1}$ on $U_j\setminus \pi^{-1}(p_j)$ the Higgs field extends smoothly to $p_j$ and gives rise to $\wh\Psi\in H^0(\Sigma_k,K\mathfrak{sl}(\wh{\mathcal V})).$
By construction  $\wh{\mathcal V}$ has a natural non-trivial $\Z_k$-action
extending 
\[\wh\sigma\colon U_j \times \C^2\to U_j \times \C^2;\; (p,\begin{pmatrix} a\\b\end{pmatrix})\longmapsto (\sigma(p), \begin{pmatrix} e^{\tfrac{2\pi i l}{k}}a\\e^{\tfrac{-2\pi i l}{k}}b\end{pmatrix})\]
in the trivialisation $\wh g_j,$ and $\wh\Psi$ is equivariant:
\[\sigma^*\wh\Psi=\wh\sigma^{-1}\wh\Psi\wh\sigma.\] 
Furthermore, let $\nabla$ be a logarithmic connection with induced parabolic structure $\mathcal P$. Recall that
this implies that 
\[\text{res}_{p_j}\nabla_{\mid \ell_j}=\tfrac{l}{k}\Id_{\ell_j}.\]
A short computation similar to the case for a strongly parabolic Higgs field then shows that $\pi^*\nabla$ gives rise to an equivariant global holomorphic connection $\wh\nabla$
 on $\wh{\mathcal V}\longrightarrow \Sigma_k$
\[\sigma^*\wh\nabla=\wh \nabla.\wh\sigma.\]
We will refer to both operations $(\mathcal V,\mathcal P,\Psi)\longmapsto(\wh{\mathcal V},\wh\Psi)$ 
and $\nabla\longmapsto\wh\nabla$
as {\em twisted lifts}. Note that gauge equivalent strongly parabolic Higgs pairs (or respectively logarithmic connections)
have gauge equivalent twisted lifts. In other words, let  $\mathcal M(\Sigma_k)$ and
$\mathcal M(\C P^1)$ denote the moduli spaces of (parabolic) Higgs bundles or flat 
(logarithmic) connections on $\Sigma_k$ and the 4-punctured $\C P^1$, respectively, then
\begin{equation}\label{eq:pistartwistedliftmap}\wh{ }\; \colon\mathcal M(\C P^1)\longrightarrow\mathcal M(\Sigma_k)\end{equation}
is a well-defined and smooth map between these moduli spaces.
Every polystable strongly parabolic Higgs pair $(\mathcal P,\Psi)$ gives rise to a tame harmonic metric $h$, see Theorem \ref{NAHSi}. For rational weights this can be deduced directly from the fact that a stable and $\Z_k$-symmetric Higgs pair $(\wh{\mathcal V}, \wh \Psi)$ on $\Sigma_k$ gives rise to a harmonic metric $\wh h$ which is, by uniqueness of the solution, equivariant with respect to the $\Z_k$-action. Then twisting leads to an invariant harmonic metric $h$ on $\pi^* V$ descending to the quotient $4$-punctured sphere.
Using \eqref{def:twistbundle}, it can be shown that the metric $h$ on the 
quotient is tame.

In particular, we obtain a commutative non-abelian Hodge diagram 
 \begin{equation}\label{commNAH}
\begin{tikzcd}
\{(\mathcal P,\Psi)\mid \mathrm{polystable}\}/\sim \arrow[rr, "\mathrm{twisted\,lift}"] \arrow[dd, "\mathrm{pNAH}"] && \{(\wh{\mathcal V} ,\Psi)\mid \mathrm{polystable}\} /\sim\arrow[dd, "\mathrm{NAH}"]  \\
\\
\{\nabla\mid \mathrm{logarithmic\, and\, totally\, reducible}\}/\sim\arrow[rr, "\mathrm{twisted\,lift}"] &  &\{\wh\nabla\mid \mathrm{totally\, reducible}\}/\sim
\end{tikzcd}
\end{equation}
where $\sim$ denotes modulo isomorphism/gauge equivalence leading to
\begin{theorem}\label{thm:twistedlifttwistorlines}
Let $(\mathcal V,\mathcal P,\Psi)$ be a polystable strongly parabolic Higgs pair with associated family
of flat logarithmic connections $\nabla^\lambda = D^h + \lambda^{-1}\Psi + \lambda\Psi^*$. Then, the associated family $D^{\wh h} + \lambda^{-1}\wh\Psi + \lambda\wh \Psi^*$ of the twisted lift $(\wh{\mathcal V},\wh\Psi)$
is (gauge equivalent) to the twisted lift $\wh{\nabla^\lambda}$ of the family $\nabla^\lambda$.
\end{theorem}

\subsection{The parabolic Deligne-Hitchin moduli space}
Analogously  to the case of compact Riemann surfaces, Deligne-Hitchin moduli spaces can be defined in the strongly parabolic setup, see \cite{AlGo} and \cite{Si21} for more details. We  restrict ourselves again to the $\mathrm{SL}(2,\C)$ case with fixed local conjugacy classes.
\begin{definition}
Let $\Sigma$ be a compact Riemann surface and $p_1,\dots,p_n\in\Sigma$ be pairwise distinct points.
A  parabolic $\mathrm{SL}(2,\C)$ $\lambda$-connection with singular divisor $D=p_1+\dots+p_n$ is a quadruple 
$(\lambda,\mathcal V,\mathcal P,\mathcal D)$ consisting of a holomorphic $\mathrm{SL}(2,\C)$ bundle $\mathcal V$ over $\Sigma$ with a trace-free parabolic structure $\mathcal P$ and holomorphic trace-free $\lambda$-connection $\mathcal D$
on $\mathcal V\mid_{\Sigma\setminus\{p_1,\dots,p_n\}}$ such that
with respect to any local holomorphic frame of $\mathcal V$ around $p_j,$ the $\lambda$-connection 1-form of $\mathcal D$ is meromorphic with first order pole at $p_j$ only, and such that the quasiparabolic line $\ell_j$ at $p_j$ is an eigenline of the residue $\text{res}_{p_j}\mathcal D$
with eigenvalue $(\lambda\alpha_j)$, where $\alpha_j$ is the parabolic weight at $p_j.$ 
 \end{definition}
 For $\lambda\neq0$, $\tfrac{1}{\lambda}\mathcal D+\bar\partial^{\mathcal V}$ is a logarithmic connection with parabolic structure $\mathcal P$. For $\lambda=0,$ $\mathcal D$ is a strongly parabolic Higgs field for the parabolic structure $\mathcal P.$
 
 A parabolic $\lambda$-connection is stable (semi-stable) if every $\mathcal D$-invariant holomorphic line subbundle has negative (non-positive) parabolic degree, and unstable otherwise. As before, we also have the notion of polystability and for $\lambda\neq0,$ there are no unstable parabolic $\lambda$-connections.
 
 An isomorphism
between two parabolic $\lambda$-connections $(\lambda,\mathcal V_1,\mathcal P_1,\mathcal D_1)$ and $(\lambda,\mathcal V_2,\mathcal P_2,\mathcal D_2)$ is given by a $\mathrm{SL}(2,\C)$ gauge transformation $g$ between the two parabolic structures  $(\mathcal V_1,\mathcal P_1)$ and $(\mathcal V_2,\mathcal P_2)$ such that  $\mathcal D_2\circ g=g\circ \mathcal D_1.$  The only automorphisms of a stable parabolic $\lambda$-connection are $\pm\Id.$

In order to construct a moduli space, we fix the Riemann surface $\Sigma$, the divisor $D$ and the parabolic weights, but neither the holomorphic bundle nor the quasiparabolic lines. Then, the parabolic Hodge moduli space is the space of all polystable $\lambda$-connections with $\lambda \in \C$ modulo isomorphism. It is denoted by $\mathcal M_{Hod}^{D,\alpha_1,\dots,\alpha_n}(\Sigma)$, or by 
$\mathcal M_{Hod}^{par}(\Sigma)$ for short. It is a smooth complex manifold at its stable locus, and it  fibers over 
$\C.$

Let $\nabla$ be a logarithmic $\mathrm{SL}(2,\C)$ connection on $\Sigma$ with singular divisor $D$
and induced parabolic structure $\mathcal P.$
Let $\bar\Sigma$ be the complex conjugate Riemann surface, with divisor $\bar D=p_1+\dots+p_n.$ Then $\nabla$
is a flat connection on $\bar\Sigma\setminus\{p_1,\dots,p_n\}$. By Deligne extension, there is a logarithmic $\mathrm{SL}(2,\C)$ connection $\widetilde\nabla$ on $\bar\Sigma$ with the same parabolic weights as $\nabla,$ which is gauge equivalent to $\nabla$ on the smooth surface
$\Sigma\setminus\{p_1,\dots,p_n\}=\bar\Sigma\setminus\{p_1,\dots,p_n\}.$
More explicitly, choose a local holomorphic coordinate around $p_j \in \Sigma$ and a trivialization of $\mathcal V$
such that the connection $\nabla$ is given by
 $$\nabla= d+ \begin{pmatrix} \alpha_j&0\\0&-\alpha_j\end{pmatrix}\tfrac{dz}{z},$$  
and choose a gauge in this trivialization given by
 $g_j=\text{diag}((z\bar z)^{-\alpha_j},(z\bar z)^{\alpha_j})$, then
 the logarithmic connection $\wt \nabla$ on $\bar\Sigma$ is given by
 \[\wt \nabla = \nabla.g_j = d+ \begin{pmatrix} -\alpha_j&0\\0&\alpha_j\end{pmatrix}\tfrac{d\bar z}{\bar z}\]
around $\bar p_j \in \bar\Sigma$ in the same trivialization. 
  It is easy to check that this gives a well-defined map
 from the moduli space of logarithmic connections $\mathcal M_{dR}^{D,\alpha_1,\dots,\alpha_n}(\Sigma)$  on $\Sigma$ to the moduli space of logarithmic connections $\mathcal M_{dR}^{\bar D,\alpha_1,\dots,\alpha_n}(\bar\Sigma) $  on 
 $\bar\Sigma.$ Furthermore, this map can be extended to $\lambda \in \C_*$ to identify the Hodge moduli spaces of $\Sigma$ and $\bar \Sigma$ restricted to $ \C^*$ 

\[\mathcal G\colon\mathcal M_{Hod}^{par}(\Sigma)\mid_{\C^*}=\C^*\times \mathcal M_{dR}^{D,\alpha_1,\dots,\alpha_n}(\Sigma) \,
 \longrightarrow \,\C^*\times \mathcal M_{dR}^{\bar D,\alpha_1,\dots,\alpha_n}(\bar\Sigma)=\mathcal M_{Hod}^{par}(\bar\Sigma)\mid_{\C^*}.\]
 
As in the smooth case, $\mathcal G$ maps the stable (and smooth) locus to the stable (and smooth) locus. In this way, we obtain the
 parabolic Deligne-Hitchin moduli space
\[\mathcal M_{DH}^{par}(\Sigma,D,\alpha_1,\dots,\alpha_n)=\mathcal M_{Hod}^{par}(\Sigma)\cup_\mathcal G \mathcal M_{Hod}^{par}(\bar\Sigma)\longrightarrow\C P^1.\]
If $\alpha_1=\dots=\alpha_n=t$ we will denote the space by $\mathcal M_{DH}^{par}(\Sigma,D,\alpha_1,\dots,\alpha_n)=\mathcal M_{DH}^{par}(\Sigma,D,t).$ The parabolic Deligne-Hitchin moduli space is naturally equipped with a holomorphic $\C^*$ action covering 
 \[\mu\in\C^*\mapsto (\lambda\in\C P^1\mapsto \mu\lambda\in\C P^1).\]
 As before, we denote the automorphism obtained by multiplying with $\mu=-1$ by $N.$ Furthermore, 
 there exists a natural real structure $C$.  In order to define $C$ we recall that the underlying complex rank 2 vector bundle has a fixed $C^\infty$ trivialization $V=\Sigma\times\C^2$ with complex conjugation $c\colon V\to \bar V\cong V.$
For a linear differential operator $A$ let $\bar A:=c^{-1}\circ A\circ c.$ Thus if $A$ is complex linear,
 $\bar A$ is also complex linear but in general not gauge equivalent to $A$. If $\bar\partial$ is a holomorphic structure on $V\to\Sigma$, then $\overline{\bar\partial}$ is a holomorphic structure on $V\to\bar\Sigma.$  
  \begin{lemma}\label{lem:cmap}
 If $(\lambda,\bar\partial, \mathcal P,\mathcal D)$ is a parabolic $\lambda$-connection on $\Sigma$ 
 with singular divisor $D=p_1+\dots+p_n$, then
 \[(\bar\lambda,\overline{\bar\partial}, \bar{\mathcal P}, \overline{\mathcal D})\]
 is a parabolic $\lambda$-connection on $\bar\Sigma$, where $ \bar{\mathcal P}$ is given by the quasiparabolic lines 
 $\bar\ell_j=c(\ell_j)$ and the
 parabolic weights $\bar\alpha_j=\alpha_j.$
 
 If two parabolic $\lambda$-connections are isomorphic via a gauge $g$ then their complex conjugate 
 $\lambda$-connections on $\bar\Sigma$ are isomorphic via the gauge $c^{-1}\circ g\circ c.$ In particular, there
 exists an anti-holomorphic involution $C$ of $\mathcal M_{DH}^{par}$ covering $\lambda\mapsto\bar\lambda^{-1}$ and commuting
  with $N.$
 \end{lemma}
 \begin{example}
 Let $(\lambda,\bar\partial, \mathcal P,\mathcal D)$ be a parabolic $\lambda$-connection on $\Sigma$  with $\lambda\in S^1$ such that $\bar\partial+\tfrac{1}{\lambda}\mathcal D$ is unitary. Then, $[\lambda,\bar\partial, \mathcal P,\mathcal D]\in \mathcal M_{DH}^{par}$
 is a fixed point of $C.$
 \end{example}
 
For a given polystable strongly parabolic Higgs pair $p=(\mathcal P,\Psi)$ together with its tame harmonic metric $h$,
the associated family of flat connections \eqref{eq:asslambdafami} gives rise to a holomorphic section $s_p$ of
$\mathcal M_{DH}^{par}$ called twistor line, see
 \cite[Theorem 2.7]{Si21} for details. With respect to the real structure $\mathcal T=CN=NC$ this section is real, i.e., $\mathcal T(s_p(-\bar\lambda^{-1}))=s_p(\lambda)$
 for all $\lambda\in\C^*$.

We now discuss the necessary and sufficient conditions of a real holomorphic section $s$ of  $\mathcal M_{DH}^{par}$ to be a twistor line. 
This is motivated  by (and analogous to) the compact case described in Section \ref{pre}.

\begin{lemma}\label{lem:sufsec}
Let $\epsilon>0$.
Let $\lambda\in \D_{1+\epsilon}^*\mapsto \nabla^\lambda$ be
a holomorphic family of irreducible  logarithmic $\mathrm{SL}(2,\C)$ connections with first order pole at $\lambda = 0$ and $\Psi= $res$_{\lambda = 0} \nabla^\lambda \in \Gamma(K_\Sigma \mathfrak{sl} (\mathcal V(\lambda)))$ on a family of holomorphic bundles $\mathcal V(\lambda)\to\Sigma$ with $\lambda$-independent parabolic divisor $p_1+\dots+p_n$
and parabolic weights $\alpha_j\in (0,\tfrac{1}{2})$. 

Assume that there exists $g=g(\lambda)\colon\Sigma\setminus\{p_1,\dots,p_n\}\to \Lambda\mathrm{SL}(2,\C)$
 such that
 \begin{equation}\label{eq:realityconditionpara}\overline{\nabla^{-\bar\lambda^{-1}}}=\nabla^\lambda.g(\lambda)\,.\end{equation}
 Then  $\nabla^\lambda$ gives rise to a
real holomorphic section $s$ of the parabolic Deligne-Hitchin moduli space.
\end{lemma}
\begin{proof}
Note that $\lambda\mapsto\nabla^\lambda$ gives rise to a local section of $\mathcal M_{DH}^{par}$ over 
$\D_{1+\epsilon}^*$ 
via
\[\D_{1+\epsilon}^*\times \mathcal M_{dR}^{D,\alpha_,\dots,\alpha_n}\subset\mathcal M_{Hod}^{par}(\Sigma)\mid_{\C^*}\subset
\mathcal M_{Hod}^{par}(\Sigma)\subset \mathcal M_{DH}^{par}(\Sigma).\]
We first extend this section to $\lambda=0.$ 
Consider the holomorphic structure $\bar\partial_{\mathcal V}$.
Then, $\Psi$ is meromorphic with at most first order poles at $p_1,\dots,p_n$. 
By assumption the eigenvalues $\pm\alpha_j$ of res$_{p_j}\nabla^\lambda$ are independent
of $\lambda.$
This implies that res$_{p_j}\Psi$ 
is nilpotent for all $j= 1, ..., n$. We define the quasiparabolic line at $p_j$ as follows: if $\text{res}_{p_j}(\Psi)\neq0$, then $\ell_j:=\text{ker}(\text{res}_{p_j}(\Psi)),$ else   we define the line to be $\ell_j:=\text{Eig}((\text{res}_{p_j}\nabla)\mid_{\lambda=0},\alpha_j).$
This gives rise to a parabolic structure $\mathcal P$ on $\mathcal V$ with 
$\Psi$ being a compatible strongly parabolic Higgs field. To extend the section to $\lambda= 0,$ we need to ensure that the strongly parabolic Higgs pair $(\mathcal P, \Psi)$ is polystable.
\begin{remark}
In this paper,  all $\D^*_{1+\varepsilon}$-families of logarithmic connections considered will  induce stable strongly parabolic Higgs pairs at $\lambda= 0$. Thus there is no further gauge necessary to obtain a proper section of the Deligne-Hitchin moduli space.
\end{remark}

Assume the parabolic Higgs pair $(\mathcal P, \Psi)$  is unstable, then there exists an invariant holomorphic line bundle $L$ with respect to $\bar\partial_{\mathcal V}$ (on $\Sigma$) of positive parabolic degree. Consider a complementary complex line bundle $\widetilde L$, which is not holomorphic with respect to $\bar\partial_{\mathcal V}$ in general. Define the holomorphic family of gauge transformations
by $\lambda\in\C^*\mapsto \wt g(\lambda)=\text{diag}(1,\lambda)$ with respect to the $C^\infty$ decomposition $V=L\oplus \widetilde L.$ Then, it can be shown analogously to  \cite[Section 1.4]{HH} that the family of logarithmic connections
\begin{equation}\label{eq:twist}\widetilde\nabla^\lambda:=\nabla^\lambda.\wt g(\lambda)=\lambda^{-1}\widetilde\Psi+\widetilde\nabla+ \text{higher order terms in } \lambda\end{equation} 
is gauge equivalent to $\nabla^\lambda$ for $\lambda\neq 0$ and gives rise to a stable parabolic structure at $\lambda = 0$ with strongly parabolic Higgs field $\wt \Psi.$

If the parabolic structure $(\mathcal P , \Psi)$ is semi-stable but not polystable, we can replace the Higgs pair at $\lambda = 0$ by a polystable pair in the same gauge orbit. (This case does  not lead to twistor lines, but it does neither occur for the real holomorphic sections constructed by in this paper. In fact, either the parabolic Higgs field is
non-vanishing and induces a stable parabolic Higgs pair by Lemma \ref{lem:stabhiggspar}, or the Higgs field  vanishes and the section of the parabolic Deligne-Hitchin moduli space is induced by a unitary logarithmic connection by Theorem \ref{blowuplimit} below.)

To extend the section to $\C P^1$ recall that the underlying $C^\infty$ trivialization of the holomorphic bundles $\mathcal V(\lambda)$ is of fixed type, so that the complex conjugation
$c$ is well-defined. Then, by Lemma \ref{lem:cmap} and \eqref{eq:realityconditionpara}, $s$ extends holomorphically to
$\{\lambda \in \C \cup \{\infty\} | |\lambda| >\tfrac{1}{1+\varepsilon}\}$, i.e., $s$ gives rise to a global holomorphic section of $\mathcal M_{DH}^{par}(\Sigma)\to\C P^1$, which is real by \eqref{eq:realityconditionpara}.
\end{proof}
We call the transformation $\nabla^\lambda\mapsto\widetilde\nabla^\lambda$ in \eqref{eq:twist} the {\em twist} of the family $\nabla^\lambda.$ The twist transformation can also be applied if the Higgs pair is stable. 
In fact, assume that $\nabla^\lambda$ is the associated family of a tame harmonic metric for some nilpotent strongly parabolic Higgs field, which can be interpreted as an equivariant conformal harmonic map $f$, i.e., a minimal surface, into the hyperbolic 3-space. Then, the twisted  family $\wt\nabla^\lambda$ corresponds to the associated family of the equivariant harmonic conformal Gauss map of $f$ into the deSitter 3-space, see \cite{HH}.

In the following, we will always assume, without lost of generality, that we start with a family of logarithmic connections $\nabla^\lambda$ inducing a stable (or polystable)
strongly parabolic Higgs pair, without applying the twist transformation first.
In particular, if $\nabla^\lambda=\lambda^{-1}\Psi+\nabla+\lambda\Psi_1+\dots$ such that the determinant of the strongly parabolic Higgs field $\Psi$ has only simple zeros or poles, the corresponding Higgs pair must be stable, as it cannot have a holomorphic $\Psi$-invariant line bundle.

\begin{proposition}\label{pro:nectwist}
Let  $s$ be the real holomorphic section of $\mathcal M_{DH}^{par}$ given by a family of irreducible logarithmic connections $\nabla^\lambda$ as in Lemma \ref{lem:sufsec} together with the gauge $g$ satisfying \eqref{eq:realityconditionpara}.
Then, necessary conditions for $s$ to be a twistor line are
\begin{itemize}
\item $g_x$ lies in the big cell for all $x\in\Sigma\setminus\{p_1,\dots,p_n\},$ i.e., $g_x=g_x^+g_x^-$ for some
$g_x^+\in \Lambda_+\mathrm{SL}(2,\C)$ and $g_x^-\in \Lambda_-\mathrm{SL}(2,\C),$ \\ \item  $g(\lambda)\overline{g(-\bar\lambda^{-1})}=-\mathrm{Id}.$\end{itemize}

\end{proposition}
\begin{proof}
If the family of logarithmic connections $\nabla^\lambda$ induces a twistor line, then $\nabla^\lambda$ is gauge
equivalent to the associated family of flat connections of a self-duality solution   \eqref{eq:asslambdafami} by a positive gauge $\wt g$, i.e.,  $\wt g_x \in  \Lambda_+\mathrm{SL}(2,\C)$ for all $x\in\Sigma\setminus\{p_1,\dots,p_n\}$. 
Since by Example \ref{exa:negative} associated families are negative sections and the gauge is given by \eqref{g=delta}, 
we obtain by \eqref{eq:realityconditionpara} and by irreducibility of $\nabla^\lambda$ that 
$\wt g^{-1}g\in\Lambda_-\mathrm{SL}(2,\C).$  The second condition can be derived as in the compact case, for details see \cite[Section 1.3]{HH}.
\end{proof}

The next proposition shows that being a twistor line is an open condition. 

\begin{proposition}\label{pro:sufftwist}
Let $s$ be a twistor line of $\mathcal M_{DH}^{par}(\Sigma,D,\alpha_1,\dots,\alpha_n)$ with rational parabolic weights $\alpha_j\in(0,\tfrac{1}{4})\cap \mathbb Q$. Then there exists an open neighborhood $U$ around $s$ in the space of the real holomorphic sections of $\mathcal M_{DH}^{par}(\Sigma,D,\alpha_1,\dots,\alpha_n)$ such that every $\widetilde s\in  U$ is a twistor line.
\end{proposition}
\begin{proof}
Consider the covering surface $\wh\Sigma\to\Sigma$ which branches over $p_1,\dots,p_n$ of appropriate order. Over $\wh\Sigma$ we obtain, after applying the twisted lift construction, a holomorphic family of smooth connections 
$$\wh\nabla^\lambda=\lambda^{-1}\wh\Psi+\wh\nabla+ \text{ higher order terms in }\lambda,$$
 see Theorem \ref{thm:twistedlifttwistorlines}, such that the induced Higgs pair at $\lambda = 0$ is stable. 
Since $s$ is a twistor line of  $\mathcal M_{DH}^{par}(\Sigma,D,\alpha_1,\dots,\alpha_n)$, the family $\wh\nabla^\lambda$ on $\wh\Sigma$ gives rise to a twistor line of $\mathcal M_{DH}(\wh\Sigma)$ by Theorem \ref{thm:twistedlifttwistorlines} and the gauge $g$ with 
$$\wh\nabla^\lambda.g=\overline{\wh\nabla^{-\bar\lambda^{-1}}}$$ 
possesses a global  corresponding loop group factorization $g=g^+g^-$ on $\wh\Sigma.$

Let $\widetilde s$ be another real holomorphic section of $\mathcal M_{DH}^{par}(\Sigma,D,\alpha_1,\dots,\alpha_n)$ lying in an appropriate open neighborhood of $s$. Then its twisted lift $\wh{\widetilde s}$ is a holomorphic section of $\mathcal M_{DH}(\wh\Sigma)$ which can be represented
by a family $\wh {\wt\nabla}^\lambda$ of flat and smooth connections on $\wh\Sigma.$ The family $\wh {\wt\nabla}^\lambda$ can be chosen in a way such that for all $\lambda \in K$, where $K$ is a compact subset of $\C^*$ containing the unit circle, $\wh {\wt\nabla}^\lambda$
lies in an appropriate open neighborhood of $\wh\nabla^\lambda$ in the space of flat connections. Consequently, the
family of $\mathrm{SL}(2,\C)$ gauge transformations $\wt g$ satisfying 
$$\wh {\wt\nabla}^\lambda.\wt g(\lambda)=\overline{\wh {\wt\nabla}^{-\bar\lambda^{-1}}}$$ lies in an open neighborhood of $g$ for all $\lambda\in \{\lambda\in\C\mid \lambda\in K,-\bar\lambda^{-1}\in K\}.$
By Theorem \ref{thm:birk}, we obtain that $\wt g$ also lies in the big cell, and using the construction in Section \ref{sec:reconstruction} we obtain $\wh {\wt s}$ is a twistor line of $\mathcal M_{DH}(\wh\Sigma)$. Then Theorem \ref{thm:twistedlifttwistorlines} gives that $\wt s$ is also a twistor line of
$\mathcal M_{DH}^{par}(\Sigma,D,\alpha_1,\dots,\alpha_n)$ as claimed.
\end{proof}
\begin{remark}\label{rem:sufftwist}
Proposition \ref{pro:sufftwist} 
also holds for non-rational parabolic weights. More generally, the same conclusion holds if we allow the
(positive) parabolic weights as well as the conformal structure of $\Sigma$ and the singular points $p_1,\dots,p_n$
to vary. 
This follows from the results of \cite{Si} combined with \cite{KiWi}.
In fact, as a generalization of Theorem \ref{connectedcomponent}, the space of twistor lines is open and closed in the space of real  holomorphic sections.
\end{remark}

\subsection{The symplectic form on the space of logarithmic  connections}
The symplectic form on the moduli space of flat connections over compact surfaces \cite{AB,Gold} has been generalized
to moduli spaces of flat connections with prescribed local conjugacy classes over punctured Riemann surfaces,
see for example  \cite{AlMa} or \cite{Audin}. We  provide a short self-contained account to the construction of this symplectic form, and describe its relationship to the symplectic form on a appropriate compact covering surface
in the case of rational parabolic weights.

Let $\Sigma$ be a compact Riemann surface and $p_1,\dots,p_n\in\Sigma$ be pairwise distinct points. Fix $\mathrm{SL}(2,\C)$
conjugacy classes
\[\mathcal C_1,\dots,\mathcal C_n\]
of diagonal matrices $C_1,\dots,C_n\in\mathfrak{sl}(2,\C)$ at the punctures such that the eigenvalues of each $C_j$ are contained in $ (0,\tfrac{1}{2}).$ Let $z_j$ be a centered holomorphic coordinate at $p_j$ for $j=1,\dots,n$.
For $\mathcal C_1,\dots,\mathcal C_n$ fixed let $\mathcal A$ denote the infinite dimensional space of all flat $\mathrm{SL}(2,\C)$ connections $\nabla$
on $\Sigma^0:=\Sigma\setminus\{p_1,\dots,p_n\}$ which are of the form   \begin{equation}
\label{nfconn}
\nabla=A_j\frac{dz_j}{z_j}+\text{smooth connection},\end{equation}
where $z_j$ is a centered coordinate around $p_j$ and $A_j\in\mathcal C_j$.
\begin{lemma}\label{lem:tnablaxia}
Let $X\in\Omega^1(\Sigma^0,\mathfrak{sl}(2,\C))$ be a tangent vector to $\nabla\in\mathcal A.$  Then, there exists
smooth $\xi\in\Gamma(\Sigma,\mathfrak{sl}(2,\C))$ and $\wh X\in\Omega^1(\Sigma,\mathfrak{sl}(2,\C))$
such that 
\[X=d^\nabla \xi+\wh X \and  d^\nabla\wh X=0\]
on the punctured surface $\Sigma^0.$
\end{lemma}
\begin{proof}
Since $X$ is a tangent vector at $\nabla$ to the infinite dimensional space of flat connections, we have in particular $d^\nabla X = 0.$ Moreover, 
 $X$ preserves the form  \eqref{nfconn}, and  therefore we can write
\[X=[A_j,\xi_j]\frac{dz_j}{z_j}+\widetilde X\]
around $p_j$ for appropriate $\xi_j\in\mathfrak{sl}(2,\C)$ and a smooth $\widetilde X$ on $\Sigma.$ Let 
$\xi\in\Gamma(\Sigma,\mathfrak{sl}(2,\C))$ be a section with $\xi(p_j)=\xi_j.$
Then the splitting \[X=d^\nabla\xi+(X-d^\nabla\xi)\] is of the required form.
\end{proof}

\begin{lemma}\label{lem:compsympbound}
Let $\nabla\in \mathcal A$ with $A_j$ as in \eqref{nfconn} and consider two tangent vectors $$X=d^\nabla \xi+\wh X, \quad Y=d^\nabla \mu+\wh Y\in T_{\nabla}\mathcal A$$
for $\xi,\mu\in\Gamma(\Sigma,\mathfrak{sl}(2,\C))$ and $\wh X,\wh Y\in\Omega^1(\Sigma,\mathfrak{sl}(2,\C))$ as in Lemma \ref{lem:tnablaxia}.
Then
\[\tfrac{1}{8\pi }\int_{\Sigma^0}\tr\left (X\wedge Y\right)=\tfrac{1}{8\pi }\int_\Sigma\tr\left(\wh X\wedge\wh Y\right) - i\sum_{j=1}^n 
\frac{1}{8\tr(A_j^2)} \tr(A_j[\text{Res}_{p_j}(X),\text{Res}_{p_j}(Y)]\,]  ).\]
\end{lemma}
\begin{proof}
Consider the punctured Riemann surface $\Sigma^0$ and centered holomorphic coordinates $z_j$ at the punctures $p_j.$ For $t>0$ small let
$$\gamma^j_t\colon S^1\longrightarrow \Sigma^0; \quad e^{2\pi i \varphi}\longmapsto (z_j)^{-1}(te^{2\pi i \varphi}).$$

When splitting $X, Y$ according to
$$X=d^\nabla \xi+\wh X, \quad Y=d^\nabla \mu+\wh Y$$
we obtain
\[\int_{\Sigma^0}\tr\left (X\wedge Y\right) = \int_{\Sigma^0}\tr\left (\wh X\wedge \wh Y\right) + \tr\left (d^\nabla \xi \wedge \wh Y\right) + \tr\left (\wh X\wedge d^\nabla \mu \right) +\tr\left (d^\nabla \xi \wedge d^\nabla \mu\right).\]
For the second term in the above expression we have
\begin{equation}
\begin{split}
\int_{\Sigma^0}\tr\left (d^\nabla \xi \wedge \wh Y\right) &=\int_{\Sigma^0}d\,\tr(\xi \wh Y)-\int_{\Sigma^0}\tr(\xi (d^\nabla \wh Y ))\\
&=-\lim_{t\to 0} \sum_j\int_{\gamma^j_t}\tr(\xi \wh Y)=0
\end{split}
\end{equation}
since $d^\nabla \wh Y = 0$ and $\wh Y, \xi$ are both smooth on $\Sigma.$ Analogously, we have
\[\int_{\Sigma^0}\tr(\wh X\wedge d^\nabla \mu) =0.\]
For the last term we compute
\begin{equation}
\begin{split}
\int_{\Sigma^0}\tr(d^\nabla\xi\wedge d^\nabla \mu)&=\int_{\Sigma^0}d\,\tr(\xi d^\nabla \mu) =-\lim_{t\to 0} \sum_j\int_{\gamma^j_t}\tr(\xi d^\nabla\mu)\\
&=-\lim_{t\to 0} \sum_j\int_{\gamma^j_t}\tr(\xi [A_j,\mu] \frac{dz_j}{z_j})-\lim_{t\to 0} \sum_j\int_{\gamma^j_t}\tr(\text{something smooth})\\
&=-2\pi i\sum_j \tr(\xi(p_j) [A_j,\mu(p_j)] )=2\pi i\sum_j \tr(A_j[\xi(p_j) ,\mu(p_j)] )\\
&=-2\pi i\sum_j \frac{1}{2\,\tr(A_j^2)}\tr(A_j[\,[A_j,\xi(p_j)] ,[A_j,\mu(p_j)]\,] )\\
&=-2\pi i\sum_j \frac{1}{2\,\tr(A_j^2)}\tr(A_j[\text{Res}_{p_j}(X),\text{Res}_{p_j}(Y)]\,] )
\end{split}
\end{equation}
as claimed.\end{proof}

Define a holomorphic complex bilinear and skew-symmetric form on the (infinite dimensional) tangent space $T_{\nabla}\mathcal A$ by 
\begin{equation}\label{def:symf}
\begin{split}
\mathcal O \colon &T_{\nabla}\mathcal A\times T_{\nabla}\mathcal A\to \C\\
&(X,Y)\mapsto \tfrac{1}{8\pi }\int_{\Sigma^0}\tr(X\wedge Y)+ i\sum_j \frac{1}{8\,\tr((\text{Res}_{p_j}(\nabla))^2)}
\tr(\, \text{Res}_{p_j}(\nabla)[ \text{Res}_{p_j}(X),\text{Res}_{p_j}(Y)]\, ).
\end{split}
\end{equation}
From Lemma \ref{lem:compsympbound} we then obtain:
\begin{corollary}\label{cor:vanishing}
For $d^\nabla\xi,Y\in T_{\nabla}\mathcal A$ we have
\[\mathcal O(d^\nabla \xi,Y) = 0.\]
Therefore, the bilinear form $\mathcal O$ descends to a well defined holomorphic 2-form on the quotient of $\mathcal A$ by gauge transformations.
\end{corollary}
It can be shown that $\mathcal O$ is indeed a symplectic form. 
\begin{remark}
If the infinitesimal deformation of the flat connection preserves the holomorphic structure, i.e., the tangent vector fields $X,$ $Y$ are both meromorphic, then the first summand of $\mathcal O$  in \eqref{def:symf}vanishes and only the residue terms remain.
For $A\in \mathcal C_j$, the complex skew bilinear form
\[X=[A,\xi],Y=[A,\mu]\in T_{A} \mathcal C_j\longmapsto \tr(A[\xi,\mu])=-\frac{1}{2\tr(A^2)}\tr(A[X,Y])\]
is well-defined and known as the Kirillov symplectic form on the adjoint orbit $\mathcal C_j$.
\end{remark}

\subsubsection{Rational weights}
Assume that the parabolic weights are all rational.
Then, the symplectic form $\mathcal O$ and the Goldman symplectic form coincide up to scaling and taking 
the twisted lift.  We give the proof in the case of the 4-punctured sphere, with rational weights
$\alpha_1=\dots=\alpha_4=\tfrac{l}{k}.$ As before, consider the $k$-fold covering $\Sigma_k\to\C P^1$ given by \eqref{eq:sigma_k} which totally branches over the singular points.

\begin{proposition}\label{pro:sympupdown}
Let $\mathcal M(\C P^1)$ be the space of logarithmic connections  on $\C P^1$ with singular points $p_1,\dots,p_4$
and rational weights $\alpha_1=\dots=\alpha_4=\tfrac{l}{k}$  equipped with its symplectic form $\mathcal O$.
Let $\pi\colon\Sigma_k\to\C P^1$ be the $k$-fold covering given by \eqref{eq:sigma_k} which totally branches over the singular points, and let $\mathcal M(\Sigma_k)$ be the moduli space of flat connections on $\Sigma_k$ with Goldman symplectic form $\Omega.$
Then \[32\pi \,k\,\mathcal O=\Omega.\] 
\end{proposition}
\begin{proof}
Because of Corollary \ref{cor:vanishing} $\mathcal O$ is well-defined on the moduli space $\mathcal M(\C P^1)$, and for tangent vectors $X,Y,$ we can choose representatives  that vanish in an open set $U \subset \C P^1$ containing all the singular points.
Then, only the first summand in \eqref{def:symf}
contributes to $\mathcal O(X,Y).$ On the other hand, the 
pull-back of $\pi^*X$ and $\pi^*Y$ to $\Sigma_k$ vanish on $\pi^{-1}(U)$, and the tangent vectors $\wh X$ and $\wh Y$
of $\mathcal M(\Sigma_k)$
 are represented by $\pi^*X$ and $\pi^*Y$ in the  trivialization of the holomorphic bundle 
 $\wh{\mathcal V}\cong\mathcal V$ over $\Sigma_k\setminus\pi^{-1}\{p_1,\dots,p_4\}$. Since $\pi\colon\Sigma_k\to\C P^1$ is of degree $k,$ and due to the different scalings in \eqref{GoldmanO}
 and \eqref{def:symf}, the result follows.
\end{proof}
\begin{remark}\label{rem:sympupdown}
Proposition \ref{pro:sympupdown}  directly generalizes to $\lambda$-connections. As a consequence, we are able to compute
the twisted holomorphic  symplectic form \eqref{deftwistedOmega2} on the open
dense subset of $\mathcal M_{DH}(\C P^1,p_1+\dots+p_4,t)$ consisting of Fuchsian
$\lambda$-connections.
\end{remark}

\section{Initial conditions at $t=0.$}\label{Ansatz}
In this section we write down $\lambda$-dependent connection 1-forms on the 4-puncture sphere $\C P^1\setminus\{p_1,\dots,p_4\}$ depending on a parabolic weight $\alpha=t \in [0, \tfrac{1}{4})$ and a strongly parabolic Higgs field.
We then explicitly compute the initial conditions at $t=0$ which we will deform via an implicit function theorem argument in Section \ref{IFT} to obtain equivariant tame harmonic maps from $\C P^1\setminus\{p_1,\dots,p_4\}$ into $\H^3$ for $t>0$.

\subsection{The potential}
To construct real holomorphic sections explicitly, we restrict to the case where the underlying holomorphic structure is trivial for all
  $\lambda \in \mathbb D_a=\{\lambda\mid |\lambda|^2< a^2\}$ for some $a >1.$ 
This restriction is motivated by the fact that an
open and dense subset of the moduli space of strongly parabolic Higgs
field (respectively logarithmic connections) has trivial underlying holomorphic
bundle, see Proposition \ref{pro:classparahiggs} for strongly parabolic Higgs fields (and Lemma \ref{lemOK1} for
logarithmic connections).  In general (e.g. for large Higgs field with trivial underlying holomorphic structure), the associated family of flat connections does not have trivial underlying holomorphic bundle for all $\lambda$ in the punctured unit disc. On the other hand, 
for fixed weight $t$ and fixed compact subset $K$ of the $\lambda$-spectral plane, we will show that there is a constant $C$ such that for 
any strongly parabolic Higgs field with trivial underlying holomorphic structure and norm less than $C$, the associated logarithmic connection $\nabla^\lambda$ is Fuchsian for all $\lambda\in K\setminus\{0\}$.  This will follow by continuity of the 
solution depending on the parabolic Higgs data (with trivial underlying holomorphic bundle), and the fact that for vanishing Higgs field the associated family is constant (and Fuchsian).

 A {\em potential} is given by a holomorphic and complex linear loop algebra-valued 1-form
\[\eta\in\Omega^{1,0}(\C P^1\setminus\{p_1,\dots,p_4\},\Lambda \mathfrak{sl}(2,\C))\] with  
\[(\lambda \eta)\in \Omega^{1,0}(\C P^1\setminus\{p_1,\dots,p_4\},\Lambda_+\mathfrak{sl}(2,\C)),\]
where $\Lambda \mathfrak{sl}(2,\C)$ and $\Lambda_+\mathfrak{sl}(2,\C)$ are as in Section \ref{sec:loops} 
and $(\lambda \eta)$ denotes pointwise multiplication. 
Motivated by Lemma \ref{lem:sufsec} and its proof we call the residue at $\lambda=0$ 
\[\Psi=\eta_{-1}:=\text{Res}_{\lambda=0} (\eta)\]
the parabolic Higgs field of the potential $\eta$. 
On the $4$-punctured sphere we  further specialize to $\eta$ being a Fuchsian system  for every $\lambda \in \D_{a}\setminus\{0\}$, i.e., we consider potentials of the form

$$\eta =  \sum_{j = 1}^4 A_j \frac{dz}{z-p_j} \quad \text{ satisfying }\quad \sum_{j= 0}^4 A_j = 0$$

with $z$ being the homogenous coordinate of $\C P^1,$ $ (\lambda A_j) \in \Lambda_+\mathfrak {sl}(2, \C)$.  To be more explicit, the coefficients of $A_j$, as functions of $\lambda,$
are in the functional space $\cal{W}^{\geq -1}_{a}$ where for $l\in\Z$
\[\cal{W}^{\geq l}_{a}:=\left\{h=\sum_{k=l}^\infty h_k \lambda^k \;\text{ such that } \; \sum_{k=l}^\infty |h_k| a^{|k|} < \infty \right\}.\]
 For $l\geq0$ we obtain the space of absolutely convergent power series in the disk of radius $a>1$. We denote the subspace of convergent power series with vanishing constant term $h_0$ by $\cal{W}^{+}_{a}.$ Whenever the dependence on a particular $a$ is not important, we will omit this index. More generally, $\cal{W}_{a}$ denotes the space of convergent Laurent series on the annulus $\mathbb A_a   := \{\tfrac{1}{a^2}<|\lambda|^2< a^2\}.$
Every element  $h \in \mathcal W =\mathcal W_a$ can be decomposed into its positive part $h^+\in \mathcal W^+$, its constant part $h^0 := h _0$ and its negative part $h^-= h- h^+ -h^0.$

\begin{remark}
Even though the choice of $a>1$ is irrelevant for the application of the implicit function theorem for $t\sim0,$
we expect that for larger $t$ and general strongly parabolic Higgs field (with trivial underlying holomorphic structure), the associated family of flat connection is not Fuchsian
for all $\lambda$ satisfying $|\lambda|^2\leq a^2.$ 
Instead, we expect unstable logarithmic connections (where the power series expansion in $\lambda$ will necessarily have poles) moving into the unit disc within the $t$-deformation.
This is in contrast to the case of the equations describing (equivariant) harmonic maps into the 3-sphere, which differ from the self-duality equation only by a sign, where unstable points cannot cross the unit circle, see \cite{HH3}.
\end{remark}
 
 To fix further notations let $\C^+$ denote the upper-right quadrant of the plane
$$\C^+=\{z\in\C\,\mid\, \Re(z)>0,\Im(z)>0\}.$$
Let  $p\in\C^+$
and consider the 4-punctured sphere
$$\Sigma=\Sigma_p=\C P^1\setminus\{p_1,p_2,p_3,p_4\}$$
with
\begin{equation}\label{pointsonp}p_1=p,\quad p_2=-1/p,\quad p_3=-p\and p_4=1/p.\end{equation}
Up to M\"obius transformations, every 4-punctured sphere is of this form.
By definition $\Sigma$ is invariant under the holomorphic involutions $\delta(z)=-z$ and $\tau(z)=1/z$. Consider the Pauli matrices 
$$\mathfrak m_1=\matrix{1&0\\0&-1}
\qquad
\mathfrak m_2=\matrix{0&1\\1&0}
\qquad
\mathfrak m_3=\matrix{0&i\\-i&0}$$
and the holomorphic 1-forms on $\Sigma$

\begin{equation}\label{omega}
\begin{split}
\omega_1&=\frac{dz}{z-p_1}-\frac{dz}{z-p_2}+\frac{dz}{z-p_3}-\frac{dz}{z-p_4}\\
\omega_2&=\frac{dz}{z-p_1}-\frac{dz}{z-p_2}-\frac{dz}{z-p_3}+\frac{dz}{z-p_4}\\
\omega_3&=\frac{dz}{z-p_1}+\frac{dz}{z-p_2}-\frac{dz}{z-p_3}-\frac{dz}{z-p_4}.
\end{split}
\end{equation}

Then the $\omega_i$ have the symmetries
\begin{equation}
\label{eq:symmetry-omega}
\begin{cases}
\delta^*\omega_1=\omega_1,\quad \delta^*\omega_2=-\omega_2,\quad\delta^*\omega_3=-\omega_3\\
\tau^*\omega_1=-\omega_1,\quad\tau^*\omega_2=\omega_2,\quad\tau^*\omega_3=-\omega_3\,.\end{cases}
\end{equation}

\begin{ansatz}\label{etatttinfront}
In the following we will consider potentials of the form
\begin{equation}
\eta_t=t\sum_{j=1}^3 x_j(t)\mathfrak m_j\omega_j\end{equation}
where $t\sim 0$ is a real parameter and $x_1(t)$, $ x_2(t)$, $x_3(t) \in \mathcal W^{\geq -1}$  are parameters depending on $t$.
\end{ansatz}
We aim to determine 
$\eta_t$
in dependence of $t$ through the implicit function theorem
by imposing the reality condition \eqref{eq:realityconditionpara}. 
The maps $t\mapsto x_j(t)$, $j=1,2,3$, are always assumed to be smooth in some neighbourhood $(-\epsilon,\epsilon)$ of $t=0.$ 
We denote the parameter vector  by \[ x(t)=(x_1(t),x_2(t),x_3(t)) \in  (\mathcal W^{\geq -1})^3.\] 
\begin{remark}\label{rem:inidata}
For $t\sim0$, the parameter $x(t)$ is not uniquely determined by the reality condition \eqref{eq:realityconditionpara} as we have a non-trivial moduli space of solutions, which is parametrized by Higgs data. Hence, we have to incorporate the dependency on the parabolic Higgs pair
when applying the implicit function theorem.
We omit this dependency for now, and refer to Section \ref{sec:t=0para} below and to Theorem \ref{thm:IFT} below for precise statements.
\end{remark}

With this ansatz $\eta_t$ is a Fuchsian system for every $\lambda \in \mathbb D_a\setminus\{0\}$ and the residues at a puncture is given by a $\Lambda \mathfrak{sl}(2,\C)$ element
$$\Res_{z=p_j} \eta_t = t A_j$$
 with 
\begin{equation}\label{symmetry-fix1-4}
\begin{split}
A_1&=\matrix{x_1(t)&x_2(t)+ix_3(t)\\x_2(t)-ix_3(t)&-x_1(t)} \quad\quad
A_2=\matrix{-x_1(t) &-x_2(t)+ix_3(t)\\-x_2(t)-ix_3(t)&x_1(t)}\\
A_3&=\matrix{x_1(t)&-x_2(t)-ix_3(t)\\-x_2(t)+ix_3(t)&-x_1(t)}\quad
A_4=\matrix{-x_1(t)&x_2(t)-ix_3(t)\\x_2(t)+ix_3(t)&x_1(t)}
\end{split}
\end{equation}
satisfying
\begin{equation}\label{parabolicweightstxx}\det(A_j)=-x_1(t)^2-x_2(t)^2-x_3(t)^2\end{equation}
for all $j= 1, ..., 4$. Moreover, $\Psi=\text{res}_{\lambda=0}\eta_t$ takes the form \eqref{eq:psiAs} with
$R_j=\text{res}_{\lambda=0}A_j.$

This ansatz is chosen such that the potential has the following symmetries
\begin{equation}\label{symmetry-fix4}
\begin{split}\delta^*\eta_t&=D^{-1}\eta_t D\with D=\matrix{i&0\\0&-i}\\
\tau^*\eta_t&=C^{-1}\eta_t C\with C=\matrix{0&i\\i&0}.
\end{split}
\end{equation}
We will show that these symmetries can be imposed without loss of generality
if we start with a symmetric Higgs field and $t\sim0$ is small.
\begin{remark}
With this ansatz we only study harmonic maps for which the monodromy becomes trivial for $t\rightarrow 0$ and consequently the corresponding Higgs fields $\Psi(t) \rightarrow 0$.
In fact, evaluating Ansatz \eqref{etatttinfront} at $t=0$ gives the trivial connection with trivial monodromy due to the factor $t$ in Ansatz \eqref{etatttinfront}.
 \end{remark}

\subsection{The choice of initial data at $t=0$, and parabolic Higgs fields}\label{sec:t=0para}
In this section, we discuss the initial data at $t=0$ for the implicit function theorem. These initial data are naturally parametrized by the coordinates $(u,v)$ from \eqref{eq:psier}, which determine the Higgs pair.

We denote the initial value of the parameters with an underscore $\cvx_j:=x_j(0)$, and accordingly  $\cv{A}_j$. These are chosen such that  $\cvx_1$, $\cvx_2$, $\cvx_3$ at $t=0$ satisfy
\begin{equation}\label{initialproblem}
\begin{cases}\det(\cv{A}_1)=-1\\
\cv{A}_1=\cv{A}_1^*,\end{cases}
\end{equation}
where $()^*$-operator for matrices is defined as
$$M^*(\lambda)=\overline{M(-1/\overline{\lambda})}^T.$$
The first equation  in \eqref{initialproblem}  implies that the local monodromies of the associated Fuchsian systems $d+\eta_t$ lie in the conjugacy class of diag$\left(\exp(2\pi it),\exp(-2\pi i t)\right)$ for every $\lambda \in \mathbb D_a.$ 
The second equation   in \eqref{initialproblem}  is an infinitesimal version of $\mathcal T$-reality at $t=0$. The equations \eqref{initialproblem}  on $\cv{A}_1$  are equivalent to
$$\begin{cases}
\cvx_1^2+\cvx_2^2+\cvx_3^2=1\\
\cvx_j=\cvx_j^*, \quad \forall j,
\end{cases}$$
where the induced $()^*$-operator for functions is defined to be
\begin{equation}\label{fstar}f^*(\lambda)=\overline{f(-1/\overline{\lambda})}.\end{equation}

 Since potentials have at most a first order pole at $\lambda = 0$ the eligible $\cvx_j$ must be a degree 1 Laurent polynomial 
\begin{equation}\label{eq:defcvxj}\cvx_j=\cvx_{j,-1}\lambda^{-1}+\cvx_{j,0}+\cvx_{j,1}\lambda \end{equation}
with
\begin{equation}\label{eq:detcvxj}
\begin{cases}
\cvx_{j,0}\in\R\\
\cvx_{j,1}=-\overline{\cvx_{j,-1}}.
\end{cases}
\end{equation}
Then $\cvx_1^2+\cvx_2^2+\cvx_3^2=1$ is equivalent to
\begin{equation}
\label{eq:det1}
\begin{cases}
{\displaystyle \sum_{j=1}^3}\cvx_{j,-1}^2=0\\

{\displaystyle\sum_{j=1}^3}\cvx_{j,-1}\cvx_{j,0}=0\\

{\displaystyle\sum_{j=1}^3}\cvx_{j,0}^2- 2|\cvx_{j,-1}|^2=1.
\end{cases}
\end{equation}
To solve the first equation of \eqref{eq:det1}, consider the standard parametrization of the quadric $\{x^2+y^2+z^2=0\}$ in $\C^3$ given by
\begin{equation}
\label{eq:ambmcm}
\cvx_{1,-1}=u\,v,\quad
\cvx_{2,-1}=\tfrac{1}{2}(v^2-u^2)\and
\cvx_{3,-1}=\tfrac{i}{2}(u^2+v^2)
\end{equation}
with $(u,v)\in\C^2\setminus\{ (0,0)\}$.
Then the second equation of \eqref{eq:det1} gives
\begin{equation}
\label{eq:det2}
u\,v\, \cvx_{1,0}+\tfrac{1}{2}(v^2-u^2)\, \cvx_{2,0}+\tfrac{i}{2}(u^2+v^2)\,\cvx_{3,0}=0
\end{equation}
and its (real) solutions are 

\begin{equation}
\label{eq:a0b0c0}
\cvx_{1,0}=\rho(|u|^2-|v|^2),\quad
\cvx_{2,0}=2\rho\,\Re(u\overline{v})\and
\cvx_{3,0}=2\rho\,\Im(u\overline{v}),\quad \text{ for some }
\rho\in\R.\end{equation}

Finally, the third equation of \eqref{eq:det1} becomes
\begin{equation}\label{eq:rho0}\rho^2(|u|^2+|v|^2)^2-(|u|^2+|v|^2)^2=1\end{equation}  which
determines $\rho$ up to sign.
The geometric meaning of these equations and its relationship to the classical Weierstrass representation of minimal surfaces in Euclidean 3-space 
is discussed in Section \ref{sectionNAHt0}. 
We will show in Section \ref{sec:twist} below that the correct sign of $\rho$ to obtain twistor lines is 
\begin{equation}
\label{eq:rho}
\rho=\sqrt{1+(|u|^2+|v|^2)^{-2}}>0.
\end{equation}
But we obtain $\mathcal T$-real holomorphic sections of $\mathcal M_{DH}(\C P^1,p_1+\dots+p_4,t)$
for $t\sim0$ for both choices of sign in \eqref{eq:rho0}.

 We have thus fixed the initial conditions of $\eta_t$ depending on a pair $(u, v) \in \C^2\setminus\{ (0,0)\}$ which parametrizes the space of all eligible non-vanishing $\lambda$-residues of the $\eta_t$.
This gives a 8-fold covering of the open subset $\mathcal U$ of the parabolic  Higgs bundle moduli space specified by a trivial underlying holomorphic bundle. In fact, $(u,v)$ uniquely determines the nilpotent residue
$R_1$ of the Higgs field at $p_1$ up to sign, and by Lemma \ref{41cov}, $R_1$ determines the gauge class of the Higgs pair uniquely up to the $\Z_2\times\Z_2$ action generated by conjugation with $C$ and $D.$
\begin{convention}
In the following, we slightly abuse  notation and neglect the $\mathbb Z_2\times \mathbb Z_2$ action when referring to completion of the nilpotent orbit $\{R_1\in\mathfrak{sl}(2,\C)\mid \det(R_1)=0\}$ as the moduli space of parabolic Higgs fields. This 4-fold covering can actually be identified with an open dense subset of the space of parabolic Higgs fields on the 1-punctured torus with parabolic weight $\tfrac{1}{2}-2t.$

We will also ignore the complex line of stable parabolic Higgs fields with underlying holomorphic structure $\mathcal O(-1)\oplus\mathcal O(1)$ at $t=0$, since the  $t\to0$ limit of the corresponding representations is not the trivial representation.
\end{convention}

\begin{remark}
When considering Ansatz \eqref{etatttinfront}  for $t>0$ we have to rephrase  the second condition in \eqref{initialproblem} which encodes the reality condition at $t=0$. 
Instead to ensure \eqref{eq:realityconditionpara}
we have to require  that $\eta_t^*$ and $\eta_t$ lie in the same gauge class. By Lemma \ref{lem:sufsec} this corresponds to constructing families of flat connections that descend to real holomorphic sections of $\mathcal M_{DH}(\C P^1,p_1+\dots+p_4,t)$ for prescribed strongly  parabolic Higgs field at $\lambda= 0$ by varying parabolic weight $t$.
\end{remark}
 
 Using the above $(u,v)$-parametrization the parabolic Higgs field $\Psi=\Res_{\lambda=0}\,\eta_t$ is given by
 \eqref{eq:psier}
and has determinant 
\begin{equation}
\label{eq:detPsi}\det(\Psi)=\frac{-4t^2(u^4-(p^2+p^{-2})u^2v^2+v^4)}{z^4-(p^2+p^{-2})z^2+1},\end{equation}
where the singular points $p_1,\dots,p_4$ are determined by $p\in \C^+$ via \eqref{pointsonp}.
The 0-eigenlines of the nilpotent residues of $\Psi$
viewed as points in $\C P^1$ have homogenous coordinates $u/v$, $-v/u$, $-u/v$ and $v/u$, respectively.
The cross-ratio of these four eigenlines, which together with the parabolic weight $t$ determines the parabolic structure 
on the trivial holomorphic bundle, is computed to be
\begin{equation}\label{comp:crossratio}\left(\frac{u}{v},\frac{-v}{u};\frac{-u}{v},\frac{v}{u}\right)=\frac{-4 u^2v^2}{(u^2-v^2)^2},\end{equation}
where the cross-ratio is defined to be
$$(z_1,z_2;z_3,z_4)=\frac{(z_3-z_1)(z_4-z_2)}{(z_3-z_2)(z_4-z_1)}.$$
The cross-ratio \eqref{comp:crossratio} is already uniquely determined by the ratio $u/v$, up to
a $\Z_2\times\Z_2$ symmetry generated by $u/v\mapsto -u/v,\,u/v\mapsto v/u$ in accordance with Proposition \ref{pro:classparahiggs} (since scaling $u$ and $v$ simultaneously just scales the Higgs field).

\section{Constructing real holomorphic sections}\label{IFT}
In order for the potential $d+\eta_t$ to be the lift of a real holomorphic section $s$ we impose the condition that
\begin{equation}\label{DPWreality}
d+ \eta_t(\lambda) \quad\text{is gauge equivalent to}\quad  d+\overline{\eta_t(-\bar\lambda^{-1})} 
\end{equation}
for all $\lambda \in \mathbb A_a=\{\lambda\mid\,  a^{-2}<|\lambda|^2< a^2\}$
as in \eqref{eq:realityconditionpara}.
To make this condition more explicit, we identify 
the space of gauge equivalence classes of flat connections with
the space of representations of $\pi_1(\Sigma,z_0)$ modulo conjugation via the monodromy representation. 
Let $p\in \C^+$ and $\Sigma = \Sigma_p$ be the 4-punctured sphere and fix the base point $z_0=0.$ Choose the generators $\gamma_1$, $\gamma_2$, $\gamma_3$
and $\gamma_4$ of the fundamental group $\pi_1(\Sigma,0)$ as in \cite{HHT2}, i.e, let $\gamma_1$ be the composition of the real half-line from $0$ to $+\infty$ with the imaginary half-line from $+i\infty$ to $0$, and more generally
$\gamma_k$ is the product of the half-line from $0$ to $i^{k-1}\infty$ with the half line from
$i^k\infty$ to $0$, so $\gamma_k$ encloses $p_k$ and $\gamma_1\gamma_2\gamma_3\gamma_4=1$.
Let $\Phi_t$ be the fundamental solution of the Cauchy Problem
\begin{equation}
\label{eq:cauchy} d_{\Sigma}\Phi_t=\Phi_t\eta_t\quad\text{with initial condition}\quad
\Phi_t(z=0)=\Id\end{equation}
and let $M_k(t)=\mathcal M(\Phi_t,\gamma_k)$ be the monodromy of $\Phi_t$ along $\gamma_k$. 
This gives the {\em left} monodromy representation of the connection $d+\eta_t,$ i.e.,
\[M_1(t)M_2(t)M_3(t)M_4(t)=\mathrm{Id}.\]
The identification of the left and the right monodromy representation is done by taking the inverse, which is in  
accordance with the fact that for a solution $\Phi$ of $d\Phi=\Phi\eta$ its inverse solves $d\Phi^{-1}=-\eta\Phi^{-1}.$
In particular, two connections are gauge equivalent if and only if their left monodromies are conjugate to each other.
In order to directly  use results and explicit computations from our previous work (e.g. \cite{HHT1,HHT2}), we will work with the left monodromy in the following.

\subsubsection*{Fricke coordinates}
The moduli space of representations
 can be parametrized using the so-called Fricke coordinates. Define \[s_k:=\tr (M_k); \quad \quad s_{kl}=\tr (M_kM_l).\] Since $M_k \in \SL(2, \C), $ the trace $s_k$ determines the eigenvalues of $M_k, $ i.e., the conjugacy class of the local monodromy $M_k.$ For the symmetric case considered in this paper we restrict to $$s_1= s_2=s_3 = s_4 = 2\cos(2 \pi t)$$ 
for $t\in (0,\tfrac{1}{4})$ and the following classical result by Fricke-Voigt holds. 
 \begin{proposition}\label{Pro:Friecke} Consider a $\SL(2,\C)$-representation on the $4$-punctured sphere $\Sigma$. Let  $$s=s_1=\dots =s_4\in(0,2)$$ and   let $U=s_{12}$, $V=s_{23}$, $W=s_{13}.$ Then the following algebraic equation holds 
  \begin{equation}\label{eq:quadratic}
 U^2+V^2+W^2+U\,V\,W-2 s^2(U+V+W)+4(s^2-1)+s^4=0.\end{equation} When satisfying \eqref{eq:quadratic} the parameters $s$ and $U,V,W$ together determine a monodromy representation $\rho \colon \pi_1(\Sigma) \rightarrow$ SL$(2, \C)$ from the first fundamental group of $\Sigma$ into SL$(2, \C)$.  
By imposing the symmetries \eqref{symmetry-fix4}, this representation is unique up to conjugation. \end{proposition}  
   
\begin{proof}
For the first (classical) part of the proposition see for example \cite{gold2}. Whenever the representations are irreducible,  \cite{gold2} moreover shows that they are uniquely determined by their global traces $U,V, W $ up to conjugation, even without symmetry assumptions.
By \cite[Lemma 5]{gold2} a representation is reducible if and only if   $$s_{ij}\in\{2,-2+s^2\}=\{2,2 \cos (4 \pi  t)\}$$
for all $i,j$. For $t\in (0,\tfrac{1}{4})$  \eqref{eq:quadratic} then implies that
$$(U,V,W)\in\{(2,2,2 \cos (4 \pi  t)),(2,2 \cos (4 \pi  t),2),(2 \cos (4 \pi  t),2,2)\}.$$
It can be  easily checked that for each of the 3 possibilities there is a Fuchsian potential, unique up to conjugation with $C$, $D$, and $CD,$ satisfying the symmetry assumptions \eqref{symmetry-fix4} and inducing the corresponding totally reducible representation.
\end{proof}   
Using the above proposition the reality condition \eqref{DPWreality} on the potentials $\eta_t$ is equivalent to 
$$s_{jk}=s_{jk}^*\quad \text{ for } (j,k)\in\{(1,2),(1,3),(2,3)\},$$ with $()^*$ as defined in \eqref{fstar}.
The goal of this section is thus to solve the following Monodromy Problem 
\begin{equation}
\label{eq:monodromy-problem1}
\begin{cases}
s_{jk}=s_{jk}^*\quad \text{ for } (j,k)\in\{(1,2),(1,3),(2,3)\}\\
{\displaystyle \sum_{j=1}^3} x_j(t)^2=1
\end{cases}
\end{equation}
using a similar implicit function theorem argument as in \cite{HHT1, HHT2}.
By Lemma \ref{lem:sufsec} and Proposition \ref{Pro:Friecke} we then obtain real holomorphic sections of
$\mathcal M_{DH}(\C P^1,p_1+\dots+p_4,t).$
It then remains to show that the so-constructed real holomorphic sections are negative and in the connected component of twistor lines. This is done in Theorem \ref{computingthecomponent} below.

\subsection{Setup}
As in \cite{HHT2}, consider for a fixed potential $\eta$ the extended frame $\Phi$ satisfying $d\Phi = \Phi \eta$ and $\Phi(z=0)= \Id$. Let $\mathcal P=\Phi(z=1)$ and $\mathcal Q=\Phi(z=i),$ where we omitted the index $t$. Then the traces $s_{jk}$ are given by squares of holomorphic functions in terms of the entries of $\mathcal P= (\mathcal P_{ij})$ and $\mathcal Q = (\mathcal Q_{ij})$ as follows:
\begin{proposition}
\label{prop:traces}
With the notation above we have
\begin{equation}\label{eq:s12}s_{12}=2-4\mathfrak p^2\end{equation}
\begin{equation}\label{eq:s23}s_{23}=2-4\mathfrak q^2\end{equation}
\begin{equation}\label{eq:s13}s_{13}=2-4\mathfrak r^2\end{equation}
with
\begin{equation*}
\mathfrak p=\mathcal P_{11}\mathcal P_{21}-\mathcal P_{12}\mathcal P_{22}, \quad 
\mathfrak q= i(\mathcal Q_{11}\mathcal Q_{21}+\mathcal Q_{12}\mathcal Q_{22})
\end{equation*}
and
\begin{eqnarray*}
\mathfrak r&=&
\frac{i}{2}(\mathcal P_{22} \mathcal Q_{11}+\mathcal P_{12}\mathcal Q_{21})^2
+\frac{i}{2}(\mathcal P_{22} \mathcal Q_{12}+\mathcal P_{12}\mathcal Q_{22})^2
\\&&
-\frac{i}{2}(P_{21} Q_{11}+P_{11}Q_{21})^2
-\frac{i}{2}(P_{21} Q_{12}+P_{11}Q_{22})^2
.\end{eqnarray*}
\end{proposition}
\begin{proof}
The equations
\eqref{eq:s12} and \eqref{eq:s23} have been proven in \cite[Proposition 16]{HHT2}.
It remains to show \eqref{eq:s13}. Using the symmetries $\delta$ and $\tau\circ\delta$ which fix $z=1$ and $z=i$ respectively we have as in \cite[Proposition 16]{HHT2}
\begin{equation}
\begin{split}
\Phi(+\infty)&=\mathcal P C \mathcal P^{-1}C^{-1}, \quad \quad\quad\quad\;\;
\Phi(+i\infty)=\mathcal Q DC \mathcal Q^{-1} C^{-1}D^{-1},\\
M_1&=\Phi(+\infty)\Phi(+i\infty)^{-1}, \quad \quad\quad\quad
M_3=DM_1D^{-1}.
\end{split}
\end{equation}
This gives with $C^2=D^2=-\Id$ and $CD=-DC$ that
$$M_1M_3=-(\mathcal P C\mathcal P^{-1}C^{-1}DC\mathcal Q C^{-1}D^{-1}Q^{-1}D)^2
=-\left(\mathcal P C \mathcal P^{-1} D\mathcal Q CD \mathcal Q^{-1} D\right)^2.$$
For $A\in SL(2,\C)$ we have
$$\tr(-A^2)=2-\tr(A)^2,$$
hence \eqref{eq:s13} holds with
$$\mathfrak r=\frac{1}{2}\tr\left(\mathcal P C \mathcal P^{-1} D\mathcal Q DC \mathcal Q^{-1} D\right)$$
which coincides with the formula given in  Proposition \ref{prop:traces} after a tedious computation.
\end{proof}
\begin{proposition}
\label{prop:derivative-pqr}
At $t=0$, we have
$$\mathfrak p(0)=\mathfrak q(0)=\mathfrak r(0)=0$$
with derivatives
$$\mathfrak p'(0)=2\pi x_3(0),\quad \mathfrak q'(0)=2\pi x_2(0)\and \mathfrak r'(0)=2\pi x_1(0).$$
\end{proposition}
\begin{proof}
At $t=0$ we have $\eta_t=0$ thus all monodromies are trivial and $\mathcal P=\mathcal Q=\Id$ from which the first point follows. For the assertion on the derivatives define
as in \cite{HHT2}
\begin{equation}\label{Omega}
\Omega_j(z)=\int_0^z\omega_j(z)
\end{equation}
 for $j=1,2,3,$ where the integral is computed on the segment from $0$ to $z$.
Then
$$\mathcal P'(0)=\sum_{j=1}^3 x_j(0)\Omega_j(1)\mathfrak m_j, \and \mathcal Q'(0)=\sum_{j=1}^3 x_j(0)\Omega_j(i)\mathfrak m_j,$$
from which we can compute
\begin{equation*}
\begin{split}
\mathfrak p'(0)&=\mathcal P_{21}'(0)-\mathcal P_{12}'(0)=-2 i x_3(0)\Omega_3(1)\\
\mathfrak q'(0)&=i(\mathcal Q_{21}'(0)+\mathcal Q_{12}'(0))=2i x_2(0)\Omega_2(i)\\
\mathfrak r'(0)&=i(\mathcal P_{22}'(0)+\mathcal Q_{11}'(0)-\mathcal P_{11}'(0)-\mathcal Q_{22}'(0))
=-2i x_1(0)(\Omega_1(1)-\Omega_1(i)).
\end{split}
\end{equation*}
By the Residue Theorem, we have for $j=1,2,3$
$$2\pi i=\int_{\gamma_1}\omega_j=\int_0^1\omega_j-\int_0^1\tau^*\omega_j
-\int_0^i\omega_j+\int_0^i(\tau\delta)^*\omega_j.$$
Using the symmetries \eqref{eq:symmetry-omega}, this gives
\begin{equation}
\label{eq:easy-integrals}
\Omega_1(1)-\Omega_1(i)=\pi i,\quad
\Omega_2(i)=-\pi i\and
\Omega_3(1)=\pi i
\end{equation}
proving Proposition \ref{prop:derivative-pqr}.
\end{proof}
In view of Proposition \ref{prop:traces}, the monodromy problem \eqref{eq:monodromy-problem1} can be reformulated to be
\begin{equation}
\label{eq:monodromy-problem2}
\begin{cases}
\mathfrak p=\mathfrak p^*\\
\mathfrak q=\mathfrak q^*\\
\mathfrak r=\mathfrak r^*\\
{\displaystyle\sum_{j=1}^3} x_j(t)^2=1
\end{cases}
\end{equation}
and the following proposition implies that it suffices to solve for two of the three traces only.
\begin{proposition}
\label{prop:two-traces}
Assume that $x_1(t),x_2(t),x_3(t)\colon (-\varepsilon, \varepsilon) \rightarrow \mathcal W^{\geq 1}$ are analytic in $t$,
with $x_j(0)=\cvx_j,$ satisfying the following equations for all $t$
\begin{equation}
\label{eq:monodromy-problem3}
\begin{cases}
\mathfrak p=\mathfrak p^*\\
\mathfrak q=\mathfrak q^*\\
{\displaystyle\sum_{j=1}^3} x_j(t)^2=1.
\end{cases}
\end{equation}
Then also $\mathfrak r=\mathfrak r^*$ for all $t$.
Analogous statements hold by cyclicly permutating $\mathfrak{p,q,r}$.
\end{proposition}
\begin{proof}
Let
$$U=s_{12},\quad V=s_{23},\quad
W=s_{13}\quad \text{and} \quad s=\tr({M_k})$$
as in Proposition \ref{Pro:Friecke} satisfying the quadratic equation \eqref{eq:quadratic}
$$Q : = U^2+V^2+W^2+U\,V\,W-2 s^2(U+V+W)+4(s^2-1)+s^4 =0.$$

By substitution of $U=2-4\mathfrak p^2$, $V=2-4\mathfrak q^2$ and
$W=2-4\mathfrak r^2$, $Q $ factors as
$$Q=Q_1 Q_2\with
Q_j=s^2+4(\mathfrak p^2+\mathfrak q^2+\mathfrak r^2-1)+8(-1)^j \mathfrak{pqr}.$$
Since $Q=0$ for all $t$, and $Q_1$, $Q_2$ are analytic functions of $t$, one of them must be identically zero. (We will see in Remark \ref{remark:r''} below that $j=2$.)
The discriminant of  $Q_j,$ considered as a polynomial in the variable $\mathfrak r$ is give by 
$$\Delta=64(1-\mathfrak p^2)(1-\mathfrak q^2)-16 s^2$$
and is independent of $j=1,2$.
Since $\mathcal P$, $\mathcal Q$ are well-defined analytic functions in $t$ (with values in $\mathcal W_a$),
$\mathfrak r$ is a well-defined analytic function in $t$ as well by Proposition \ref{prop:traces}. Therefore, $\Delta$ admits a well-defined square root $\delta$
such that 
$$\mathfrak r=(-1)^{j+1}\mathfrak p\mathfrak q+\frac{\delta}{8}$$
for all $t.$ 
From the hypotheses of the proposition, we have $\Delta=\Delta^*$ hence $\delta^*=\varepsilon \delta,$
where the sign $\sign=\pm 1$ does not depend on $t$ because $\mathfrak r^*$ is a well-defined analytic function in $t$.
Hence
$$\mathfrak r^*=(-1)^{j+1}\mathfrak p\mathfrak q+ \sign\frac{\delta}{8}.$$

The sign $\varepsilon$ can be determined using the first order derivatives at $t=0$. We have
$$\mathfrak r'=2\pi\cvx_1=\frac{\delta'}{8} \and \mathfrak r'^*=2\pi\cvx_1^*=2\pi\cvx_1=\sign\frac{\delta'}{8}.$$
Since $\cvx_1\not\equiv 0$, we obtain $\sign=1$.
Hence $\mathfrak r=\mathfrak r^*$ for all $t$.
\end{proof}
\subsection{Solving the Monodromy Problem}
\label{section:monodromy:IFT}
As remarked before, see Remark \ref{rem:inidata}, the potentials $d+\eta_t$ are not uniquely determined by 
the reality conditions \eqref{eq:monodromy-problem1} or \eqref{eq:monodromy-problem2}, but also depend on
the Higgs data.
By Proposition \ref{pro:classparahiggs}, we identify the moduli space of strongly parabolic Higgs fields on $V= \mathcal O \oplus \mathcal O$, up to taking the quotient by
$\Z_2\times\Z_2$ action, with the completion of the nilpotent orbit in $\mathfrak{sl}(2, \C),$ which is in turn given by the blow-up of $\C^2/\Z_2$ at the origin. First, consider  the regular case of  $(u, v) \in\C^2\setminus\{(0,0)\}$ and consider for $j=1,2,3$ the quadratic polynomials
$$P_j(\lambda)=\lambda\cvx_j(\lambda)=\cvx_{j,1}\lambda^2+\cvx_{j,0}\lambda-\overline{\cvx_{j,1}},$$
where the $\cvx_j$ are defined in \eqref{eq:defcvxj} satisfying \eqref{eq:detcvxj}, \eqref{eq:ambmcm}, \eqref{eq:a0b0c0} and \eqref{eq:rho0}, and denote the  discriminant of $P_j$  by
$$\Delta_j=\cvx_{j,0}^2+4 |\cvx_{j,1}|^2\in\R.$$
The following arguments work for both signs of $\rho$ in \eqref{eq:rho0}, compare with Lemma \ref{lem:signmean} below. When
statements do depend on the choice of sign, e.g., in Theorem \ref{blowuplimit}, we state them for the correct choice \eqref{eq:rho} only.

\begin{proposition}\label{prop:rootPj}
$\;$\\
\vspace{-0.6cm}
\begin{enumerate}
\item The polynomial $P_j$ has a complex root $\mu_j$ with $|\mu_j|<1$ if and only if $\cvx_{j,0}\neq 0$. In this case $\mu_j \in \mathbb D_1$ depends real-analytically on $(u,v) \neq (0,0)$.
\item With the same notation,
$P_k(\mu_j)$ and $P_{\ell}(\mu_j)$ are $\R$-independent complex numbers, if  $\{j,k,\ell\}=\{1,2,3\}$.
\end{enumerate}
\end{proposition}
\begin{proof}
If $\cvx_{j,1}=0,$ the first point is trivial, because in this case $P_j=\cvx_{j,0} \lambda$
and $\cvx_{j,0}\neq 0$ since $(u,v)\neq(0,0)$.
Hence assume in the following $\cvx_{j,1}\neq 0$.
If $\mu$ is a root of $P_j$, then $\mu\neq 0$. Moreover, $\cv{x}_j=\cv{x}_j^*$ gives that 
$-\tfrac{1}{\overline{\mu}}\neq \mu$ is the other root of $P_j$. Thus
$$|\mu|=1\Leftrightarrow \mu-\frac{1}{\overline{\mu}}=0\Leftrightarrow \cvx_{j,0}=0.$$
For $x_{j, 0} \neq 0$ the root $\mu_j$ with $|\mu_j|<1$ is given  by
$$\mu_j=\frac{-\cvx_{j,0}+\operatorname{sign}(\cvx_{j,0})\sqrt{\Delta_j}}{2 \cvx_{j,1}},$$
which we can rewrite as
$$\mu_j=\frac{2\,\overline{\cvx_{j,1}}}{\cvx_{j,0}}f\left(\frac{4|\cvx_{j,1}|^2}{\cvx_{j,0}^2}\right)
\with
f(z)=\frac{\sqrt{1+ z}-1}{z}.$$
The function $f$ extends holomorphically to $z=0$, therefore $\mu_j$ depends analytically on
$(u,v)$. 
To prove the second point, assume for simplicity of notation that $j=1$.
Suppose by contradiction that $P_2(\mu_1)$ and $P_3(\mu_1)$ are linearly dependent over $\R$.
First assume that $\mu_1\neq 0$.
Then the complex numbers $\cvx_2(\mu_1)$, $\cvx_3(\mu_1)$
are linearly dependent over $\R$. Moreover, $\cvx_1(\mu_1)=0$ so
$$\cvx_2(\mu_1)^2+\cvx_3(\mu_1)^2=1$$
and this implies that $\cvx_2(\mu_1)$ and $\cvx_3(\mu_2)$ are real.
Since all $\cvx_{k,0}$ are real, we obtain
$$\cvx_{k,1}\mu_1-\overline{\cvx_{k,1}}\mu_1^{-1}\in\R
\quad\text{for $k=1,2,3$}.$$
Then
$$\sum_{k=1}^3\left(\cvx_{k,1}\mu_1-\overline{\cvx_{k,1}}\mu_1^{-1}\right)^2\geq 0.$$
Expanding the squares and using $\sum_{k=1}^3 \cvx_{k,1}^2=0$ we obtain
$$\sum_{k=1}^3|\cvx_{k,1}|^2\leq 0$$
which implies $u=v=0$ which is a contradiction.

\noindent
If $\mu_1=0$, we have $P_2(0)=\cvx_{2,-1}$
and $P_3(0)=\cvx_{3,-1}$. From $P_1(0)=\cvx_{1,-1}=0$ we obtain
$$\cvx_{2,-1}^2+\cvx_{3,-1}^2=0$$
and since $\cvx_{2,-1}$ and $\cvx_{3,-1}$ are linearly dependent over $\R$, $\cvx_{2,-1}=\cvx_{3,-1}=0$
which again results in $u=v=0$ leading to a contradiction.
\end{proof}

\begin{theorem}
\label{thm:IFT}
Let $(u,v)\neq (0,0)$ be fixed.  
This determines $\cvx_{j,-1}$ for $j=1, 2, 3$. Then, there are $\epsilon_0>0$ and $a >1$ such that there exists unique values of the parameters $x(t)=(x_1(t),x_2(t),x_3(t)) \in (\mathcal W_a^{\geq -1})^3$ in a neighborhood of
$\cvx$ for all $t\in (-\epsilon_0,\epsilon_0)$ depending real analytically on $(t,p,u,v)$ solving \eqref{eq:monodromy-problem1} with $x(0)=\cvx$ and prescribed $x_{j,-1}(t)=\cvx_{j,-1}$.
Moreover, $\epsilon_0$  and $a>1$ are uniform with respect to $(p,u,v)$ on compact subsets of
$\C^+\times\C^2\setminus\{(0,0)\}$. \end{theorem}

\begin{proof}
Fix $(u,v) \neq (0,0)$.
By Proposition \ref{prop:rootPj} and \eqref{eq:a0b0c0}, at least one of the polynomials $P_j$ has a root $\mu_j$ inside the unit $\lambda$-disc. By symmetry of the roles played by the parameters $\mathfrak p$, $\mathfrak q$, $\mathfrak r$, we may assume without loss of generality that $j=1$.
By Proposition \ref{prop:two-traces}, it suffices to solve Problem \eqref{eq:monodromy-problem3}.
We fix $a>1$ such that $a|\mu_1|<1$ and consider the corresponding function space
$\mathcal W^{\geq 0}_a$. We
introduce a parameter $y=(y_1,y_2,y_3)$ in a neighborhood of $0$ in
$(\mathcal W^{\geq 0}_a)^3$ and set
\begin{equation}\label{eq:defxkyk}x_k(t)=\cvx_k+y_k,\quad k=1,2,3.\end{equation}
Note that the negative part of the potential is fixed to its initial value $x_{k,-1}=\cvx_{k,-1}.$
Since $\mathfrak p$ and $\mathfrak q$ are analytic functions of $(t,y)$
which vanish at $t=0$, the functions
$$\wh{\mathfrak p}(t,y):=\frac{1}{t}\mathfrak p(t,y)
\and
\wh{\mathfrak q}(t,y):=\frac{1}{t}\mathfrak q(t,y)$$
extend analytically at $t=0$ and by Proposition \ref{prop:derivative-pqr} we have at $t=0$
\begin{equation}\label{eq:whmathfrakpqt0}\wh{\mathfrak p}(0,y)=2\pi(\cvx_3+ y_3)
\and
\wh{\mathfrak q}(0,y)=2\pi(\cvx_2+y_2).\end{equation}
We define
\begin{equation}
\begin{split}
\mathcal F(t,y)&:=\wh{\mathfrak p}(t,y)-\wh{\mathfrak p}(t,y)^*\\
\mathcal G(t,y)&:=\wh{\mathfrak q}(t,y)-\wh{\mathfrak q}(t,y)^*\\
\mathcal K(y)&:=\sum_{k=1}^3 x_k(t)^2=\sum_{k=1}^3(\cvx_k+y_k)^2.
\end{split}
\end{equation}
Then solving Problem\eqref{eq:monodromy-problem3} is equivalent to solving the equations $\mathcal F=\mathcal G=0$ and $\mathcal K=1$.
By our choice of the central value \eqref{eq:defcvxj} and due to \eqref{eq:whmathfrakpqt0}, these equations hold at $(t,y)=(0,0)$.
By definition we have $\mathcal F^*=-\mathcal F$ so $\mathcal F=0$ is equivalent to
$\mathcal F^+=0$ and $\Im(\mathcal F^0)=0$, and the analogous statement holds for $\mathcal G$ as well.
By Proposition \ref{prop:derivative-pqr} and substituting \eqref{eq:defxkyk}
the partial derivatives with respect to $y=(y_1,y_2,y_3)$ are
\begin{equation}\label{eq:partialderFGplus}
\begin{split}
d\mathcal F(0,0)^+&=2\pi  dy_3^+\\
\Im(d\mathcal F(0,0)^0)&=2\pi \Im(d y_{3,0})\\
d\mathcal G(0,0)^+&=2\pi  d y_2^+\\
\Im(d\mathcal G(0,0)^0)&=2\pi \Im(d y_{2,0}).
\end{split}
\end{equation}
Hence, the partial derivative of 
$$\left(\mathcal F^+,\mathcal G^+,\Im(\mathcal F^0),\Im(\mathcal G^0)\right)$$ 
with respect to 
$$\left(y_2^+,y_3^+, \Im(y_{2,0}),\Im(y_{3,0})\right)$$ is an automorphism
of $(\mathcal W^+_a)^2\times\R^2$.
The implicit function theorem therefore  uniquely determines 
$(y_2^+,y_3^+, \Im(y_{2,0}),\Im(y_{3,0}))\in(\mathcal W^+_a)^2\times\R^2$ as analytic functions of $t$
and the remaining parameters $y_1$, $\Re(y_{2,0})$ and $\Re(y_{3,0})$.
Furthermore, the partial derivative of $(y_2^+,y_3^+, \Im(y_{2,0}),\Im(y_{3,0}))$ at $(t,y)=(0,0)$ with respect to these remaining parameters is zero by \eqref{eq:partialderFGplus}.

It remains to solve the equation $\mathcal K=1$.
We write the Euclidean division of the polynomial $P_k$ by $(\lambda-\mu_1)$ as
$$P_k(\lambda)=(\lambda-\mu_1)Q_k+P_k(\mu_1)$$
with $Q_k\in\C[\lambda]$.
Note that the $Q_k$ are real analytic in $(u,v)$ by Proposition \ref{prop:rootPj}.
Observe that since $$\sum_{k=1}^3\cvx_{k,-1}^2=0,$$ $\mathcal K$ has no $\lambda^{-2}$ term so
$\lambda\mathcal K\in\mathcal W^{\geq 0}$.
We write the division of $\lambda\mathcal K$ by $(\lambda-\mu_1)$ as
$$\lambda\mathcal K=(\lambda-\mu_1)\mathcal S+\mathcal R$$
where $\mathcal R\in\C$ and $\mathcal S\in\mathcal W^{\geq 0}$.
Note that since $|\mu_1|<1$, $\mathcal R$ and $\mathcal S$ are analytic functions of all parameters by \cite[Proposition 5]{HHT2}.
We have, since $P_1(\mu_1)=0$ and $dy_k=\Re(dy_{k,0})$ for $k=2,3$
\begin{eqnarray*}
d(\lambda\mathcal K)(0,0)&=&\sum_{k=1}^3 2P_k(\lambda)\,dy_k=2(\lambda-\mu_1)Q_1d y_1+\sum_{k=2}^3 2\big((\lambda-\mu_1)Q_k+P_k(\mu_1)\big)\Re(dy_{k,0})\end{eqnarray*}
so by uniqueness of the division
\begin{equation*}
\begin{split}
d\mathcal R(0,0)&=2P_2(\mu_1)\Re(d y_{2,0})+2P_3(\mu_1)\Re(d y_{3,0})\\
d\mathcal S(0,0)&=2 Q_1d y_1 +2 Q_2\Re(d y_{2,0})+2 Q_3\Re(d y_{3,0}).
\end{split}
\end{equation*}
If $\cvx_{1,1}\neq 0$, we have $\mu_1\neq 0$ and the other root of $P_1$ is
$-1/\overline{\mu_1}$ so
$$Q_1=\cvx_{1,1}\left(\lambda+\frac{1}{\overline{\mu_1}}\right)$$
is invertible in $\mathcal W^{\geq 0}_a$ because $\frac{1}{|\mu_1|}>a$.
(Note that $\lambda-c$ is invertible in $\mathcal W^{\geq 0}_a$
if and only if $|c|>a$.)

If $\cvx_{1,1}=0$, we have $$P_1=\cvx_{1,0}\lambda$$ so $Q_1=\cvx_{1,0}\in\C^*$ is invertible
in $\mathcal W^{\geq 0}_a$ as well.
By Proposition \ref{prop:rootPj}, $P_2(\mu_1)$ and $P_3(\mu_1)$ are $\R$-independent complex numbers, and
 the partial derivative  of $(\mathcal S,\mathcal R)$ with respect to
$\left(y_1,\Re(y_{2,0}),\Re(y_{3,0})\right)$  at $(t,y)=(0,0)$ is an isomorphism from
$\mathcal W^{\geq 0}_a\times\R^2$ to $\mathcal W^{\geq 0}_a\times\C$.
Thus, the implicit function theorem uniquely determines $(y_1$, $\Re(y_{2,0}))$ and $\Re(y_{3,0})$
as analytic functions of $t$ in a neighborhood of $t=0$.
\end{proof}
\subsection{The limit $(u,v)\to (0,0)$}
In this section, we highlight the dependency of the solutions $x(t)=x(t,p,u,v)$ provided by Theorem \ref{IFT} on the parameters $p,u,v$ when necessary,
and we drop the dependency of all parameters when convenient. 

By Proposition \ref{pro:classparahiggs} the Higgs bundle moduli space is given by the blow up of $\C^2/\Z_2$ at the origin. Rather than expecting the solution $x(t,p,u,v)$ to extend continuously to $(u,v)=(0,0)$, the limit $(u,v)\to (0,0)$ should therefore depend on the direction in the blow-up.
We write
$$u=r\wt{u}\and v=r\wt{v}\with
|\wt{u}|^2+|\wt{v}|^2=1.$$
Let $0<|r|\leq \frac{1}{2}$ so that $(u,v)\neq (0,0)$.
With a slight abuse of notation, we write $x=x(t,p,r,\wt{u},\wt{v})$.
Recall that $(u,v)\to (-u,-v)$ does not change the initial value $\cvx$, so the map
$x(t,p,r,\wt{u},\wt{v})$ is even with respect to $r$.

Our goal is to prove 
\begin{theorem}\label{blowuplimit}$\;$\\\vspace{-0.5cm}
\begin{enumerate}
\item There exists $\epsilon_2>0$ such that for $|t|<\epsilon_2$, the function $x(t,p,r,\wt{u},\wt{v})$ extends analytically 
to $r=0$. Moreover, $\epsilon_2$ is uniform with respect to
$p$ in compact subsets of $\C^+$ and $(\wt{u},\wt{v})$ in $\S^3$.
\item At $r=0$, $x(t,p,0,\wt{u},\wt{v})$ does not depend on $\lambda$ and solves the following problem:
\begin{equation}
\label{monodromy-problem-r0}
\begin{cases}
\exists U\in SL(2,\C),\quad\forall k,\;U M_k U^{-1}\in SU(2)\\
x_1^2+x_2^2+x_3^2=1.
\end{cases}
\end{equation}
\item At $(t,r)=(0,0)$, we have for $\rho>0$
\begin{equation}\label{eq:blowuplimit}x_1=(|\wt{u}|^2-|\wt{v}|^2),\quad
x_2=2\,\Re\left(\wt{u}\,\overline{\wt{v}}\right)\and
x_3=2\,\Im\left(\wt{u}\,\overline{\wt{v}}\right).\end{equation}
\end{enumerate}
\end{theorem}
\begin{remark}
For $\rho<0$, we have to replace in $(3)$ the limit $(x_1, x_2, x_3)$ by $-(x_1, x_2, x_3)$.
 \end{remark}
\begin{proof}[Proof of  Theorem \ref{blowuplimit}]
Rewrite the central value $\cvx$ in terms of $(r,\wt{u},\wt{v})$ as
$$\cvx_j=r^2\wtcv{x}_{j,-1}\lambda^{-1}+\wt{\rho}(r)\wtcv{x}_{j,0}+r^2\wtcv{x}_{j,1}\lambda$$
with
$$\wtcv{x}_{1,-1}=\wt{u}\,\wt{v},\quad
\wtcv{x}_{2,-1}=\tfrac{1}{2}(\wt{v}^2-\wt{u}^2),\quad
\wtcv{x}_{3,-1}=\tfrac{i}{2}(\wt{u}^2+\wt{v}^2),$$
$$\wtcv{x}_{1,0}=|\wt{u}|^2-|\wt{v}|^2,\quad
\wtcv{x}_{2,0}=2\,\Re\left(\wt{u}\,\overline{\wt{v}}\right),\quad
\wtcv{x}_{3,0}=2\,\Im\left(\wt{u}\,\overline{\wt{v}}\right),$$
$$\wtcv{x}_{j,1}= - \overline{\wtcv{x}_{j,-1}}$$
and
$$\wt{\rho}(r):=r^2\rho(r)=\sqrt{1+r^4} = 1+ O(r^4).$$

Observe that
\begin{equation}
\label{eq:wtcvx}
\sum_{k=1}^3\wtcv{x}_{k,-1}^2=0,\quad
\sum_{k=1}^3\wtcv{x}_{k,0}^2=1,\quad
\sum_{k=1}^3\wtcv{x}_{k,-1}\wtcv{x}_{k,0}=0
\end{equation}
and
\begin{equation}
\label{eq:wtcvxmodule}
\sum_{k=1}^3|\wtcv{x}_{k,-1}|^2=\frac{1}{2}.
\end{equation}
For given $(\wt{u},\wt{v})\in\S^3$, fix $j\in\{1,2,3\}$ such that $\wtcv{x}_{j,0}\neq 0$.
(This is possible by \eqref{eq:wtcvx}).
We take as ansatz that the parameter $y_j$ is of the following form
\begin{equation}
\label{eq:ansatz-wtyj}
y_j=y_{j,0}+r^2\wt{y}_j^+\with \wt{y}_j^+\in\mathcal W^+_a.
\end{equation}
We solve the equations $\mathcal F=\mathcal G=0$ using the implicit function theorem as in Section \ref{section:monodromy:IFT}. This uniquely determines the parameters $y_k^+$ and $\Im(y_{k,0})$ with $k\neq j$ as functions of $(t,r)$ and the remaining parameters $y_{j,0}$, $\wt{y}_j^+$
and $\Re(y_{k,0})$ with $k\neq j$. Moreover, at $r=0$, $y_j^+=0$ and  $y_k^+=0$, $k\neq j$, solve the equation $\mathcal F^+=\mathcal G^+=0$. 
Since $y_k^+$ is an even function of $r$ (due to uniqueness), this means that we can also write for $k\neq j$
$$y_k=y_{k,0}+r^2\wt{y}_k^+\with \wt{y}_k^+\in\mathcal W^+.$$
We decompose $\mathcal K=x_1^2+x_2^2+x_3^2$ as
$$\mathcal K=\mathcal K_{-1}\lambda^{-1}+\mathcal K_0+\mathcal K^+
\with \mathcal K^+\in\mathcal W^+.$$
When $r=0$, $\mathcal K$ does not depend on $\lambda$, i.e.,  $\mathcal K_{-1}=0$ and $\mathcal K^+=0$. Thus since $\mathcal K_{-1}$ and $\mathcal K^+$ are even functions of $r$, this
means that $\wt{\mathcal K}_{-1}=r^{-2}\mathcal K_{-1}$ and
$\wt{\mathcal K}^+=r^{-2}\mathcal K^+$ extend analytically at $r=0$.
More explicitly,
\begin{eqnarray*}
\mathcal K&=&\sum_{k=1}^3\left(r^2\wtcv{x}_{k,-1}\lambda^{-1}+\wt{\rho}\,\wtcv{x}_{k,0}+r^2\wtcv{x}_{k,1}\lambda+y_{k,0}+r^2\wt{y}_k^+\right)^2\\
&=&\sum_{k=1}^3\left(\wtcv{x}_{k,0}+y_{k,0}\right)^2+2 r^2\left(\wtcv{x}_{k,0}+y_{k,0}\right)\left(\wtcv{x}_{k,-1}\lambda^{-1}+\wtcv{x}_{k,1}\lambda+\wt{y}_k^+\right)
+O(r^4)
\end{eqnarray*}
from which we obtain at $r=0$ 
\begin{equation*}
\begin{split}
\wt{\mathcal K}_{-1}\mid_{r=0}&=2\sum_{k=1}^3\left(\wtcv{x}_{k,0}+y_{k,0}\right)\wtcv{x}_{k,-1}\\
\mathcal K_0\mid_{r=0}&=\sum_{k=1}^3\left(\wtcv{x}_{k,0}+y_{k,0}\right)^2\\
\wt{\mathcal K}^+\mid_{r=0}&=2\sum_{k=1}^3\left(\wtcv{x}_{k,0}+y_{k,0}\right)\left(\wtcv{x}_{k,1}\lambda+\wt{y}_k^+\right).
\end{split}
\end{equation*}
In particular, at the central value $y=0$, we have by Equation \eqref{eq:wtcvx} that $\mathcal K_0=1$, $\wt{\mathcal K}_{-1}=0$ and $\wt{\mathcal K}^+=0$ and the differentials with respect to $y$ at $(t,r,y)=(0,0,0)$ are
\begin{equation*}
\begin{split}
d\wt{\mathcal K}_{-1}&=2\sum_{k=1}^3 \wtcv{x}_{k,-1} dy_{k,0}\\
d\mathcal K_0&=2\sum_{k=1}^3\wtcv{x}_{k,0}dy_{k,0}\\
d\wt{\mathcal K}^+&=2\sum_{k=1}^3\wtcv{x}_{k,0}d\wt y_k^+ +\lambda\wtcv{x}_{k,1}dy_{k,0}.
\end{split}
\end{equation*}
Keep in mind that $\wt y_k^+$ and $\Im(y_k^0)$ for $k\neq j$ have already been determined and their differential 
with respect to the remaining parameters  $y_{j,0}$, $\wt{y}_j^+$
and $\Re(y_{k,0})$ with $k\neq j$, and are zero at $(t,r)=(0,0)$.
Consider the polynomials
$\wt P_k(\lambda)=\wtcv{x}_{k,0}\lambda+\wtcv{x}_{k,-1}.$ Recall that $\wtcv{x}_{j,0}\neq0$,
and let $\wt\mu_j$ be the root of $P_j$.
Then
$$d\wt{\mathcal K}_{-1}+\wt\mu_j d\mathcal K_0=2\sum_{k\neq j}\wt P_k(\wt\mu_j)\Re(d y_{k,0}).$$
By Claim \ref{claim:wtPk} below, this is an isomorphism from $\R^2$ to $\C$.
Hence the partial derivative of $(\wt{\mathcal K}_{-1},\mathcal K_0,\wt{\mathcal K}^+)$
with respect to $\big((\Re(y_{k,0}))_{k\neq j},y_j^0,\wt y_j^+\big)$ is an isomorphism from
$\R^2\times\C\times\mathcal W^+$ to $\C^2\times\mathcal W^+$.
The implicit function theorem then gives that there  exists $\epsilon_1>0$ such that for
$|t|<\epsilon_1$ and $|r|<\epsilon_1$, there exists unique values of
the parameters which solve Problem \eqref{eq:monodromy-problem1}.
When $\epsilon_1/2\leq |r|\leq 1/2$, Theorem \ref{thm:IFT} gives us a uniform $\epsilon_0>0$
such that for $|t|<\epsilon_0$, there exists unique values of the parameters which solve
Problem \eqref{eq:monodromy-problem1}. They certainly satisfy the ansatz \eqref{eq:ansatz-wtyj} since $|r|\geq \epsilon_1/2$.
Take $\epsilon_2=\min(\epsilon_1,\epsilon_0)$. In the overlap $\epsilon_1/2\leq |r|<\epsilon_1$, the solutions agree for
$|t|<\epsilon_2$ by uniqueness of the implicit function theorem.

So far, we have proven point (1), and point (3) follows by construction. Since for $r=0,$ the connections are independent of $\lambda$, the reality condition \eqref{eq:monodromy-problem1} implies that the traces $s_{jk}$ are real. Thus, the monodromy representation is either conjugated to a real representation or to a unitary representation.
Using (3), it immediately follows that the traces  for all $(j,k)\in\{(1,2),(1,3),(2,3)$ 
$s_{jk}\in[-2,2]$ for small $t$, which implies unitarity. 
\end{proof}
\begin{claim}
\label{claim:wtPk}
For $k\neq j$ the two complex numbers $\wt P_k(\wt\mu_j)$  are linearly independent over $\R$.
\end{claim}
\begin{proof}
Assume for the simplicity of notation that $j=1$ and that 
$\wt P_2(\wt\mu_1)$ and $\wt P_3(\wt\mu_1)$ are linearly dependent.
If $\mu_1\neq 0$, let
$$\alpha_k=\wt\mu_1^{-1}\wt P_k(\wt\mu_1)=\wtcv{x}_{k,0}+\wtcv{x}_{k,-1}\wt\mu_1^{-1}.$$
Using Equation \eqref{eq:wtcvx} we have
$$\sum_{k=1}^3 \alpha_k^2=1.$$
Since $\alpha_1=0$ and $\alpha_2$, $\alpha_3$ are linearly dependent,
all $\alpha_k$ must be real numbers. Since $\wtcv{x}_{k,0}$ are real, 
each $\wtcv{x}_{k,-1}\wt\mu_1^{-1}$ must be real. Then
using Equation \eqref{eq:wtcvx} again, we have
$$\sum_{k=1}^3(\wtcv{x}_{k,-1}\wt\mu_1^{-1})^2=0$$
implying that $\wtcv{x}_{k,-1}=0$ for all $k$ contradicting
Equation \eqref{eq:wtcvxmodule}. Thus, $\wt P_2(\wt\mu_1)$ and $\wt P_3(\wt\mu_1)$ must be linearly independent, if $\mu_1 \neq0.$

If $\mu_1=0$, then $\wtcv{x}_{1,-1}=0$.  Then, if $\wtcv{x}_{2,-1}$, $\wtcv{x}_{3,-1}$ are linearly
dependent, we would obtain
$$\sum_{k=1}^3(\wtcv{x}_{k,-1})^2=0,$$
which again gives $\wtcv{x}_{k,-1}=0$ for all $k$ contradicting Equation \eqref{eq:wtcvxmodule}.
\end{proof}

\begin{remark}
The constructed real holomorphic sections are uniquely determined by the residue $A_1(t, \lambda)$ at $z= p$ of the  potential $\eta_t$ (when fixing the sign in \eqref{eq:rho0}). The deformation of $A= A_1$ in the parameter $t$ can in fact be expressed by a Lax pair type equation
$$A' = [A, X],$$
for some $X\in  \Lambda^+ \mathfrak{sl}(2, \C)$.  Moreover, $X$ is unique up to adding $g \cdot A$, where $g \colon \lambda \mapsto g(\lambda) \in\C$ is holomorphic around $\lambda = 0.$

To see this recall
$$A= \begin{pmatrix} x_1 & x_2 + i x_3 \\ x_2- i x_3 & - x_1 \end{pmatrix}$$
satisfies $\det A= -1$ and thus $A^2= \Id.$ Therefore, we obtain $A' A + A A' = 0$ which gives 
$$x_1' x_1 + x_2'x_2  + x_3'x_3 = 0.$$
 Let
$$ A'= \begin{pmatrix} x_1' & x_2'+ i x_3'  \\  x_2' - ix_3' & -x_1' \end{pmatrix} \and X= \begin{pmatrix} \alpha & \beta  \\  \gamma & - \alpha \end{pmatrix}. $$  Then we can choose $X$ to be given by
\begin{equation}
\begin{split}
\alpha &= - \tfrac{1}{2}({x_2 - i x_3})(x_2' + ix_3') + x_1 B \\
\beta &= \tfrac{1}{2}x_1 (x_2' + ix_3')  + (x_2+ i x_3) B \\
\gamma &= - \tfrac{1}{2}x_1(x_2' - ix_3') + (x_2- i x_3)(B + a). 
\end{split}
\end{equation}

with  $B= -\tfrac{1}{2}(x_{2,0}' + ix_{3,0}') \frac{x_{1,-1}}{x_{2,-1}+i x_{3,-1}}$ chosen to remove the negative powers of $\lambda.$
\end{remark}

\subsection{Twistor lines}\label{sec:twist}
It remains to determine whether the real holomorphic sections constructed in Theorem \ref{thm:IFT} and Theorem \ref{blowuplimit} are actually twistor lines.
This finally determines the sign in
\eqref{eq:rho}. We start with the following observation:
\begin{lemma}\label{lem:signmean}
Denote by $x^{\pm}(t)\in(\mathcal W_a^{\geq -1})^3$ the solutions of the monodromy provided by Theorem \ref{thm:IFT} for the two possible choices \eqref{eq:a0b0c0} of $\rho$ in their initial data \eqref{eq:defcvxj}. 
We then have
\[x^-(-t)(-\lambda)=x^+(t)(\lambda)\]
for all $\lambda\in \mathbb D_a^*.$ 
\end{lemma}
\begin{proof}
This directly follows from the uniqueness part in the implicit function theorem, compare with the proof of \cite[Proposition 24]{HHT2}.
\end{proof}

\begin{remark}
This Lemma shows how to relate the two different choices of the sign for $\rho$ in the initial data by changing the sign for the deformation parameter. Since we will show that only one of the sign choices leads to actual twistor lines, we will restrict to solutions for $t>0$ in the following.
\end{remark}
\begin{lemma}\label{lem:wrongsignlim}
For small $t>0$,
let $x^-_{u,v}(t)\in(\mathcal W_a^{\geq -1})^3$ be the solution of the monodromy problem provided by Theorem \ref{thm:IFT} with initial data \eqref{eq:defcvxj} and $\rho <0$ in \eqref{eq:a0b0c0}, where $(u,v)\in \C^2\setminus\{0\}$ parametrizes the Higgs field via \eqref{eq:ambmcm}.
Let $s^-_{u,v}(t)$ be the corresponding real holomorphic section   of the parabolic Deligne-Hitchin moduli space. Then  the limit 
\[\lim_{r\to 0}s^-_{ru,rv} =s_{[u,v]}\]
is not a real holomorphic section, and $s^-_{ru,rv}$ are not twistor lines.
\end{lemma}
\begin{proof}
By Theorem \ref{blowuplimit}, the families of flat connections extend to $r=0$. This is true for both possible signs of $\rho$, and it can be directly computed that the corresponding limit for negative $\rho$ differs from \eqref{eq:blowuplimit} by an overall sign. Therefore, in the limit $(t,r)\to0$, the quasiparabolic lines (the eigenlines with respect to the
{\em positive} eigenvalue of the residue $A_j$ determined by $x_1, x_2, x_3$) at any fixed $\lambda_0\in\C^*$ are given by
\begin{equation}\label{eq:wrongparastr}\ell_1=(-\bar v,\bar u)^T\C,\quad\ell_2=(\bar u,\bar v)^T\C,\quad \ell_3=(\bar v,\bar u)^T\C,\quad\ell_4=(-\bar u,\bar v)^T\C\,.\end{equation}

On the other hand,
the quasiparabolic lines of $s^-_{ru,rv}$   for $r\neq0$ at $\lambda=0$  are given by  \eqref{eq:quasbef}. Hence, the parabolic structure at $\lambda=0$ differs from \eqref{eq:wrongparastr} for  $(u,v)\neq0$. Therefore, for 
$(u,v)\in\C^2\setminus\{0\}$, the $r \rightarrow 0$ limit cannot exist. 

But this contradicts the fact that the space of twistor lines is open and closed:
if $s^-_{ru,rv}$ would be twistor lines for all $r\neq0,$ the limit would exist and would be a twistor line as well. Thus
$s^-_{ru,rv}$ cannot be a twistor line
for all $(u,v)$.
\end{proof}

Lemma \ref{lem:wrongsignlim} shows that taking the limit $r\to0$ is delicate in general.
But using additional symmetries, we have
\begin{lemma}\label{lem:specialextension}
Consider the 4-punctured sphere given by $p=\exp(\tfrac{\pi i}{4}),$ and let $(\wt u,\wt v)=\tfrac{1}{\sqrt{2}}(1,\exp(\tfrac{\pi i}{4})).$
Consider the solution provided by Theorem \ref{blowuplimit} for $\rho >0$.
Then, we have at $r=0$ that \eqref{eq:blowuplimit} holds for all $t>0$ small enough.
In particular, the corresponding real holomorphic sections parametrised by $r$ and $t$ have a limit for $r\to0$ which is a twistor line.
\end{lemma}
\begin{proof}
The aim is to show that for this special choice of parameters in Theorem \ref{blowuplimit} and \eqref{eq:blowuplimit}
 holds for all $t\in(0,\tfrac{1}{4})$, where the implicit function theorem applies, rather than only at $(r,t) = 0$.

First, we show that the logarithmic connection given by \eqref{eq:blowuplimit}
with singular divisor determined by $p$ and quasiparabolic lines given by $(\wt u,\wt v)$ has unitary monodromy for all $t\in(0,\tfrac{1}{4})$. Due the choice of $p,$ there is an additional symmetry $z\mapsto iz$ which maps the set of singular points into itself. Moreover, the special choice of $(\wt u, \wt v)$ gives that the logarithmic connection is equivariant
with respect to this order four symmetry, and therefore descend to a logarithmic $\mathrm{SL}(2,\C)$ connection $\wt\nabla$ on the 3-punctured sphere. On the $3$-punctured sphere, any monodromy representation is uniquely determined (up to conjugation) by the parabolic weights, and $\wt \nabla$ must be unitary for $t \in (0, \tfrac{1}{4})$ , see \cite[Lemma 2.2]{Fuchs} for details. Moreover, the parabolic structure for the pull-back of $\wt\nabla$ to the 4-punctured sphere is uniquely determined and thus given by $(\wt u, \wt v)$. 

Next, consider for $(\wt u,\wt v)=\tfrac{1}{\sqrt{2}}(1,\exp(\tfrac{\pi i}{4}))$ and  $r\neq0$ the solution $(x_1,x_2,x_3)=x(t,p,r,\wt{u},\wt{v})$
provided by Theorem \ref{blowuplimit}.  Then we have
\[(x_1,x_2,x_3)(-\lambda)=-(x_1,x_3,x_2)(\lambda)\]
for all $\lambda\in\C^*$ by the uniqueness part of the implicit function theorem, together with the 
symmetry $z\mapsto iz$. As a consequence, the $r\to0$ limit has the corresponding symmetry as well.
But by Theorem \ref{blowuplimit}, point (2), the limit is given by a unitary logarithmic connection $\nabla$ independent of $\lambda$, symmetric with respect to $z\mapsto iz$. This gives that $\nabla$ descents to the 3-punctured sphere and thus has unitary monodromy for $t\in (0, \tfrac{1}{4})$. Therefore $\nabla$ is the pull-back of $\wt \nabla$, and thus \eqref{eq:blowuplimit} holds for all  small $t\in (0, \tfrac{1}{4})$. 
\end{proof}

\begin{theorem}\label{computingthecomponent}
The real sections constructed in Theorem \ref{thm:IFT} and Theorem \ref{blowuplimit} (for $t>0$ and $\rho>0$ ) are twistor lines.
\end{theorem}
\begin{proof}
For $p=\exp(\tfrac{\pi i}{4})$ and small rational $t$, this follows from Lemma \ref{lem:specialextension}
and Proposition \ref{pro:sufftwist}.  The general case then follows by continuous dependency of real holomorphic sections from the parameters $(p,t,r,\wt u,\wt v)$ (Theorem \ref{blowuplimit}) and $(p,t,u,v)$ (Theorem \ref{thm:IFT}), respectively,
together with Remark \ref{rem:sufftwist} for non-rational weights.
\end{proof}

\section{The hyper-K\"ahler structure and the non-abelian Hodge correspondence}\label{NAHCt=0}

In this section we explicitly describe the (rescaled) metric and the non-abelian Hodge correspondence at the limit $t=0$. We then compute the first order derivatives with respect to $t$ at $t=0.$

\subsection{The non-abelian Hodge correspondence at $t=0$}\label{sectionNAHt0}

The parabolic non-abelian Hodge correspondence on the rank $2$ hermitian bundle $\mathcal V$ is a diffeomorphism that associates to each stable strongly parabolic Higgs pair $(\dbar_{\mathcal V}, \Phi)$ the   logarithmic connection $\nabla^{\lambda=1}$ of the associated family of flat connections. In the case of Fuchsian potentials on a $4$-punctured sphere with parabolic weight $t \in (0, \tfrac{1}{4}), $ the underlying holomorphic structure is trivial and the Higgs pair is given by the strongly parabolic Higgs field $\Psi$ only. Due to the symmetry assumptions, $\Psi$ is fully determined by its residue at $p \in \C^+$ and since $\sum_{j=1}^3 x_{j,-1}^2 = 0,$ this residue is nilpotent.  On the other hand, the connection $1$-form for the flat SL$(2, \C)$-connection $\nabla^{\lambda=1}$ is determined by $A_1(\lambda = 1)$ satisfying $\det A_1 = -1$ and $\tr A_1 = 0$. Hence $A_1$ lies in the SL$(2, \C)$ adjoint orbit of $\begin{smatrix} 1 &0 \\ 0 &-1\end{smatrix}.$

Consider the complex 3-dimensional vector  space $\mathfrak{sl}(2,\C)$ with its
complex bilinear inner product
\[<\xi,\eta>=-\frac{1}{2}\text{tr}(\xi \eta).\]
Then its associated quadratic form is the determinant $\text{det}.$
Decompose $\mathfrak {sl}(2, \C)$ into real subspaces
\[\mathfrak{sl}(2,\C)=\mathfrak{su}(2)\oplus i\mathfrak{su}(2)\]
consisting of the subspace of skew-hermitian ($A = -\bar A^T$) and the subspace of hermitian symmetric $(A = \bar A^T)$ trace-free matrices.
Note that
$<.,.>$ is positive-definite on the 3-dimensional real subspace $\mathfrak{su}(2)$ and negative definite
on $i\mathfrak{su}(2).$ 

\begin{lemma}\label{classical-weierstrass}
There is a diffeomorphism between
the $\mathrm{SL}(2,\C)$ orbit through $\begin{smatrix}1&0\\0&-1\end{smatrix}$
without the hermitian symmetric matrices and the nilpotent $\mathrm{SL}(2,\C)$ orbit in $\mathfrak{sl}(2,\C)$.
\end{lemma}
\begin{proof}
Let $\Psi$ be an element in the nilpotent orbit. Then,
$\Phi=\Psi-\bar\Psi^T$ is skew-hermitian and there is a unique skew-hermitian $N\in\mathfrak{su}(2)$ of length 1 such that
\[<N,\Psi-\bar\Psi^T>=0 , \quad <N,i\Psi+i\bar\Psi^T> = 0\]
and 
\[N\times (\Psi-\bar\Psi^T):= \tfrac{1}{2}[N, \Psi-\bar\Psi^T] = i\Psi+i\bar \Psi^T.\]

Moreover, 
\begin{equation}\label{eq:apsia}A_\Psi:=-\sqrt{1+<\Phi,\Phi>}\, iN+\Phi \in\mathfrak{sl}(2,\C)\end{equation}
has determinant $-1$ and the map 
\[\Psi\longmapsto A_\Psi\] 
is smooth. Note that 
\[A_1(\lambda)=\lambda^{-1}\Psi-\sqrt{1+<\Phi,\Phi>}\, iN- \lambda\bar\Psi^T\] 
is of the form \eqref{initialproblem}. Plugging in the formulas for the Higgs field $\Psi$ in terms of the $(u,v)$ coordinates, $-\sqrt{1+<\Phi,\Phi> } \;iN$ is the constant term of the initial values \eqref{eq:rho0} with $\rho$ being positive. \\

For the converse direction, let $A\in\mathfrak{sl}(2,\C)$ be of determinant $-1$, such that $A$ is not hermitian symmetric, i.e.,
\[\Phi:=\frac{1}{2}(A-\bar A^T) \neq 0.\]
Since $\det(A)=-1$ and $<.,.>$ is positive definite on $\mathfrak{su}(2)$,
also 
\[\xi:=\tfrac{1}{2}(A+\bar A^T)\]
does not vanish. Define
\[N:=\frac{i}{\sqrt{-<\xi,\xi>}}\xi\in\mathfrak{su}(2).\]
Note that $<N,\Phi>=0$ since $\det(A)$ is real.
Then,
\[\Psi_A:=\tfrac{1}{2}(\Phi-i N\times \Phi)\]
is nilpotent, and using \eqref{eq:apsia}
\[A_{\Psi_A}=A.\]
Moreover, the map
$A\longmapsto\Psi_A$
is also smooth.
\end{proof}

Recall that we identify the nilpotent orbit with $\C^2 \setminus{(0, 0)}/\Z_2$ via the $2:1$ map
\[ \C^2\setminus\{(0,0)\} \longrightarrow \{\Psi\in\mathfrak{sl}(2,\C)\mid \Psi\neq0;\, \det(\Psi)=0\}, \quad (u,v)\longmapsto \Psi:=\begin{pmatrix} u\,v&-u^2\\v^2&-u\,v\end{pmatrix}.\]

Note that
\[\Psi \begin{pmatrix}u \\ v \end{pmatrix}=0.\]

In the following, we consider the {\em rescaled} space of Higgs bundles, i.e., 
we study Higgs fields of the form $t\Psi$ on parabolic structures with trivial underlying holomorphic bundle and parabolic weight $t>0$ small. On the deRham moduli space side, we consider logarithmic connections of the form $d+t\eta$
which converge to the trivial connection for $t\to0.$ We obtain
\begin{theorem}\label{the:nahdiffeot0}
The diffeomorphism in Lemma \ref{classical-weierstrass} extends to a diffeomorphism
of the blow-up of $\C^2/\Z_2$ at $(0,0)$ which is $T^*\C P^1$, to the
full adjoint $\mathrm{SL}(2,\C)$ orbit through $\begin{smatrix}1&0\\0&-1\end{smatrix}$ (including hermitian symmetric matrices).
This map is the limit of the non-abelian Hodge correspondence for $t \rightarrow 0$ for rescaled strongly parabolic Higgs fields
with  trivial underlying holomorphic bundle.
\end{theorem}

\begin{proof}
For $(u, v) \rightarrow 0$ take $u= r \wt u$ and $v = r \wt v$ with $|\wt u|^2 = |\wt v|^2 = 1$ and $r\in  \R_{>0}$ and consider  $r \rightarrow 0.$  Let $ \Psi(u, v) = r^2  \Psi(\wt u, \wt v)$ be the associated nilpotent matrix.
Then the map $\Psi \mapsto A_{\Psi}$ extends to $r=0$ with 
$$\lim_{r \rightarrow 0} A_{\Psi(u,v)}  =- i \wt N$$

with $\wt N \in \mathfrak{su}(2, \C)$  of length $1$ satisfying 
 \[<\wt N,\wt\Psi-\overline{\wt \Psi}^T>=0  \and <\wt N,i\wt \Psi+i\overline{\wt\Psi}^T> = 0.\]

Conversely, all hermitian symmetric matrices of determinant $-1$ can be realized as a limit. It follows from Theorem \ref{thm:IFT} and Theorem \ref{blowuplimit} that this map is the non-abelian Hodge correspondence for $t \rightarrow 0$.
 \end{proof}

\subsection{The rescaled metric at $t=0$}\label{sec:res-metric}

{By Theorem \ref{the:nahdiffeot0} the rescaled Higgs bundle moduli space at $t=0$ is the completion of the nilpotent $\mathrm{SL}(2,\C)$ orbit, identified with the blow-up of $\C^2 \setminus\{(0,0)\}/\Z_2$ at the origin, which is mapped by the limit of the non-abelian Hodge correspondence at $t=0$
to the $\mathrm{SL}(2,\C)$ orbit through $\begin{smatrix} 1&0\\0&-1\end{smatrix}.$ The latter can be interpreted
as the rescaled limit of the Betti moduli space at the identity, compare with the vanishing of $Q_j$ in the proof of
Proposition \ref{prop:two-traces}.
Next,  we show the rescaled limit hyper-K\"ahler metric on the rescaled Higgs bundle moduli space at $t=0$ is the Eguchi-Hanson metric.
The scaling factor $\tfrac{1}{t}$ is chosen so that the central sphere $\C P^1$ in $\mathcal M_{Higgs}$, which is the moduli space of
semi-stable parabolic bundles (i.e., Higgs pairs with vanishing Higgs fields), has constant volume independent of the weight  $t$.
This is in accordance with the scaling factor $k$ in Proposition \ref{pro:sympupdown} when $t=\tfrac{l}{k}$ is rational.

\begin{lemma}\label{rescaledform}
Let $t\sim 0$ and consider the space of solutions $\mathcal M(t)$ provided by Theorem \ref{thm:IFT} and Theorem \ref{blowuplimit}, parametrized by $(u,v) \in \C^2 \setminus \{(0,0)\}$.
In terms of the parameters $x_1(t), x_2(t), x_3(t)$ in \eqref{symmetry-fix1-4}, the rescaled twisted holomorphic symplectic form is given by
\[\varpi= \frac{d x_2(t)\wedge d x_3(t)}{x_1(t)}.\]
\end{lemma}
\begin{proof}
By Proposition \ref{proId} the twisted holomorphic symplectic from is give by the Goldman symplectic form, which can be computed by \eqref{def:symf}. Moreover, since we consider Fuchsian systems, the integral on $\Sigma^0$ in \eqref{def:symf} vanishes and we only need to compute the boundary terms.
By construction all 4 residues $B_j=tA_j$ give the same contribution.
A direct computation shows that 
\begin{equation}\label{kirilov123}\begin{split}
\frac{1}{8\tr(B_1^2)}\tr[(B_1[d B_1, d B_1])= -\tfrac{t}{4} i(x_1 dx_2\wedge dx_3-x_2dx_1\wedge dx_3+x_3 dx_1\wedge dx_2).
\end{split}
\end{equation}
On the other hand, \[0=x_1 dx_1+x_2 dx_2+x_3 dx_3\] which combines with \eqref{kirilov123}  (together with the rescaling by $\tfrac{1}{t}$) to
\[\varpi=\tfrac{4 i}{t}\frac{1}{8\tr(B_1^2)}\tr[(B_1[d B_1\wedge d B_1])= \frac{dx_2\wedge dx_3}{x_1}\]
proving the lemma.
\end{proof}

For an explicit formula of the rescaled twisted symplectic form in Lemma \ref{rescaledform} at $t=0$
let $A=A_1(\lambda)$ be the residue of the rescaled potential $\tfrac{1}{t}\eta_{t=0}$ (see \eqref{etatttinfront}) at $t=0.$ Recall that $$A=\lambda^{-1}\matrix{uv&-u^2\\v^2&-uv}
+\rho\matrix{|u|^2-|v|^2& 2u\vc\\2\uc v&|v|^2-|u|^2}
+\lambda\matrix{-\uc\vc&-\vc^2\\\uc^2&\uc\vc},$$

i.e., we compute the symplectic form using the Higgs field coordinates $(u,v)$, where  $r^2=|u|^2+|v|^2$ and $\rho = \sqrt{1+r^{-4}}$. 
A direct computation then gives  at $t=0$
\begin{align}
\label{mathematica-omega0}
\varpi=2i&\left(-\frac{r^6+|v|^2}{\rho r^6}du\wedge d\uc
-\lambda^{-1}du\wedge dv
+\frac{\uc v}{\rho r^6} du\wedge d\vc\right.\\
&\left.-\frac{u\vc}{\rho r^6}d\uc\wedge dv
-\lambda d\uc\wedge d\vc
-\frac{r^6+|u|^2}{\rho r^6} dv\wedge d\vc
\right).\nonumber\end{align}
On the other hand,
$$\varpi=\lambda^{-1}(\omega_J+i\omega_K)-2\omega_I-\lambda(\omega_J-i\omega_K)$$
from which we obtain
\begin{equation*}
\begin{split}
\omega_I&=\frac{i}{\rho r^6}\left( (r^6+|v|^2)du\wedge d\uc-\uc v \,du\wedge d\vc+u\vc \,d\uc\wedge dv+(r^6+|u|^2) dv\wedge d\vc\right)\\
\omega_J&=i\left(-du\wedge dv+d\uc\wedge d\vc\right)\\
\omega_K&=-\left(du\wedge dv+d\uc\wedge d\vc\right).
\end{split}
\end{equation*}

\begin{proposition}\label{limitEH}
The rescaled hyper-K\"ahler metric 
of strongly parabolic Higgs bundles on the 4-punctured sphere at $t=0$ is the Eguchi-Hanson space modulo a $\Z_2\times\Z_2$ action.
\end{proposition}
\begin{proof}
Since $(u,v)$ corresponds to Higgs bundle coordinates, we use the complex structure $I$ to compute $g= \omega_I(., I.)$. Consider the tangent space basis
$$\mathcal B = \left( \tfrac{\del }{\del u},  \tfrac{\del }{\del \bar u},  \tfrac{\del }{\del v},  \tfrac{\del }{\del \bar v}\right)$$
Then the complex structure $I$ can be represented by the diagonal matrix
$\mathrm{diag}(i,-i,i,-i)$.
and we obtain 
\begin{eqnarray*}
g&=&\sqrt{1+r^{-4}}\Big[
du\otimes d\uc+d\uc\otimes du +dv\otimes d\vc+d\vc\otimes dv
-\frac{1}{r^2(1+r^4)}\Big(
u\uc(du\otimes d\uc+d\uc\otimes du)\\
&&
+\uc v(du\otimes d\vc+d\vc\otimes du)
+u\vc(dv\otimes d\uc+d\uc\otimes  dv)
+v\vc(dv\otimes d\vc+d\vc\otimes dv)\Big)\Big]
\end{eqnarray*}
which by \cite[Equation 2]{Lye} identifies with the Eguchi-Hanson metric with $n=2$ and $a=1$. 
\end{proof}
Obviously, the Eguchi-Hanson metric is independent of the conformal type of the 4-punctured sphere.
On the other hand, it is known that the conformal type of the underlying 4-punctured sphere can be recovered from the Hitchin metric of the
parabolic Higgs bundle moduli space, see \cite{CVZ} and \cite{FMSW}.
We will see the non-trivial dependence of the metric on the conformal type in the first order approximation of the metric in $t$.

\subsection{First order approximations}\label{firstderivative}
The advantage of our setup is that we obtain, analogously to \cite{HHT2}, an iterative way of computing the power series expansion of the twistor lines leading to an approach towards explicitly computing the non-abelian Hodge correspondence and all involved geometric quantities 
for the case at hand.
In this section, we  compute the first order derivatives of the parameters $(x_1, x_2, x_3)$ and derive in particular first order derivatives of the relative twisted holomorphic symplectic form. This yields derivatives of the non-abelian Hodge correspondence as well as derivatives of the hyper-K\"ahler metric.

\subsubsection{First order derivatives of the parameters}
Define
for  $1\leq j,k\leq 3$ and $\omega_k$, $\Omega_j$ given in \eqref{omega} and  \eqref{Omega}, respectively,
$$\Omega_{jk}(z)=\int_0^z\Omega_j\omega_k.$$
The shuffle relation (see e.g. \cite[Appendix]{HHT2}) then gives
\begin{equation}
\label{eq:shuffle-depth2}
\Omega_j\Omega_k=\Omega_{jk}+\Omega_{kj}.
\end{equation}
Let  $()'$ and $()''$ denote the first and second order derivatives of a quantity with respect to $t$ at $t=0$. Then the following proposition holds.
\begin{proposition}
\label{prop:derivatives}
The first order derivatives $x'_j$ for $j=1,2,3$ are polynomials of degree at most 2 in $\lambda$:
$$x'_j=x'_{j,0}+x'_{j,1}\lambda+x'_{j,2}\lambda^2=x'_{j,0}+(x'_j)^+$$
with the positive parts given by
\begin{equation}
\label{eq:x'pos}
\begin{cases}
(x'_1)^+=\displaystyle\frac{4i}{\pi}\,\Im(\Omega_{21}(1)+\Omega_{31}(i))\,(\cvx_2\cvx_3)^+\\
(x'_2)^+=\displaystyle\frac{-4i}{\pi}\,\Im(\Omega_{31}(i))\,(\cvx_1\cvx_3)^+\\
(x'_3)^+=\displaystyle\frac{-4i}{\pi}\,\Im(\Omega_{21}(1))\,(\cvx_1\cvx_2)^+,
\end{cases}
\end{equation}
and the constant terms given by
\begin{equation}
\label{eq:x'0}
\begin{cases}
x'_{1,0}=\displaystyle\frac{-1}{\rho r^4}\left((|u|^2-|v|^2)X-2\rho u v Y\right)\\
x'_{2,0}=\displaystyle\frac{-1}{\rho r^4}\left(2\,\Re(u\,\overline{v})X-\rho(v^2-u^2)Y\right)\\
x'_{3,0}=\displaystyle\frac{-1}{\rho r^4}\left(2\,\Im(u\,\overline{v})X-\rho i(u^2+v^2)Y\right),
\end{cases}
\end{equation}
where $r^2=|u|^2+|v|^2$, $\rho$ satisfying \eqref{eq:rho} and 
\begin{equation}
\begin{split}
X&=\sum_{j=1}^3\cvx_{j,-1}x'_{j,1}\\
Y&=\sum_{j=1}^3\left(\cvx_{j,-1}x'_{j,2}+\cvx_{j,0}x'_{j,1}\right),
\end{split}
\end{equation}
with $x'_{j,1}$ and $x'_{j,2}$ determined by \eqref{eq:x'pos}.
\end{proposition}

 \begin{remark}
The required $\Omega$ integrals are computed in Proposition \ref{prop:integrals2} to be
$$\Im(\Omega_{21}(1))=2\pi\log\left|\frac{p^2-1}{2p}\right|\and
\Im(\Omega_{31}(i))=-2\pi\log\left|\frac{p^2+1}{2p}\right|.$$

\end{remark}
\begin{proof}

Let $\Phi$ be the fundamental solution of the equation $d\Phi_t=\Phi_t\eta_t$ with $\Phi_t(0) = \Id.$ Recall that $\eta_{t=0} =0$ and therefore $\Phi_{t=0} = \Id.$ Differentiating the equation $d\Phi_t=\Phi_t\eta_t$ twice at $t=0$ we thus obtain
$$d\Phi''=\eta''+2\Phi'\eta',$$
hence
$$\Phi''(z)=\int_0^z(\eta''+2\Phi'\eta').$$
Then we have using the proof of Proposition \ref{prop:derivative-pqr}
$$\eta'=\sum_{j=1}^3 \cvx_j\omega_j\mathfrak m_j, \quad 
\Phi'=2\sum_{j=1}^3 \cvx_j\Omega_j\mathfrak m_j, \quad
\eta''=2\sum_{j=1}^3 x'_j\omega_j\mathfrak m_j.$$

This gives
$$\int_0^z\eta''=2\matrix{x'_1\Omega_1&x'_2\Omega_2+ix'_3\Omega_3\\ x'_2\Omega_2-ix'_3\Omega_3&-x'_1\Omega_1}$$

and 
$$\int_0^z \Phi'\eta'=\matrix{\sum_{j=1}^3\cvx_j^2\Omega_{jj}+i\cvx_2\cvx_3(\Omega_{32}-\Omega_{23})&
\cvx_1\cvx_2(\Omega_{12}-\Omega_{21})+i\cvx_1\cvx_3(\Omega_{13}-\Omega_{31})\\
\cvx_1\cvx_2(\Omega_{21}-\Omega_{12})+i\cvx_1\cvx_3(\Omega_{13}-\Omega_{31})
&\sum_{j=1}^3\cvx_j^2\Omega_{jj}+i\cvx_2\cvx_3(\Omega_{23}-\Omega_{32})}.$$
Using Leibniz rule, the shuffle relation \eqref{eq:shuffle-depth2} and that we solved Equation \ref{eq:easy-integrals}  we obtain
\begin{eqnarray}\label{eq:p''}
\mathfrak p''&=&\mathcal P''_{12}-\mathcal P''_{21}+2(\mathcal P'_{11}\mathcal P'_{21}-\mathcal P'_{12}\mathcal P'_{22})\nonumber\\
&=&-4i x'_3\Omega_3(1)+4\cvx_1\cvx_2\left(\Omega_{21}(1)-\Omega_{12}(1)\right)+4\cvx_1\cvx_2\Omega_1(1)\Omega_2(1)\nonumber\\
&=&4\pi x'_3+8\cvx_1\cvx_2\Omega_{21}(1)
\end{eqnarray}

\begin{eqnarray}\label{eq:q''}
\mathfrak q''&=&i(\mathcal Q''_{21}+\mathcal Q''_{12})+2i(\mathcal Q'_{11}\mathcal Q'_{21}+\mathcal Q'_{12}\mathcal Q'_{22})\nonumber\\
&=&4i x'_2\Omega_2(i)-4 \cvx_1\cvx_3(\Omega_{13}(i)-\Omega_{31}(i))+4\cvx_1\cvx_3\Omega_1(i)\Omega_3(i)\nonumber\\
&=&4\pi x'_2+8\cvx_1\cvx_3\Omega_{31}(i).
\end{eqnarray}

Since we have solved $\mathfrak p(t)=\mathfrak p(t)^*$ for all $t$, we have
$\mathfrak p''=(\mathfrak p'')^*$. Moreover, $\cvx_j=\cvx_j^*,$ hence
\begin{equation}
\label{eq:p''star}
4\pi x'_3-4\pi(x'_3)^*=-16 i\,\cvx_1\cvx_2\Im(\Omega_{21}(1)).
\end{equation}
Projecting onto $\mathcal W^+$ and remembering that $x'\in\mathcal W^{\geq 0},$ we obtain the formula for $(x'_3)^+$ stated in Equation
\eqref{eq:x'pos}, and $(\cvx_1\cvx_2)^+$ is a degree-2 polynomial.
In the same vein $\mathfrak q''=(\mathfrak q'')^*$ gives
\begin{equation}
\label{eq:q''star}
4\pi x'_2-4\pi(x'_2)^*=-16 i\,\cvx_1\cvx_3\Im(\Omega_{31}(i))
\end{equation}
which determines $(x'_2)^+$.

For  $(x'_1)^+$ consider the equation $\mathcal K=1$ which holds for all $t$. Therefore, \begin{equation}
\label{eq:K'}
 \mathcal K'=0=2\sum_{j=1}^3 \cvx_j x'_j.
\end{equation}
Then $\mathcal K'=\mathcal K'^*$ and Equations \eqref{eq:p''star}, \eqref{eq:q''star} give the equation 
\begin{equation}
\label{eq:r''star}
4\pi \cv{x}_1 x'_1-4\pi\cv{x}_1(x'_1)^*=16 i\,\cvx_1\cvx_2\cvx_3\Im(\Omega_{21}(1)+\Omega_{31}(i)).
\end{equation}
Dividing by $\cv{x}_1\not\equiv 0$ and taking the positive part determines $(x'_1)^+$.

To compute the constant terms, we consider the coefficients of $\lambda^{-1}$, $\lambda^0$ and $\lambda$ in $\mathcal K$:
\begin{equation*}
\begin{split}
\mathcal K'_{-1}&=2\sum_{j=1}^3\cvx_{j,1}x'_{j,0}=0\\
\mathcal K'_0&=2\sum_{j=1}^3(\cvx_{j,0}x'_{j,0}+\cvx_{j,-1}x'_{j,1})=0\\
\mathcal K'_1&=2\sum_{j=1}^3(\cvx_{j,1}x'_{j,0}+\cvx_{j,0}x'_{j,1}+\cvx_{j,-1}x'_{j,2})
\end{split}
\end{equation*}
which yield the system of equations
$$\begin{cases}
{\displaystyle \sum_{j=1}^3}\cvx_{j,-1} x'_{j,0}=0\\
{\displaystyle \sum_{j=1}^3}\cvx_{j,0} x'_{j,0}=-X\\
{\displaystyle \sum_{j=1}^3}\cvx_{j,1} x'_{j,0}=-Y
\end{cases}$$
with $X$, $Y$ as in Proposition \ref{prop:derivatives}.
Its determinant simplifies to
$$\det(\cvx_{j,k})_{1\leq j\leq 3,-1\leq k\leq 1}=\frac{i}{2}\rho r^6.$$
Using the Cramer rule, we obtain \eqref{eq:x'0} after simplification.
\end{proof}

\begin{remark}\label{remark:r''}
 The second derivative of $\mathfrak r$ is computed in the following. Through the character variety equation and the formulas for $\mathfrak p''$ and $\mathfrak q''$ it will give rise to an identity between $\Omega$-integrals. Using Proposition \ref{prop:traces} we have
\begin{eqnarray*}
\mathfrak r''&=&i\big[\mathcal P''_{22}+2\mathcal P'_{22}\mathcal Q'_{11}+\mathcal Q''_{11}+2\mathcal P'_{12}\mathcal Q'_{21}-2\mathcal P'_{21}\mathcal Q'_{12}-\mathcal P''_{11}-2\mathcal P'_{11}\mathcal Q'_{22}-\mathcal Q''_{22}\\
&&+(\mathcal P'_{22}+\mathcal Q'_{11})^2
+(\mathcal P'_{12}+\mathcal Q'_{12})^2
-(\mathcal P'_{21}+\mathcal Q'_{21})^2
-(\mathcal P'_{11}+\mathcal Q'_{22})^2\big]
\\
&=&i\big[\mathcal P''_{22}-\mathcal P''_{11}+\mathcal Q''_{11}-\mathcal Q''_{22}
+2(\mathcal P'_{12}-\mathcal P'_{21})(\mathcal Q'_{12}+\mathcal Q'_{21})
+(\mathcal P'_{12})^2-(\mathcal P'_{21})^2+(\mathcal Q'_{12})^2-(\mathcal Q'_{21})^2\big]\end{eqnarray*}
where we used that $\mathcal P'_{11}+\mathcal P'_{22}=\mathcal Q'_{11}+\mathcal Q'_{22}=0$.
This gives
\begin{eqnarray}
\mathfrak r''&=& i\big[-4x'_1\Omega_1(1)+4i \cvx_2\cvx_3(\Omega_{23}(1)-\Omega_{32}(1))
+4x'_1\Omega_1(i)-4i\cvx_2\cvx_3(\Omega_{23}(i)-\Omega_{32}(i))\nonumber\\
&&+8i\cvx_2\cvx_3\Omega_3(1)\Omega_2(i)
+4i\cvx_2\cvx_3\Omega_3(1)\Omega_2(1)
+4i\cvx_2\cvx_3\Omega_2(i)\Omega_3(i)\big]\nonumber\\
&=&4\pi x'_1-8\cvx_2\cvx_3\left (\Omega_{23}(1)+\Omega_{32}(i)+\pi^2 \right)
\label{eq:r''}\end{eqnarray}
using that we solved Equation \ref{eq:easy-integrals} and the shuffle relation \eqref{eq:shuffle-depth2}.
Observe the similarity with \eqref{eq:p''} and \eqref{eq:q''}.
Remember from the proof of Proposition \ref{prop:two-traces}
that $\mathfrak p$, $\mathfrak q$ and $\mathfrak r$ satisfy the character variety equation for all $t$
\begin{equation}
\label{mathematica-fricke}
4\cos(2\pi t)^2 + 4(\mathfrak p^2+\mathfrak q^2+\mathfrak r^2-1)+8(-1)^j \mathfrak{pqr}=0
\end{equation}
with either $j=1$ or $j=2$.
Taking the third order derivative of the above equation yields (since $\mathfrak p (0) = \mathfrak q (0) = \mathfrak r (0) = 0$)
$$\mathfrak{p'p''+q'q''+r'r''}+2(-1)^j\mathfrak{p'q'r'}=0$$
Using Proposition \ref{prop:derivative-pqr} and Equations \eqref{eq:p''}, \eqref{eq:q''}, \eqref{eq:r''} we then obtain
$$8\pi^2(\cvx_1 x'_1+\cvx_2 x'_2+\cvx_3 x'_3)
+16\pi\cvx_1\cvx_2\cvx_3\big(\Omega_{21}(1)+\Omega_{31}(i)-\Omega_{23}(1)-\Omega_{32}(i)-\pi^2+(-1)^j\pi^2\big)=0.$$
The first summand vanishes, as $x_1^2+x_2^2+x_3^2=1$ for all $t$, thus
$$\Omega_{21}(1)+\Omega_{31}(i)-\Omega_{23}(1)-\Omega_{32}(i)-\pi^2+(-1)^j\pi^2=0.$$
In the most symmetric case of the 4-punctured sphere where $p=e^{i\pi/4}$ we have by symmetry
$$\Omega_{31}(i)=-\Omega_{21}(1)\and
\Omega_{32}(i)=-\Omega_{23}(i).$$
Hence $j=2$ and we have proved the following identity that holds for all $p\in\C^+$:
\begin{equation}
\label{eq:Omega-identity}
\Omega_{23}(1)+\Omega_{32}(i)=\Omega_{21}(1)+\Omega_{31}(i).
\end{equation}
\end{remark}

\begin{proposition}
\label{prop:integrals2}
For $p\in\C^+$ we have
$$\Omega_{21}(1)=2\pi i\log\left(\frac{p^2-1}{2ip}\right)\quad\text{and}\quad\Omega_{31}(i)=-2\pi i\log\left(\frac{p^2+1}{2p}\right),$$
where $\log$ denote the principal valuation of the logarithm on $\C\setminus\R^-$.
\end{proposition}
\begin{proof}
We first prove the Proposition for $p=e^{i\varphi}$ with $0<\varphi<\pi/2$. In that case, we have by
\cite[Proposition 35]{HHT2}

\begin{equation}
\begin{split}
\Omega_{21}(1)-\Omega_{12}(1)&=4\pi i\log(\sin\varphi))-i(\pi-2\varphi)\log\left(\frac{1-\cos(\varphi)}{1+\cos(\varphi)}\right)\\
\Omega_{31}(i)-\Omega_{13}(i)&=-4\pi i\log(\cos(\varphi))+2i\varphi\log\left(\frac{1-\sin(\varphi)}{1+\sin(\varphi)}\right).
\end{split}
\end{equation}
On the other hand, by the shuffle product formula and using \eqref{eq:easy-integrals}
$$\Omega_{21}(1)+\Omega_{12}(1)=\Omega_1(1)\Omega_2(1)=i(\pi-2\varphi)\log\left(\frac{1-\cos(\varphi)}{1+\cos(\varphi)}\right)$$
$$\Omega_{31}(i)+\Omega_{13}(i)=\Omega_1(i)\Omega_3(i)=-2i\varphi\log\left(\frac{1-\sin(\varphi)}{1+\sin(\varphi)}\right).$$
Hence
$$\Omega_{21}(1)=2\pi i \log(\sin\varphi)$$
$$\Omega_{31}(i)=-2\pi i\log(\cos(\varphi)$$
proving the result for $p=e^{i\varphi}$.\\

For $p\in\C^+$, both $\frac{p^2-1}{2ip}$ and $\frac{p^2+1}{2p}$ are in $\C\setminus\R^-$ and  both sides of the formulas of Proposition \ref{prop:integrals2} are well-defined holomorphic functions in $p\in\C^+$
which coincide when $p\in\C^+\cap\S^1$, therefore they are equal.
\end{proof}

\subsubsection{First order derivative of the metric}
By Lemma \ref{rescaledform} we have 
$\varpi=\frac{dx_2\wedge dx_3}{x_1}$ for all $t.$
Recall from Proposition \ref{pro:classparahiggs} that the open dense subset of the Higgs bundle moduli space specified by trivial holomorphic bundles is independent of the weight $t$. Thus the complex structure $I$ is independent of $t$ and  also the holomorphic symplectic form $\omega_J + i\omega_K$ are independent of $t$. The same holds true for the
holomorphic symplectic form $\omega_J - i\omega_K$ for the complex structure $-I.$
This gives that the first order derivatives of the twisted holomorphic symplectic form is given by the derivative of its constant term
\begin{equation}\label{varphiconst}\varpi_0= - 2  \omega_I = \frac{1}{\cv{x}_{1,-1}}\left(
\frac{-x_{1,0}}{\cv{x}_{1,-1}}d\cv{x}_{2,-1}\wedge d\cv{x}_{3,-1}
+d x_{2,0}\wedge d\cv{x}_{3,-1}
+d\cv{x}_{2,-1}\wedge x_{3,0}
\right).\end{equation}
Using the formulas for $x'_{1,0}$, $x'_{2,0}$ and $x'_{3,0}$, we obtain

\begin{eqnarray}
\varpi_0'&=&
\frac{4\,i\,\Im(\Omega_{21}(1))}{\pi r^8}\left[\left(-r^4 +3(\uc v+u \vc)^2\right)|v|^2+r^8\left(|u|^2-|v|^2\right)\right]du\wedge d\uc
\label{mathematica-omega1}\\
&+&\frac{4\,i\,\Im(\Omega_{31}(i))}{\pi r^8}\left[\left(r^4 +3(\uc v-u \vc)^2\right)|v|^2+r^8\left(|v|^2-|u|^2\right)\right]du\wedge d\uc
\nonumber\\
&+&\frac{4\,i\,\Im(\Omega_{21}(1))}{\pi r^8}\left[\left(r^4 -3(\uc v+u \vc)^2\right)\uc v+r^8\left(-\uc v-3u\vc\right)\right]du\wedge d\vc
\nonumber\\
&+&\frac{4\,i\,\Im(\Omega_{31}(i))}{\pi r^8}\left[\left(-r^4 -3(\uc v-u \vc)^2\right)\uc v+r^8\left(\uc v-3u\vc\right)\right]du\wedge d\vc
\nonumber\\
&+&\frac{4\,i\,\Im(\Omega_{21}(1))}{\pi r^8}\left[\left(r^4 -3(\uc v+u \vc)^2\right)u \vc+r^8\left(-u \vc-3\uc v\right)\right]dv\wedge d\uc
\nonumber\\
&+&\frac{4\,i\,\Im(\Omega_{31}(i))}{\pi r^8}\left[\left(-r^4 -3(\uc v-u \vc)^2\right)u \vc+r^8\left(u \vc-3\uc v\right)\right]dv\wedge d\uc
\nonumber\\
&+&\frac{4\,i\,\Im(\Omega_{21}(1))}{\pi r^8}\left[\left(-r^4 +3(\uc v+u \vc)^2\right)|u|^2+r^8\left(|v|^2-|u|^2\right)\right]dv\wedge d\vc
\nonumber\\
&+&\frac{4\,i\,\Im(\Omega_{31}(i))}{\pi r^8}\left[\left(r^4 +3(\uc v-u \vc)^2\right)|u|^2+r^8\left(|u|^2-|v|^2\right)\right]dv\wedge d\vc.
\nonumber
\end{eqnarray}
where the $\Omega$-integrals are given by Proposition \ref{prop:integrals2}.
Note that the $(2,0)$ and $(0,2)$ terms $du\wedge dv$ and $d\uc\wedge d\vc$ vanish in accordance with the twistorial description of hyper-K\"ahler metrics \cite{HKLR}.

\subsection{Energy}\label{sec:energy}
The (Dirichlet) energy of a map $f\colon \Sigma\to (M,g)$ from a Riemann surface $\Sigma$ to a Riemannian manifold $(M,g)$ is given by
\[\mathcal E(f)=-\tfrac{1}{2}\int_\Sigma g(df\wedge* df)\]
(where $*$ denotes the Riemann surface $*$-operator). 
If $f$ is equivariant  (with respect to a discrete group acting holomorphically on $\Sigma$ and isometrically on $M$), the energy remains well-defined.
If $(\nabla,\Phi,h)$ is a solution to the self-duality equations, 
the equivariant harmonic map is given by the harmonic metric $f=h$, see for example \cite{Do, Li}.
Furthermore, the energy is then
\begin{equation}\label{def:ener}\mathcal E(h)=2i\int_\Sigma \tr(\Phi\wedge\Phi^*),\end{equation}
which is $L^2$-norm of the Higgs field $\Phi$. We refer to $\mathcal E(h)$ in \eqref{def:ener} as the energy of the self-duality solution $(\nabla,\Phi,h)$ and note that it is a K\"ahler potential for the hyperk\"ahler metric on the moduli space of solutions with respect to the complex structure $J,$ see \cite{Hi1}.

Let $\eta=\eta_t$ be given by the parameters $(x_1,x_2,x_3)$  provided by Theorem \ref{thm:IFT}.
Assume that $t=\frac{l}{k}$ and let $\pi:\Sigma_k\to\C P^1$ be the
$k$-fold covering as in \eqref{eq:sigma_k} totally branched over $p_1,\dots,p_4$.
A twistor line provided by Theorem \ref{computingthecomponent} then gives rise to a solution $(\nabla,\Phi,h)$
of the self-duality equations on $\Sigma_k$. 

Analogously to  \cite[Corollary 4.3]{HHT1} and similar to the proof of Lemma \ref{rescaledform}, 
 the  energy  of the self-duality solution $(\nabla,\Phi,h)$ corresponding to 
 the real section provided by $\eta=\eta_{x_1,x_2,x_3}$ on the covering surface $\Sigma_k$ is given by
\begin{equation}\label{eq:energy}{\mathcal E}(h)=-4\pi\sum_j \Res_{q_j}\tr\left(\eta_{-1}G_{j,1}G_{j,0}^{-1}\right).\end{equation}

In this equation $G_j=G_{j,0}+G_{j,1}\lambda+\dots$ is a desingularizing gauge, i.e., a holomorphic family of gauge transformations $G_j$ that removes the singularity of the pull-back potential $\pi^*\eta$ at $\wh p_j = \pi^{-1}(p_j)$ on $\Sigma_k$ and extends holomorphically to $\lambda=0$ (compare with \eqref{def:twistbundle} for fixed $\lambda$). The reason for the different sign in \eqref{eq:energy} compared to \cite[Corollary 4.3]{HHT1}
is that we are considering (equivariant) harmonic maps into hyperbolic 3-space here, with associated family of the form
$\nabla+\lambda^{-1}\Phi+\lambda\Phi^*$, while in \cite{HHT1} we considered harmonic maps into the 3-sphere with associated family of the form
$\nabla+\lambda^{-1}\Phi-\lambda\Phi^*$, see also \cite[Theorem 2.4]{BeHRo}.

A desingularizing gauge in a neighborhood of $\wt p_1$ is given by
$$G_1=\matrix{1&0\\ \frac{1-x_1}{x_2+ix_3}&1}\matrix{w^{-l}&0\\0&w^l}$$
with $w$ the local coordinate with $w^{k}=z-p_1$.
Then
$$\frac{1}{l}\Res_{\wt p_1}\tr\left(\wt\eta_{-1}G_{1,1}G_{1,0}^{-1}\right)
=1-x_{1,0}\,+\,\frac{x_{1,-1}(x_{2,0}+i x_{3,0})}{x_{2,-1}+i x_{3,-1}}
=1-x_{1,0}\,-\,\frac{(x_{2,-1}-i x_{3,-1})(x_{2,0}+i x_{3,0})}{x_{1,-1}}.$$
For the residues at the other punctures, we  substitute $(x_1,x_2,x_3)\to(-x_1,-x_2,x_3)$ to obtain
$$\frac{1}{l}\Res_{\wt p_2}\tr\left(\wt\eta_{-1}G_{2,1}G_{2,0}^{-1}\right)
=1+x_{1,0}\,+\,\frac{(x_{2,-1}+i x_{3,-1})(x_{2,0}-i x_{3,0})}{x_{1,-1}}$$
and substitute $(x_1,x_2,x_3)\to (x_1,-x_2,-x_3)$ to obtain for $k=1,2$
$$\Res_{\wt p_{k+2}}\tr\left(\wt\eta_{-1}G_{k+2,1}G_{k+2,0}^{-1}\right)
=\Res_{\wt p_k}\tr\left(\wt\eta_{-1}G_{k,1}G_{k,0}^{-1}\right).$$
This gives the following formula for the energy of the equivariant harmonic map on $\Sigma_k$
$${\mathcal E}=-8\pi\,l\left(1+\frac{i}{x_{1,-1}}(-x_{2,-1}x_{3,0}+x_{2,0}x_{3,-1})\right).$$

Therefore, we define for $t>0$ the {\em renormalized energy} by
\begin{equation}\label{def:renen}
\underline{\mathcal E}_t:=-8\pi\left(1+\frac{i}{x_{1,-1}}(-x_{2,-1}x_{3,0}+x_{2,0}x_{3,-1})\right).
\end{equation}

Using the central value  of the parameters we find  that $\mathcal E$ extends to $t=0$ with
\begin{equation}
\label{mathematica-energy0}
\underline{\mathcal E}=-8\pi(1-\rho r^2)=8\pi (\sqrt{1+r^{4}}-1),
\end{equation}
where $r^2=|u|^2+|v|^2.$ The first order derivatives of the parameters then yield
\begin{eqnarray}
\mathcal E'&=&-8\,\Im(\Omega_{21}(1))\left(|u|^4+|v|^4-3\uc^2v^2-3u^2\vc^2-4|u|^2|v|^2\right)\label{mathematica-energy1}\\
&&+8\,\Im(\Omega_{31}(i))\left(|u|^4+|v|^4+3\uc^2v^2+3u^2\vc^2-4|u|^2|v|^2\right).
\nonumber\end{eqnarray}
For the most symmetric case of $p=e^{i\pi/4}$, using $\Omega_{21}(1)=-\pi i\log(2)$ from Appendix A of \cite{HHT2} and $\Omega_{31}(i)=-\Omega_{21}(1)$, this simplifies to
\begin{equation}
\label{mathematica-energy1sym}
\mathcal E'=16\pi\log(2)\left(|u|^4+|v|^4-4|u|^2|v|^2\right).
\end{equation}
\begin{remark}
The energy $\mathcal E$ can also be used to give a different proof of Lemma \ref{lem:wrongsignlim}, because for $\rho <0$, the (rescaled) energy in the limit
\[\mathcal E=-8\pi (1+\sqrt{1+r^{4}})<0\]
is negative (instead of \eqref{mathematica-energy0}), 
but twistor lines must have non-negative energy.
In the case of nilpotent Higgs fields, the (negative) energy can be interpreted as (the negative of) the Willmore energy
of a Willmore surface, see \cite[Section 4]{BeHRo}.
\end{remark}
\subsection{Another hyper-K\"ahler metric}
Recall that the implicit function theorem (Theorem \ref{thm:IFT}) also applies for the choice of negative $\rho$ in the initial condition \eqref{eq:rho0}.
The corresponding real holomorphic sections are not twistor lines by Lemma \ref{lem:wrongsignlim}.
On the other hand, we can still apply the general twistor space theory \cite{HKLR} to obtain a hyper-K\"ahler metric locally.  In fact, the rescaled metric at $t=0$ is still the Eguchi-Hanson metric $g_{EH}$, and 
therefore, for small $t$ and small Higgs fields, the metric (and the smooth structure) remains non-degenerate.
Using Lemma
\ref{lem:signmean} we can compare the $t$-families of hyper-K\"ahler metrics $g^{\pm}$ for both choices of sign in the initial condition \eqref{eq:rho0}
and obtain
\[g^\pm=g_{EH}+\sum_{k\geq1} (\pm1)^k g_k t^k.\]

\subsection{Convergence of Hitchin metrics on symmetric components}

The aim of this section is to show Theorem \ref{thm:compactcon}, which we restate here for the convenience of the reader.
\begin{theorem*}
Let $C>0$ and $l\in\N^{>0}$ be fixed.  Consider the compact subspaces 
\[\mathcal C^l_k:=\{[\nabla,\Phi]\in \mathcal M^l_{SD}(\Sigma_k)\mid \mathcal E(\nabla,\Phi)\leq C\}\]
 of $\mathcal M^l_{SD}$ with the induced hyper-K\"ahler metric $g_k$.
Then for $k \rightarrow \infty$, we obtain smooth convergence
\[(\mathcal C^l_k,g_k)\longrightarrow (\{v\in T^*\C P^1\mid \mathcal E_{EH}(v)\leq C\}, 32\pi\,l\, g_{EH})/(\Z_2\times\Z_2).\]\end{theorem*}

For the setup of the proof, let $\pi\colon\Sigma_k\to\C P^1$ be the totally branched covering  defined in \eqref{eq:sigma_k}.
Consider the moduli space of polystable Higgs bundles, and the moduli space of solutions to the self-duality equations $\mathcal M_{SD}(\Sigma_k),$ respectively, 
equipped with the hyper-K\"ahler metric on its smooth locus $\mathcal M^{irr}_{SD}(\Sigma_k)$. Consider the natural $\Z_k$-action (generated by $\sigma$) on 
$\mathcal M_{SD}(\Sigma_k)$. The fixed point set of $\sigma$ in $\mathcal M^{irr}_{SD}(\Sigma_k)$ has multiple connected components, which can be distinguished as follows. A fixed point $(\nabla,\Phi,h)$ of the $\Z_k$ action can be represented by a $\Z_k$-equivariant solution of the self-duality equations, i.e., there exists a gauge transformation $g_\sigma\colon\Sigma_k\to\mathrm{SU}(2)$ such that
\[\sigma^*(\nabla,\Phi,h)=(\nabla,\Phi,h).g_{\sigma}.\]
By assumption, $(\nabla,\Phi,h)$ lies in the smooth locus and in particular $\nabla$ is irreducible. Hence $g_\sigma$ is unique up to sign.
Moreover, at every fixed point $q\in\Sigma_k$ of $\sigma$, we have $g_{\sigma}^k(q)=\pm\text{Id}.$ 
Hence, the eigenvalues of $g_\sigma(q)$ are of the form $\exp(\pm 2\pi i\tfrac{ l_q}{2k})$ for some integer $l_q\in\{0,\dots,k\}.$ The integers $l_{p_1},\dots,l_{p_4}$ are then the invariants of a connected component $\mathcal C$ of the fixed point set.
In fact, for $l_{p_1},\dots,l_{p_4}\in\{1,\dots,k-1\}$, this component is diffeomorphic
to $\mathcal M_{SD}(\C P^1,p_1,\dots,p_4,\tfrac{l_{p_1}}{2k},\dots\tfrac{l_{p_4}}{2k}).$ For even 
$l_{p_1},\dots,l_{p_4}$, this directly follows from the discussion in Section \ref{ssec:rat}. For odd $l_{p_j}$, the gauge
\eqref{def:twistbundle} has to be modified to involve square roots, see for example \cite{NaSt}, or \cite{HeHeSch} for the specific $\Sigma_k$. Likewise, one can first go to the double covering $\Sigma_{2k}\to\Sigma_k$ and then proof that
the solutions on $\Sigma_{2k}$ are actually obtained as pull-backs of solutions on $\Sigma_k.$

Note that the component  $\mathcal C$ is a hyper-K\"ahler submanifold 
of $\mathcal M^{irr}_{SD}(\Sigma_k)$ (see for example \cite[Theorem 8]{HSch}).
In fact, this easily follows from the uniqueness of solutions to the self-duality equations.
These components  can be `completed'
by adding  orbifold points consisting of gauge classes of reducible connections.

For fixed $l\in\N^{>0}$ and every $k\in\N^{>2l}$ we denote 
by $\mathcal M^l_{SD}(\Sigma_k)$ the  component of equivariant solutions corresponding to the integers
\[l_{p_1}=l_{p_2}=l_{p_3}=l_{p_4}=l.\]
It then follows that this space is isomorphic to the hyper-K\"ahler space
$\mathcal M_{SD}(\C P^1,p_1,\dots,p_4,\tfrac{l}{2k})$ up to the scaling constant $32\pi k,$ see 
Proposition \ref{pro:sympupdown}. We consider the energy (of the solution of the self-duality equations) $\mathcal E$
on $\mathcal M^l_{SD}(\Sigma_k)$, which corresponds up to a constant with
the energy $\underline{\mathcal E}_{\tfrac{l}{2k}}$ on 
$\mathcal M_{SD}(\C P^1,p_1,\dots,p_4,\tfrac{l}{2k})$ by Section \ref{sec:energy}.

We consider the Eguchi-Hanson metric $g_{EH}$ on $T^*\C P^1.$ 
By identifying  $T^*\C P^1$ with the blow up of $(\C^2\setminus\{(0,0)\})/\Z_2$ at the origin, we have the energy
\[\mathcal E_{EH}(u,v)=8\pi(\sqrt{1+(u\bar u+v\bar v)^2}-1)\]
defined in terms of the coordinates $(u,v)\in\C^2.$
Consider also the $\Z_2\times\Z_2$-action on $T^*\C P^1$ generated by
\[\zeta\in\C P^1\mapsto -\zeta\in\C P^1\quad\text{and}\quad \zeta\in\C P^1\mapsto \zeta^{-1}\in\C P^1.\]
 This defines an isometric action on $(T^*\C P^1,g_{EH})$, and has 6 fixed points contained in the zero section $\C P^1\subset T^*\C P^1.$  
Their images in  $T^*\C P^1/(\mathbb Z_2\times\mathbb Z_2)$ represent the 3 strictly semi-stable parabolic bundles
(with vanishing Higgs field), see Proposition \ref{pro:paramodul}. 
 Note that the energy $\mathcal E_{EH}$ is invariant under this action, and furthermore $\mathcal E_{EH}$ is
 a K\"ahler potential for $32\pi g_{EH}$ on $T^*\C P^1$ with respect to the complex structure $J$ (where $I$ is the 
 natural complex structure induced from the cotangent bundle of $\C P^1$).
 
 \begin{proof}[Proof of Theorem \ref{thm:compactcon}]
We first show that for $k$ large enough, any 
$[\nabla,\Phi]\in \mathcal C_k^l$ has (semi-)stable underlying holomorphic structure $\bar\partial^\nabla$.
Recall that every equivariant stable Higgs pair $(\bar\partial^\nabla,\Phi)\in \mathcal M^l_{SD}$ gives rise to stable strongly parabolic Higgs
pair on $\C P^1$ with 4 singular points and parabolic weight $t=\tfrac{l}{2k}.$ If the underlying holomorphic structure $\bar\partial^\nabla$ is unstable, then also the  parabolic structure is unstable. By Lemma \ref{lem:unstablepara},
there is only a complex affine line of corresponding stable Higgs pairs with unstable underlying holomorphic structure.
The Higgs pair with least energy in this affine line has vanishing determinant, and it remains to estimate its energy.
This energy is given by $2\pi d$, where $d$ is the degree of the destabilizing holomorphic line bundle $L$, see for example the proof of \cite[Proposition 7.1]{Hi1}. It therefore remains to estimate the degree $d$, which depends on $k.$
Note that the equivariant nilpotent Higgs field is determined by a holomorphic section
\[\phi\in H^0(\Sigma_k,KL^{-2}).\]
Since $\phi$ is equivariant with respect to $\sigma$ as well, it must have a zero of order $(l-1)$ at all 4 singular points 
$p_1,\dots,p_4,$ compare with \cite[Theorem 3.3]{HeHeSch}. Hence, the degree of $KL^{-2}$ is $4(l-1)$ and since the genus of $\Sigma_k$ is $(k-1)$ we obtain
\[d=k-2l>\frac{C}{2\pi}\] for all $k$  large enough.

By Theorem \ref{thm:IFT} 
there exists for every $\widetilde C>0$ an $\epsilon>0$ such that for all $(u,v)\in\C^2$  with $0<|(u,v)|^2<\widetilde C$ and all
$\epsilon>t>0$, we have constructed a real holomorphic section in $\mathcal M^{par}_{DH}(\C P^1,p_1+\dots+p_4,t)$ with parabolic Higgs field
$t\Psi$
for $\Psi$ as in \eqref{eq:psiAs} with residues \eqref{eq:res1higgs} and \eqref{eq:psiAs2}. By Theorem \ref{blowuplimit} and Theorem \ref{computingthecomponent},
these real holomorphic sections are twistor lines and our construction extends  to polystable parabolic structures with  vanishing Higgs fields.
Denote this subset of strongly parabolic Higgs fields (with trivial underlying holomorphic bundle) by $\wt{\mathcal C}_t,$
and (by a slight abuse of notation) use $(u,v)$ as coordinates on it.

From the results of Section \ref{sec:energy} and particularly \eqref{mathematica-energy0}, we can chose
for all $C>0$ a constant $\widetilde C>0$ and $\epsilon>0$ such that
\[\underline{\mathcal E}_t(u,v)<\tfrac{C}{l}\quad \Longrightarrow \quad (u,v)\in \wt{\mathcal C}_t.\]
Let $t=\tfrac{l}{2k}<\epsilon.$
By the definition of $\underline{\mathcal E}_t$ and by the results of Section \ref{sec:energy} we  have 
\[\mathcal E(\nabla,\Phi,h)=l\underline{\mathcal E}_{\tfrac{l}{2k}}(u,v)\]
 for the twisted lift $(\nabla,\Phi,h)$ 
corresponding
to the strongly parabolic Higgs field given by $(u,v).$

Recall that
unless $(\nabla,\phi)$ is a fixed point of the
$\C^*$-action,
the energy along rays  $r\in\R^{\geq0}\mapsto \mathcal E(\nabla,r\phi)$ is strictly increasing, see for example the proof of \cite[Proposition 9.1]{Hi1}. For $k$ large enough, we therefore obtain that $\mathcal C_k$ is completely contained
in   $\wt{\mathcal C}_{\tfrac{l}{2k}}$ (after taking the twisted lift), and we are therefore in the domain covered by our implicit function theorem.
Thus, the theorem then directly follows   from Proposition \ref{limitEH} and the results from Section \ref{sec:res-metric}.
\end{proof}

\section{Higher order derivatives}\label{hod}
\subsection{The algorithm}
Just as in \cite[Section 5]{HHT2} we give an iterative algorithm for computing higher order derivatives of the parameters in terms of the multiple polylogarithm (MPL) function.  The difference to the minimal surface   case \cite{HHT2} is that we have no extrinsic closing conditions but complex parameters here.

The computation of the higher order derivatives of the parameters involves 
the iterated integral $\Omega_{i_1,\cdots,i_n}$ defined recursively by
\begin{equation}\label{eqn:def:omegaintegral}
\Omega_{i_1,\cdots,i_n}(z)=\int_0^z\Omega_{i_1,\cdots,i_{n-1}}\omega_{i_n},
\end{equation}
where $\omega_i$ is as in \eqref{omega} for $i=1,2,3.$ It is shown in \cite[A.3]{HHT2}
that (and how) these iterated integrals can be expressed in terms of multiple polylogarithm $ \Li_{n_1,\ldots,n_d} $. For  positive integers $n_1,\ldots,n_d \in \mathbb{Z}_{>0} $, and $z_i \in \mathbb{C}^d $ in the region given by $ |z_i \ldots z_d| < 1 $, the multiple polylogarithm  \( \Li_{n_1,\ldots,n_d} \) is defined by
	\begin{equation*}
		\Li_{n_1,\ldots,n_d}(z_1,\ldots,z_d)
	=\sum_{0<k_1<k_2<\cdots<k_d}\frac{z_1^{k_1}\cdots z_d^{k_d}}{k_1^{n_1}\cdots k_d^{n_d}} \,.
	\end{equation*}
	This function is extended to a multivalued function by analytic continuation.  Here the \emph{depth} $d $ counts the number of indices $n_1,\ldots,n_d $, and the \emph{weight} is given by the sum $ n_1 + \cdots + n_d $ of the indices. The functions $\Omega_{i_1,\cdots,i_n}(z)$ can be expressed in terms of multiple polylogarithms of depth $n$ and weight  $n.$

In the following we denote by $x_i^{(n)}$, $\mathcal P^{(n)}$, $\mathcal Q^{(n)}$ the  $n$-th order derivatives of $x_i$, $\mathcal P$, $\mathcal Q$ with respect to $t$ at $t=0$ and  we suppress the dependence of $(u,v)$ and $p.$}
\begin{proposition}\label{lastPro}
For $n\geq 1$ and $1\leq i\leq 3$, $x_i^{(n)}$ are polynomials (with respect to $\lambda$) of degree at most $n+1$ and the coefficients of $\mathcal P^{(n+1)}$, $\mathcal Q^{(n+1)}$ are Laurent polynomials of degree at most $n+1$, which can be expressed explicitly in terms of  multiple-polylogarithms  of depth $n+1$ and weight $n+1$.
\end{proposition}
\begin{proof}
The proof is by induction on $n$. Let $H_n$ be the statement of the proposition.
We have already proved $H_1$ in Section \ref{firstderivative}. Fix $n\geq 1$ and assume that $H_k$ is true for all $1\leq k\leq n-1$ and let the index `lower' denote all terms of a quantity that depends only on derivatives of lower order.
As in  \cite[Proposition 37]{HHT2} we have:
$$\mathcal P^{(n+1)}=\sum_{\ell=1}^{n+1}\frac{(n+1)!}{(n+1-\ell)!}
\sum_{i_1,\cdots,i_{\ell}}x_{i_1,\cdots,i_{\ell}}^{(n+1-\ell)}\mathfrak m_{i_1,\cdots,i_{\ell}}\Omega_{i_1,\cdots,i_{\ell}}(1)$$
with
$$x_{i_1,\cdots,i_{\ell}}=\prod_{j=1}^{\ell} x_{i_j}
\and\mathfrak m_{i_1,\cdots,i_{\ell}}=\prod_{j=1}^{\ell}\mathfrak m_{i_j}.$$

We rewrite this as
$$\mathcal P^{(n+1)}=(n+1)\sum_{i=1}^3 x_i^{(n)}\mathfrak m_i\Omega_i(1)
+\low{\mathcal P}^{(n+1)}$$
with
$$\low{\mathcal P}^{(n+1)}=\sum_{\ell=2}^{n+1}\frac{(n+1)!}{(n+1-\ell)!}
\sum_{i_1,\cdots,i_{\ell}}x_{i_1,\cdots,i_{\ell}}^{(n+1-\ell)}\mathfrak m_{i_1,\cdots,i_{\ell}}\Omega_{i_1,\cdots,i_{\ell}}(1).$$ 
Similar formula holds for $\mathcal Q^{(n+1)}$ with
$\Omega_{i_1,\cdots,i_{\ell}}(1)$ replaced by $\Omega_{i_1,\cdots,i_{\ell}}(i)$.
Using Leibniz rule we have
$$\mathfrak p^{(n+1)}=2\pi (n+1) x_3^{(n)}+\low{\mathfrak  p}^{(n+1)}$$
with
$$\low{\mathfrak p}^{(n+1)}=\sum_{k=1}^n {n+1\choose k}\left(
\mathcal P_{11}^{(k)}\mathcal P_{21}^{(n+1-k)}-\mathcal P_{12}^{(k)}\mathcal P_{22}^{(n+1-k)}\right)+\lowind{\mathcal P}{21}^{(n+1)}-\lowind{\mathcal P}{12}^{(n+1)}$$
and from $\mathfrak p^{(n+1)}=(\mathfrak p^{(n+1)})^*$ we obtain by taking the positive part:
$$(x_3^{(n)})^+=\frac{1}{2\pi(n+1)}\left((\low{\mathfrak p}^{(n+1)})^{-*}-(\low{\mathfrak p}^{(n+1)})^+\right).$$
In the same way
$$\mathfrak q^{(n+1)}=2\pi (n+1) x_2^{(n)}+\low{\mathfrak q}^{(n+1)}$$
with $$\low{\mathfrak q}^{(n+1)}=i\sum_{k=1}^n{n+1\choose k}\left(
\mathcal Q_{11}^{(k)}\mathcal Q_{21}^{(n+1-k)}+\mathcal Q_{12}^{(k)}\mathcal Q_{22}^{(n+1-k)}\right)+i\lowind{\mathcal Q}{21}^{(n+1)}+i\lowind{\mathcal Q}{12}^{(n+1)}$$
and we obtain $$(x_2^{(n)})^+=\frac{1}{2\pi(n+1)}\left((\low{\mathfrak q}^{(n+1)})^{-*}-(\low{\mathfrak q}^{(n+1)})^+\right).$$
By inspection and the induction hypothesis, $(x_2^{(n)})^+$ and $(x_3^{(n)})^+$
are polynomials of degree at most $n+1$.
\begin{remark} We could also determine $\Im(x_{3,0}^{(n)})$ and $\Im(x_{2,0}^{(n)})$ from
the zero part of $\mathfrak p^{(n)}=(\mathfrak p^{(n)})^*$ and $\mathfrak q^{(n)}=(\mathfrak q^{(n)})^*$, but it is simpler to determine the three complex parameters
$x_{i,0}^{(n)}$ by solving a complex linear system, see below. 
\end{remark}
Using Leibniz rule:
$$0=\mathcal K^{(n)}=2\sum_{i=1}^3\cv{x}_ix_i^{(n)}+\low{\mathcal K}^{(n)}$$
with
$$\low{\mathcal K}^{(n)}=\sum_{i=1}^3\sum_{k=1}^{n-1}
{n\choose k}x_i^{(k)}x_i^{(n-k)}.$$
We multiply by $\lambda$ and obtain (recall that $P_j=\lambda \cv{x}_j$):
\begin{equation}
\label{eq:higher-order1}
P_1x_1^{(n)}+P_2 x_{2,0}^{(n)}+P_3 x_{3,0}^{(n)}=
-\frac{\lambda}{2}\low{\mathcal K}^{(n)}-P_2(x_2^{(n)})^+-P_3(x_3^{(n)})^+.
\end{equation}
The right hand side of \eqref{eq:higher-order1} is already known and is a polynomial in $\lambda$ of degree at most $n+3$.
Hence $P_1 x_1^{(n)}$ is a polynomial of degree at most $n+3$.
When $u\sim v$, then both roots of $P_1$ are in $\D_{a}$,
so since $x_1^{(n)}$ cannot have poles in $\D_{a}$, it must be a polynomial of degree at most $n+1$. This remains true for all $(u,v)$ by analyticity.
Let $Q,R$ be the quotient and remainder of the division of the right side of \eqref{eq:higher-order1} with respect to $P_1$.
Then
$$(x_1^{(n)})^+=Q^+$$
and looking at the coefficients of $\lambda^0$, $\lambda^1$ and $\lambda^2$ in
$\eqref{eq:higher-order1}$, we obtain a system of three complex equations with
unknowns $x_{1,0}^{(n)}$, $x_{2,0}^{(n)}$ and $x_{3,0}^{(n)}$,
whose determinant is
$$\det\left(\cv{x}_{i,j}\right)_{-1\leq i\leq 1\atop 1\leq j\leq 3}=\tfrac{1}{2}i\rho r^6\neq 0$$.
\end{proof}

As a direct corollary we can express the
hyper-K\"ahler metric of the moduli space $\mathcal M (t)$  in terms of multiple polylogarithms.
\begin{proof}[Proof of Theorem \ref{mainT}]
By twistor theory, we can compute the hyper-K\"ahler metric explicitly in terms of the relative holomorphic symplectic form
$\varpi=32\pi\frac{dx_2\wedge dx_3}{x_1}.$
Since $x_1,x_2,x_3$ depend real analytic on $t$ the theorem follows from Proposition \ref{lastPro}.
\end{proof}

The algorithm has been implemented and using Mathematica we obtain, for example, for $p = e^{i\pi/4}$
\begin{equation}
\label{mathematica-energy2}
\mathcal E''=\frac{32\,\pi}{\rho r^6}(|u|^2-|v|^2)^2\big(3 r^8 +2 r^4+4 |u|^2|v|^2\big)\log(2)^2.
\end{equation}

 \subsection{Results for the nilpotent cone of the most symmetric case}
All following computations are conducted and simplified using Mathematica.
 The fully documented Mathematica notebook can be found with the arxiv version of the paper or on the webpage
\begin{verbatim}
https://www.idpoisson.fr/traizet/
\end{verbatim}

 We restrict to the case $p=e^{i\pi/4}$ and assume $v=p u$,
 i.e., by \eqref{eq:detPsi} we are in the nilpotent cone.
 The computations of higher order derivatives then simplify to:
   $$\mathcal E''=0$$
  \begin{equation}
  \label{mathematica-energy3}
  \mathcal E'''=-192\, \pi | u| ^4 \left(127 | u| ^4+20\right)\zeta (3)
  \end{equation}
where $\zeta$ is the Riemann $\zeta$-function.
If $\wt\varpi$ denotes the restriction of $\varpi$ to the nilpotent cone, we obtain
$$\wt\varpi'=8\, i |u|^2\log(2)\, du\wedge d\overline{u}$$
$$\wt\varpi''=0$$
\begin{equation}
\label{mathematica-omega3}
\wt\varpi'''=96 i\,\zeta(3)(127|u|^6+10 |u|^2)\,du\wedge d\uc.
\end{equation}

\subsection{New $\Omega$-identities}\label{nOid}
Write
$$\varpi=\sum_{k=0}^{\infty}\varpi_k\lambda^k.$$
We can compute $\varpi''_k$ for $k=0,1,2...$ from the derivatives of the parameters.
On the other hand, we know from section \ref{sec:goldman} and Lemma \ref{rescaledform} that $\varpi''_k=0$ for $k\geq 1$, and also
$$\varpi''_{0,u,v}=\varpi''_{0,\uc,\vc}=0.$$
Note that
$\varpi''_{0,\uc,\vc}=0$
is trivial because of \eqref{varphiconst} and because $d\cv{x}_{2,-1}$ and $d\cv{x}_{3,-1}$ are holomorphic 1-forms so vanish on
$(\frac{\partial}{\partial \uc},\frac{\partial}{\partial\vc})$.
On the other hand, $\varpi''_{0,u,v}=0$ is not trivial and gives
identities involving $\Omega$-integrals.
In this section, we use the notation
$$\Omega_{i_1,\cdots,i_n}=\Omega_{i_1,\cdots,i_n}(1)
\quad\text{and}\quad
\Theta_{i_1,\cdots,i_n}=\Omega_{i_1,\cdots,i_n}(i).$$
From
$$\varpi''_0=\frac{1}{\cv{x}_{1,-1}}\left(
\frac{-x_{1,0}''}{\cv{x}_{1,-1}}d\cv{x}_{2,-1}\wedge d\cv{x}_{3,-1}
+d x_{2,0}''\wedge d\cv{x}_{3,-1}
+d\cv{x}_{2,-1}\wedge dx_{3,0}''
\right)$$
we obtain
\begin{align}
\label{mathematica-omega''uv}
\varpi''_{0,u,v}&=\frac{ \rho}{\pi^2}\big[
6(u\uc^2\vc-\uc v\vc^2)(I_1+\overline{I_1})
+8(\uc^3 v - u\vc^3)(I_2+\overline{I_2})\\
&+\left(\frac{u\uc^6}{\vc^3}-\frac{v\vc^6}{\uc^3}+\frac{3\uc^5 v}{\vc^2}-\frac{3u\vc^5}{\uc^2}\right)(I_3+\overline{I_3})\big]
\nonumber\end{align}
with
$$\begin{cases}
I_1=6\pi(\Omega_{333}-\Theta_{222})+i((\Omega_{21})^2+(\Theta_{31})^2)
+2\pi(\Omega_{223}-\Theta_{332})-8\pi(\Omega_{311}-\Theta_{211})
+10 i \Omega_{21}\Theta_{31}\\
I_2=i((\Omega_{21})^2-(\Theta_{31})^2)+2\pi(\Omega_{223}+\Theta_{332}+\Omega_{311}+\Theta_{211})-4\pi(\Omega_{333}+\Theta_{222})\\
$$I_3=-i((\Omega_{21})^2+(\Theta_{31})^2)+2\pi(\Omega_{333}-\Theta_{222}-\Omega_{223}+\Theta_{332})-2i\Omega_{21}\Theta_{31}.\end{cases}$$
Since $\varpi''_{0,u,v}=0$, $I_1$, $I_2$ and $I_3$ are pure imaginary. Since they are holomorphic functions of $p\in\C^+$, they must be constant. We find the constants by evaluating at $p=e^{i\pi/4},$ where all integrals are known 
from \cite[Appendix]{HHT2}, and find
$$\begin{cases}
I_1=-i\frac{\pi^4}{3}\\
I_2=0\\
I_3=-i\pi^4.\end{cases}$$
Using the elementary values
$$\Omega_{333}=\frac{1}{6}(\Omega_3)^3=\frac{-i\pi^3}{6}
\quad\text{and}\quad
\Theta_{222}=\frac{1}{6}(\Theta_2)^2=\frac{i\pi^3}{6}$$
we obtain the following identities for all $p\in\C^+$:
\begin{equation}
\label{eq:new-identities}
\begin{cases}
\Omega_{223}+\Omega_{311}=-\frac{i}{2\pi}(\Omega_{21})^2\\
\Theta_{211}+\Theta_{332}=\frac{i}{2\pi}(\Theta_{31})^2\\
\Omega_{223}-\Theta_{332}=-\frac{i}{2\pi}(\Omega_{21}+\Theta_{31})^2+\frac{i\pi^3}{6}.
\end{cases}\end{equation}

From higher order derivatives in $t$ and higher order terms $\varpi_k$ we expect a hierarchy of identities for $\Omega$-values.
Another source of identities is the character variety of the 4-punctured sphere. Analogously to Remark \eqref{remark:r''}, taking the 4th-order derivative of
\eqref{mathematica-fricke} gives rise to the following three identities, which express linear combinations of $\Omega$-integrals of depth 3 as a function of $\Omega$-integrals of depth at most 2, and are non-trivial in the sense that they do not follow from shuffle product relations alone:
   
\begin{align}
\label{mathematica-ID1}
&\Omega_{2,1,2}
-\Theta_{1,2,1}
-\Theta_{2,1,2}=
\\&
-\frac{1}{2} \Omega_{2} \Omega_{1,2}
+\frac{1}{2}\Omega_{2} \Omega_{2,1}
-\frac{i \Omega_{1,2}^2}{4 \pi }
+\frac{3 i\Omega_{2,1}^2}{4 \pi }
-\frac{i \Omega_{1,2}\Omega_{2,1}}{2 \pi}
+\frac{1}{2} \Theta_{1}\Theta_{1,2}
+\frac{1}{2} i \pi \Theta_{1,2}
-\frac{1}{2} \Theta_{1}\Theta_{2,1}
\nonumber\\&
-\frac{1}{2} i \pi \Theta_{2,1}
+\frac{i \Omega_{1}^2 \Omega_{2}^2}{4 \pi}
+\frac{1}{2} \Omega_{1} \Omega_{2}^2
-\pi ^2\Omega_{1}
+\frac{1}{2} i \pi \Theta_{1}^2
-\frac{1}{4} \Theta_{1}\Theta_{2}^2
+\frac{1}{4} \pi ^2\Theta_{1}
+\frac{2 i \pi ^3}{3}
 \nonumber\end{align}

\begin{align}
 \label{mathematica-ID2}
 & \Omega_{1,3,1}
 +\Omega_{3,1,3}
  -\Theta_{3,1,3}=
  \\&
 -\frac{1}{2} \Omega_{1} \Omega_{1,3}
 +\frac{1}{2}\Omega_{1} \Omega_{3,1}
 +\frac{1}{2} i \pi \Omega_{1,3}
 -\frac{1}{2} i \pi \Omega_{3,1}
 -\frac{i \Theta_{1,3}^2}{4 \pi }
 +\frac{3 i\Theta_{3,1}^2}{4 \pi }
 +\frac{1}{2} \Theta_{3}\Theta_{1,3}
 -\frac{1}{2} \Theta_{3}\Theta_{3,1}
 -\frac{i \Theta_{1,3}\Theta_{3,1}}{2 \pi}
 \nonumber \\&
 +\Omega_{1}^2\Theta_{1}
 -\Omega_{1}\Theta_{1}^2
 -\frac{1}{3}\Omega_{1}^3
 +\frac{1}{2} i \pi \Omega_{1}^2
 +\frac{1}{4} \Omega_{3}^2\Omega_{1}
 -\frac{1}{4} \pi ^2\Omega_{1}
 +\frac{1}{3} \Theta_{1}^3
 +\frac{i\Theta_{1}^2 \Theta_{3}^2}{4 \pi }
 -\frac{1}{2}\Theta_{1} \Theta_{3}^2
 +\pi ^2 \Theta_{1}
 +\frac{i \pi ^3}{3}
  \nonumber \end{align}

\begin{align}
  \label{mathematica-ID3}
&\Omega_{2 3 2}-\Theta_{3 2 3}=\\
&\frac{1}{2} \Omega_{2}\Omega_{2 3}
-\frac{1}{2} \Omega_{2}\Omega_{3 2}
+\frac{1}{2} \Theta_{3}\Omega_{3 2}
-\frac{1}{2} \Theta_{3}\Omega_{2 1}
+\Theta_{3}\Omega_{2 3}
+\frac{1}{2} \Theta_{3}\Theta_{2 3}
-\frac{1}{2}\Theta_{3} \Theta_{3 1}
\nonumber\\&
-\frac{i \Omega_{2 1}\Theta_{2 3}}{2 \pi }
-\frac{i \Omega_{3 2}\Theta_{2 3}}{2 \pi }
+\frac{3 i \Omega_{2 1}\Theta_{3 1}}{2 \pi }
-\frac{i \Omega_{3 2}\Theta_{3 1}}{2 \pi }
-\frac{i \Omega_{2 1}\Omega_{3 2}}{2 \pi }
-\frac{i\Theta_{2 3} \Theta_{3 1}}{2 \pi }
\nonumber\\&
+\frac{3 i \Omega_{2 1}^2}{4 \pi}
-\frac{i \Omega_{3 2}^2}{4 \pi }
-\frac{i\Theta_{2 3}^2}{4 \pi }
+\frac{3 i\Theta_{3 1}^2}{4 \pi }
+\frac{i \pi \Omega_{2}^2}{4}
+\frac{i\pi  \Theta_{3}^2}{4}
-i \pi \Omega_{2 1}
-i \pi \Omega_{3 2}
\nonumber\\&
-i \pi  \Theta_{2 3}
-i \pi \Theta_{3 1}
-\pi ^2 \Omega_{2}
+\pi ^2 \Theta_{3}
-\frac{i\pi ^3}{3}\nonumber.
   \end{align}

\end{document}